\documentclass[a4paper]{article}

\usepackage{lineno}

\usepackage{amsmath, todonotes} 
\usepackage{amstext}
\usepackage{amsthm}
\usepackage{amscd} 
\usepackage{amsopn} 
\usepackage{verbatim} 
\usepackage{amssymb}
\usepackage{amsfonts}
\usepackage{mathtools}
\usepackage[mathscr]{euscript}
\usepackage{fullpage}
\usepackage[bbgreekl]{mathbbol}
\usepackage{tikz}
\usepackage{tikz-cd}
\usetikzlibrary{babel}
\usetikzlibrary{matrix}
\usetikzlibrary{shapes}
\usetikzlibrary{arrows}
\usetikzlibrary{calc,3d}
\usetikzlibrary{decorations,decorations.pathmorphing}
\usetikzlibrary{through}
\tikzset{ext/.style={circle, draw,inner sep=1pt},int/.style={circle,draw,fill,inner sep=1pt},nil/.style={inner sep=1pt}}
\tikzset{exte/.style={circle, draw,inner sep=3pt},inte/.style={circle,draw,fill,inner sep=3pt}}
\tikzset{diagram/.style={matrix of math nodes, row sep=3em, column sep=2.5em, text height=1.5ex, text depth=0.25ex}}
\tikzset{diagram2/.style={matrix of math nodes, row sep=0.5em, column sep=0.5em, text height=1.5ex, text depth=0.25ex}}
\tikzset{every picture/.append style={baseline=-.65ex}}
\tikzset{de/.style={-latex}} 
\tikzset{ed/.style={latex-}} 

\usepackage{ulem}
\usepackage{color}

\usepackage{hyperref}
\usepackage{todonotes}

\newcommand{\ldot}{{\:\raisebox{1.5pt}{\selectfont\text{\circle*{1.5}}}}}
\newcommand{\udot}{{\:\raisebox{4pt}{\selectfont\text{\circle*{1.5}}}}}

\newcommand{\ttt}{\text{-}}

\let\le\leqslant

\let\leq\leqslant
\let\geq\geqslant

\newcommand{\ZZ}{\mathbb Z}
\newcommand{\PP}{\mathbb P}
\newcommand{\QQ}{\mathbb Q}
\newcommand{\RR}{\mathbb R}
\newcommand{\CC}{\mathbb C}
\newcommand{\FF}{\mathbb F}
\newcommand{\kk}{\Bbbk}

\newcommand{\calB}{\mathcal{B}}
\newcommand{\calC}{\mathcal{C}}
\newcommand{\calO}{\mathcal{O}}
\newcommand{\calP}{\mathcal{P}}
\newcommand{\calQ}{{\mathcal{Q}}}
\newcommand{\calF}{{\mathcal{F}}}

\newcommand{\gr}{\mathsf{gr}}

\newcommand{\dual}{\checkmark}
\newcommand{\eqvr}{{\sim}_{\text{\tiny{$\ZZ_2$}}}}
\newcommand{\Sgn}{\mathsf{Sgn}}
\newcommand{\s}{\mathsf{s}}
\newcommand{\graph}{\mathrm{G}}

\newcommand{\odd}{\mathsf{odd}}

\newcommand{\Ass}{\mathsf{As}}
\newcommand{\AsM}{[\Ass]^{\ZZ_2}}
\newcommand{\AsMd}{[\Ass]_{\ZZ_2}^{\dual}}

\newcommand\Com{\mathsf{Comm}}
\newcommand{\LL}{\mathsf{L}}
\newcommand\Lie{\mathsf{Lie}}

\newcommand{\Pois}{{\mathsf{Pois}}}

\newcommand{\TwoPois}{{\Pois}^{\odd}}

\newcommand{\ho}{\mathsf{ho}}
\newcommand{\Hom}{\mathsf{Hom}}

\newcommand{\Cacti}{{\mathcal{C}}\mathfrak{acti}}
\newcommand{\PCacti}{{\mathcal{PC}}\mathfrak{acti}}
\newcommand{\B}{\mathfrak{B}}
\newcommand{\PB}{\mathcal{P}\B}

\newcommand{\Cell}{{\mathsf{Cell}}}
\newcommand{\Conf}{{\mathsf{Conf}}}

\newcommand{\Mon}[1]{\overline{{\mathcal M}_{0,{#1}}}}
\newcommand{\MonR}[1]{\overline{{\mathcal M}_{0,{#1}}}({\mathbb R})}
\newcommand{\Mos}{{\bar{\small{\mathcal M}}}_{0,\bullet}^{\mathbb R}}

\newcommand{\tkd}{\mathfrak{t}}
\newcommand{\todd}{\tkd^{\odd}}

\newcommand{\qL}{\mathsf{L}}
\newcommand{\Lcs}{\mathscr{L}}
\newcommand{\LCacti}{\mathfrak{L}}

\newcommand{\Chains}{{\mathsf{Chains}}}
\newcommand{\Id}{\mathsf{Id}}

\newcommand{\orient}{\downarrow}
\newcommand{\ddorient}{\Downarrow}
\newcommand{\twoorient}{
\begin{tikzpicture}[every edge/.style={draw, ->},scale=0.2]
\coordinate (v0) at (-0.1,0.5);
\coordinate (v1) at (-0.2,-0.5);
\coordinate (v2) at (0.2,0.5);
\coordinate (v3) at (0.2,-0.5);
\draw (v0) edge (v1);
\draw (v3) edge (v2);				
\end{tikzpicture}
	}

\newcommand{\nosource}{
		\begin{tikzpicture}[every edge/.style={draw, ->},scale=0.3]
		\coordinate (v) at (0,-.5);
		\coordinate (w1) at (-.6,.5);
		\coordinate (w2) at (-.2,.5);
		\coordinate (w3) at (.2,.5);
		\coordinate (w4) at (.6,.5);
		\draw (w1) edge (v);
		\draw (w2) edge (v);
		\draw (w3) edge (v);
		\draw (w4) edge (v); 
		\coordinate (u1) at (-1,.4);
		\coordinate (u2) at (-1,-.4);
		\coordinate (u3) at (1,.4);
		\coordinate (u4) at (1,-.4);
		\draw[dashed] (u1) -- (u4); 
		\draw[dashed] (u2) -- (u3); 
		\end{tikzpicture}
		}

\newcommand{\extv}[1]{\begin{tikzpicture}
	\node[ext] (v) (0,0) {\small{$#1$}};
	\end{tikzpicture}	}	

\newcommand{\oneedge}[2]{	
	\begin{tikzpicture}[scale=0.5]
	\node[int] (v) at (0,0) {};
	\node (v1) at (0,0.3) {\small{$#1$}};
	\node[ext] (w) at (1.5,0) {\small{$#2$}};
	\draw[-triangle 60] (v) edge (w);
	\end{tikzpicture}
	}

\newcommand{\ttriple}[4]{
		\begin{tikzpicture}[scale=0.5]
		\node (w) at (-0.4,1) {\small{$#1$}};
		\node[int] (v) at (0,1) {};
		\node[ext] (v1) at (-1,0) {\small{$#2$}};
		\node[ext] (v0) at (0,-0.2) {\small{$#3$}};
		\node (v2) at (1,0) {\small{$#4$}};
		\draw[-triangle 60] (v) edge (v1)  edge (v0) edge (v2);
		\end{tikzpicture}
	}

\newcommand\vertices{\mathsf{Vert}}
\newcommand{\edges}{\mathsf{Edges}}
	
\newcommand{\GC}{{\mathsf{GC}}}

\newcommand{\dGC}{{\mathsf{dGC}}}

\newcommand{\TCG}{{\mathsf{TCG}}}
\newcommand{\Graphs}{{\mathsf{Graphs}}}

\newcommand{\Tre}{\mathbf{Tr}}

\newcommand{\ICG}{\mathsf{ICG}}
\newcommand{\dICG}{\mathsf{dICG}}
\newcommand{\TICG}{\mathsf{TICG}}
\newcommand{\ICGod}[1]{{\ICG}^{\orient,\odd}_{#1}}
\newcommand{\ICGS}[1]{{\ICG}^{\orient,\odd}_{d,#1}(n)}

\newcommand{\LICG}{\mathsf{L}\ICG}
\newcommand{\LICGod}[1]{\mathsf{L}\ICG^{\orient,\odd}_{#1}}
\newcommand{\LLICG}[1]{\mathsf{L}\ICG^{\twoorient,\odd}_{#1}}

\newcommand{\dGraphs}{{\mathsf{dGraphs}}}
\newcommand{\dfGraphs}{{\mathsf{dGraphs}}}

\newcommand{\Def}{\mathrm{Def}}

\newcommand{\FM}{{\mathsf{FM}}}
\newcommand{\SC}{{\mathsf{SC}}}

\numberwithin{equation}{subsection}

\newtheorem{theorem}[equation]{Theorem}
\newtheorem*{theorem*}{Theorem}
\newtheorem{proposition}[equation]{Proposition}
\newtheorem*{proposition*}{Proposition}

\newtheorem*{statement*}{Statement}
\newtheorem{lemma}[equation]{Lemma}
\newtheorem*{lemma*}{Lemma}
\newtheorem{corollary}[equation]{Corollary}
\newtheorem*{corollary*}{Corollary}

\newtheorem{definition}[equation]{Definition}
\newtheorem*{definition*}{Definition}
\newtheorem{notation}[equation]{Notation}
\newtheorem{remark}[equation]{Remark}
\newtheorem*{remark*}{Remark}
\newtheorem{example}[equation]{Example}
\newtheorem*{example*}{Example}

\providecommand{\eprint}[2][]{\href{http://arxiv.org/abs/#2}{arXiv:#2}}

\theoremstyle{definition}
\newtheorem{convention}[equation]{Convention}

\title{Real moduli space of stable rational curves revised}
\author{Anton Khoroshkin\thanks{
	Department of Mathematics, University of Haifa, Mount Carmel, 3498838, Haifa, Israel } 
\and 
Thomas Willwacher\thanks{
		Department of Mathematics 
		ETH Zurich   
		R\"amistrasse 101 
		8092 Zurich, Switzerland	}
}

\date{}

\begin{document}
\maketitle

\begin{abstract}

The real locus of the moduli space of stable genus-zero curves with marked points, $\overline{{\mathcal M}_{0,{n+1}}}({\mathbb R})$, is known to be a smooth manifold and is the Eilenberg-MacLane spaces for the so-called pure Cactus groups. 
We describe the operad formed by these spaces in terms of a homotopy quotient of an operad of associative algebras. Using this model, we identify various Hopf models for the algebraic operad of chains and homologies of $\overline{{\mathcal M}_{0,{n+1}}}({\mathbb R})$. In particular, we show that the operad $\overline{{\mathcal M}_{0,{n+1}}}({\mathbb R})$ is not formal. 
As an application of these operadic constructions, we prove that for each $n$, the cohomology ring $H^{\udot}(\overline{{\mathcal M}_{0,{n+1}}}({\mathbb R}), {\mathbb{Q}})$ is a Koszul algebra, and that the manifold $\overline{{\mathcal M}_{0,{n+1}}}({\mathbb R})$ is not formal for $n\geq 6$ but is a rational $K(\pi,1)$-space. Additionally, we describe the Lie algebras associated with the lower central series filtration of the pure Cactus groups.
\end{abstract}

\tableofcontents
\setcounter{section}{-1}
\section{Introduction}
\label{sec::intro}

The Deligne-Mumford compactification $\Mon{n}$ of the moduli space of genus zero algebraic curves with $n$ marked points is a smooth algebraic variety defined over $\mathbb Z$ (\cite{DM}). 
The natural stratification of the space $\Mon{n}$ by the number of double points on a curve defines a structure of a cyclic operad on $\Mon{n}$.    
The cohomology ring of $\Mon{n}(\CC)$ was found by Keel (\cite{Keel}) and the description of an algebraic operad formed by the union $\cup_{n\geq 3} H_{\ldot}(\Mon{n}(\CC);\QQ)$ was presented by Kontsevich and Manin (\cite{Konts_Manin}) and by Getzler(\cite{Getzler2}). Each complex projective smooth variety $\Mon{n}(\CC)$ is formal. Moreover, the operad $\cup\Mon{n}(\CC)$ is formal (\cite{GSNPR}). 
Surprisingly, the detailed description of the (rational) homotopy groups of $\Mon{n}(\CC)$ as well as the description of the  corresponding operad $\cup_{n\geq 3}\pi_{\ldot}(\Mon{n}(\CC))\otimes\QQ$ in the category of $\LL_{\infty}$-algebras was discovered just recently by Dotsenko (\cite{Dots_Mon}).

On the other hand, the homotopy type of the real points of this variety (denoted by $\MonR{n}$) was found before its cohomology. Davis-Januszkiewicz-Scott (\cite{DJS}) proved that the real manifold   $\MonR{n+1}$ is aspherical and found a presentation of its fundamental group which we call \emph{Pure Cactus group} and denote it by $\PCacti_n$ (see Section~\ref{sec::Cacti} for details). Kapranov (\cite{Kapranov::Moduli}) and Devadoss(\cite{Dev}) realized that the collection $\{\MonR{n+1},n\geq 2\}$ assembles into an operad and 
Etingof-Henriques-Kamnitzer-Rains (\cite{Etingof_Rains}) described the structure of the cohomology ring $H^{\udot}(\MonR{n+1})$ for individual $n$ and found a presentation of the algebraic operad $\cup_{n\geq 2}H^{\udot}(\MonR{n+1})$ by generators and relations.
The (pure) cactus groups $\PCacti_n$ has a lot of common properties with (pure) braid groups and are of particular interest for representation theory (e.g.~\cite{HK_cacti}). The (pure) cactus groups play the same role in coboundary categories (introduced by Drinfeld~\cite{Drinfeld_Quasi_Hopf}) as the (pure) braid groups in braided tensor categories (e.g. \cite{Kassel}). Various conjectures were stated about the lower central series and Malcev completions of the pure cactus groups in~\cite{Etingof_Rains}. These conjectures streamline the comparisons of (pure) cactus groups and (pure) braid groups (Section~3 of~\cite{Etingof_Rains}).

The main purpose of the present paper is to clarify the structure of the topological operad $\{ \MonR{n+1}\}$, present nice algebraic models of this operad and prove almost all conjectures stated in~\cite{Etingof_Rains}. 
In particular, we prove the following.
\begin{itemize}
	\setlength{\itemsep}{-0.4em}
\item (Theorem~\ref{thm::Mosaic::chains})
The Koszul dual operad to the operad of $\Chains_{\QQ}(\MonR{n+1})$ coincides with $\ZZ_2$-invariants of the associative operad  where the generator of $\ZZ_2$ interchanges an algebra and its opposite.
\item 
We find a presentation of the latter operad of $\ZZ_2$-invariants of the associative operad (Theorem~\ref{thm::AsM::def});
\item This gives another proof of the computation of the Poincar\'e polynomial and the description of the cohomology cooperad $\cup_{n\geq 2} H^{\udot}(\MonR{n+1},\QQ)$ (Corollary~\ref{thm::Mos::Pois}) compared to the one suggested in~\cite{Etingof_Rains}.
\item
 The Koszul dual statement leads to the homotopy equivalence of the operad $\MonR{n+1}$ and the following homotopy quotient (Corollary~\ref{cor::Mos=E1/2}):
 \[
 \cup_{n\geq 2} \MonR{n+1} \simeq \frac{E_1\ltimes \ZZ_2}{\ZZ_2} \simeq \frac{\Ass \ltimes \ZZ_2}{\ZZ_2}
 \]
\item
We adopt certain combinatorial Hopf models known for little balls operads $E_d$ and for its homology operads $\Pois_d$ to the case of the operads $\frac{E_d\ltimes \ZZ_2}{\ZZ_2}$. The corresponding Hopf model is called $\Graphs_d^{\orient,\odd}$ since it is spanned by oriented graphs with odd number of outputs and the corresponding homology operads is denoted by $\TwoPois_d$ (\S\ref{sec::Hopf::Mosaic} for details).
\item
The latter combinatorial Hopf models $\Graphs_d^{\orient,\odd}$ and the $\LL_\infty$-algebras of its generators $\ICGod{d}$ are used to prove the main conjectures stated in~\cite{Etingof_Rains} concerning (pure) cactus groups (our results are stated in Section~\ref{sec::results}).
\item  
In particular, we proved that the quadratic algebras $H^{\udot}(\MonR{n+1};\QQ)$ are Koszul (Corollary~\ref{cor::koszul::Twopois} and, consequently, $\MonR{n+1}$ are rational $K(\pi,1)$ spaces (Corollary~\ref{cor::Mon::Kp1}).
\end{itemize}

Moreover, we state that the technique used for the computations of Kontsevich graph complexes (\cite{Willwacher_grt,Willwacher_oriented}) can be reasonably adopted to give a description of the deformation theory of the operads $\TwoPois_d$, $\MonR{n+1}$ and the maps of operads $\frac{E_d\ltimes \ZZ_2}{\ZZ_2} \rightarrow E_{d+1}$ that is presented by Drinfeld unitarization trick for $d=1$.
In particular, the deformation complex of the operad   
$\TwoPois_d$ is supposed to be quasi-isomorphic to the subspace of the Kontsevich graph complex $\GC_{d}$ spanned by graphs with odd Euler characteristic and that the operad $\MonR{n+1}$ has no nontrivial deformations.

\subsection{Structure of the paper}

The first chapter~\S\ref{sec::MON::intro} is an extended introduction and the advertisement of the material for those who do not want to deal seriously with operads.
First, we recollect the known material related to $\MonR{n}$ (\S\ref{sec::MON::CC},\S\ref{sec::MON::RR}) and the pure cactus groups (\S\ref{sec::Cacti}). Second, we recall the notion of the coboundary category and its relation with Braided Tensor categories suggested by Drinfeld in~\S\ref{sec::Drinfeld::unit}. We finish with the announcement of the corollaries of this paper for the pure cactus groups in~\S\ref{sec::results}.

\S\ref{sec::Mosaic::All} contains the operadic description of $\MonR{n+1}$ that clarifies the cell decomposition of $\MonR{n+1}$ discovered by Devadoss and Kapranov.

We collect all computations of the operads of $\ZZ_2$-invariants on the operads of commutative, associative and Poisson algebras in Chapter~\S\ref{sec::ZZ_2}.

Chapter~\S\ref{sec::Hopf::E_d} contains an outline of different known combinatorial dgca models of the little balls operad $E_d$ that are generalized to the case of the mosaic operad in Chapter~\S\ref{sec::Hopf::Mosaic}.  

Chapter~\S\ref{sec::Rational::MONR::all} contains the proof of Theorems~\ref{thm::HICG} and~\ref{thm::todd::t::emb} which are the most technical and complicated part of this paper. 

The results on deformation of the Mosaic operad and its relations with the little discs operad are outlined in Chapter~\S\ref{sec::Deformations}.

\section*{Acknowledgement}
We would like to thank
Pavel Etingof, Nikita Markarian, Sergei Merkulov, Dmitri Piontkovski and Leonid Rybnikov for stimulating discussions. We acknowledge Vladimir Dotsenko for pointing out an impudent usage of the Distributive Laws in the first draft of this paper. We also want to thank the anonymous referee for their valuable comments and suggestions, which have helped improve the quality of this paper.

\section{Moduli space of stable rational curves and Cactus groups}
\label{sec::MON::intro}

\subsection{Complex stable curves and $\Mon{n}$}
\label{sec::MON::CC}
Recall \cite{DM} that a point of $\Mon{n}$ (the Deligne-Mumford compactification of the space of genus zero curves with $n$ marked points) is a stable curve. A stable curve of genus $0$ with $n$ labeled points is a
finite union $C$ of projective lines $C_1,...,C_p$, together with labeled distinct points
$z_1,...,z_n\in C$ such that the following conditions are satisfied
\begin{enumerate}
	\setlength\itemsep{-0.5em}	
	\item Each marked point $z_i$ belongs to a unique $C_j$.
	
	\item The intersection of projective lines $C_i\cap C_j$ is either empty or consists of one point, and in the latter
	case the intersection is transversal.
	
	\item The graph of components (whose vertices are the lines $C_i$
	and whose edges correspond to pairs of intersecting lines) is a
	tree.
	
	\item The total number of special points (i.e. marked points or
	intersection points) that belong to a given component $C_i$ is at
	least $3$.
\end{enumerate}

An equivalence between two stable curves
$C=(C_1,...,C_p,z_1,...,z_n)$ and
$C'=(C_1',...,C_p',z_1',...,z_n')$ is an isomorphism of algebraic curves
$f: C\to C'$ which maps $z_i$ to $z_i'$ for each $i$. 

The gluing of two stable curves through marked points defines a collection of maps 
\begin{equation}
\label{eq::MON:operad}
\Mon{n+1}\times \Mon{m+1} \to \Mon{m+n}
\end{equation}
that assembles a structure of a cyclic operad on $\{\Mon{n+1}\}$. We will be interested in this paper only in the structure of symmetric operad. Thus we mark one of the points of a stable curve by $0$ and call it an output.

\subsection{The moduli space $\MonR{n+1}$ (=the mosaic operad)}
\label{sec::MON::RR}
The real locus $\MonR{n+1}$ of the Deligne-Mumford compactification 
consists of equivalence classes of stable
curves of genus $0$ with $n+1$ labeled points defined over ${\RR}$.
The projective line over real numbers is a circle, thus pictorially a stable curve
is a ``cactus-like'' structure -- a tree of circles with labeled
points on them:
\begin{equation}
\label{pict::cactus::example}
\begin{tikzpicture}[scale=0.5]
\draw (0,0) circle (1);
\draw (0,-1.1) -- (0,-0.9);
\node [below] at (0,-1) {\bf{$0$}};
\draw (0,1.1) -- (0,0.9);
\node [above] at (0,1) {$2$};
\draw (-2,0) circle (1);
\draw (-2,-1.1) -- (-2,-0.9);
\node [below] at (-2,-1) {$1$};
\draw (-2,1.1) -- (-2,0.9);
\node [above] at (-2,1) {$6$};
\draw (-2-1.1,0) -- (-2-0.9,0);
\node [left] at (-3,0) {$7$};
\draw (2,0) circle (1);
\draw (3.1,0) -- (2.9,0);
\node [right] at (3,0) {$3$};
\draw (2,2) circle (1);
\draw (3.1,2) -- (2.9,2);
\node [right] at (3,2) {$5$};
\draw (2,3.1) -- (2,2.9);
\node [above] at (2,3) {$4$};
\node at (1,0) {$\bullet$};
\node at (-1,0) {$\bullet$};
\node at (2,1) {$\bullet$};
\end{tikzpicture}
\in \MonR{8}
\end{equation}

\begin{example}
	\begin{enumerate}
		\setlength\itemsep{-0.4em}	
		\item $\MonR{3}$ is a point.
		
		\item $\MonR{4}$ is a circle. The cross-ratio map defines an isomorphism 
		$\MonR{4}\to \RR\PP^1$.
		
		\item $\MonR{5}$ is a compact connected non-orientable surface with Poincar\'e polynomial equals to $1+ 4 t$ (see \cite{Dev}).
	\end{enumerate}
\end{example}

Let us summarize some of the known results about $\MonR{n+1}$.

\begin{enumerate}
		\setlength\itemsep{-0.5em}	
\item (\cite{Dev}) The $\MonR{n+1}$ are connected,
	compact, smooth manifolds
	of dimension $n-2$ and the gluing maps~\eqref{eq::MON:operad} 	
 define an operad structure on their union. This operad is called the \emph{mosaic operad} after S.\,Devadoss(\cite{Dev}). We will come back to the detailed description of this operad in Section~\ref{sec::Mosaic::All}.
		
\item (\cite{DJS}) $\MonR{n+1}$ is a $K(\pi,1)$-space.
The fundamental group $\pi_1(\MonR{n+1})$ is called pure cactus group.
We recall the definition of  (pure) cactus groups in the next subsection~\ref{sec::Cacti}.

\item (\cite{Etingof_Rains})
The rational cohomology $H^{\udot}(\MonR{n};\QQ)$ is a quadratic algebra. We recall these algebras in Section~\ref{sec::2-gerst}.

\item (\cite{Etingof_Rains}) 
The algebraic operad of homology groups $H_{\ldot}(\MonR{n+1};\QQ)$ is a quadratic Koszul operad called the operad of \emph{2-Gerstenhaber} algebras in \cite{Etingof_Rains}.
We will denote this operad by $\Pois_1^{\odd}$. The precise definition, the proof of the coincidence of the operads $H_{\ldot}(\MonR{n+1};\QQ)$ and $\Pois_1^{\odd}$ and, finally, the reason for this name will be explained in Sections~\ref{sec::Mosaic::All} and~\ref{sec::ZZ_2}.
\end{enumerate}

\subsection{(Pure) Cactus groups}
\label{sec::Cacti}
The symmetric group $S_n$ act on the space $\MonR{n+1}$ while permuting the labels of the marked points keeping untouched the label $0$.
Unfortunately this action has orbits of different size, thus the space  $\MonR{n+1}/S_n$ is an orbifold rather than a manifold. 
The orbifold fundamental group of $\MonR{n+1}/S_n$ is easy to compute (\cite{Dev},\cite{DJS},\cite{HK_cacti}) out of the natural cell decomposition of $\MonR{n+1}$. 
We call the corresponding group \emph{cactus group} and denote it by $\Cacti_n$ motivated by Picture~\eqref{pict::cactus::example}. 
The group $\Cacti_n$ has the following presentation via generators and relations(\cite{DJS}):
\[
\Cacti_n:= \left\langle 
\begin{array}{c}
s_{pq}, \\
1\leq p < q\leq n
\end{array}
\left|
\begin{array}{c}
s_{pq}^2=1; \\
s_{pq} s_{kl} =
s_{kl} s_{pq}, \text{ if } [pq]\cap[kl]=\emptyset, \\
s_{pq} s_{kl} =
s_{p+q-l,p+q-k} s_{p q}, \text{ if } [pq]\supset [kl] \\
\end{array}
\right.
\right\rangle
\]
The generator $s_{p q}$ can be presented by a path on $\MonR{n+1}$
that starts in a marked non-degenerate curve
\(
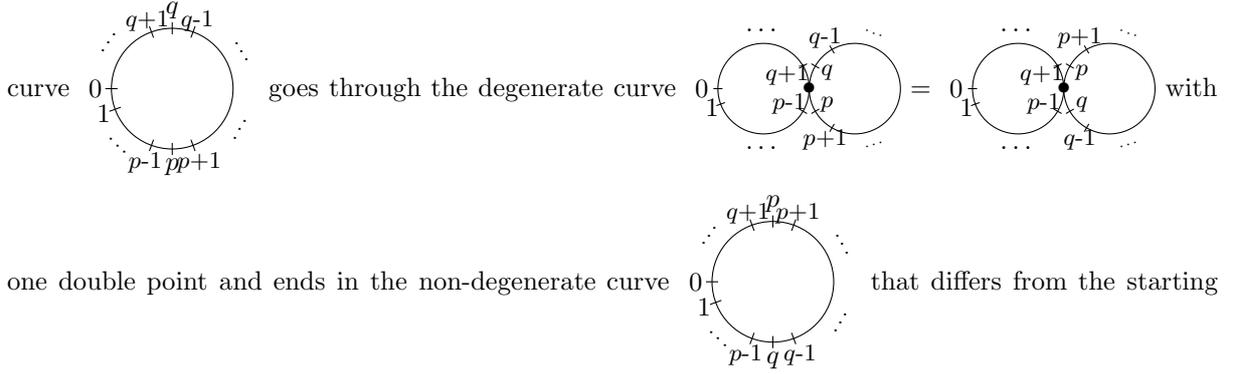
\begin{tikzpicture}[scale=0.8]
\draw (0,0) circle (1);
\draw (-1.1,0) -- (-0.9,0);
\node[left] at (-1,0) {$0$};
\draw[rotate=20] (-1.1,0) -- (-0.9,0);
\draw[rotate=20] (-1.2,0) node{$1$};
\draw[rotate=45] (-1.3,0) node[transform shape,rotate=90]{$\dots$};
\draw[rotate=-35] (-1.3,0) node[transform shape,rotate=-90]{$\dots$};
\draw[rotate=-20] (0,-1.1) -- (0,-0.9);
\draw[rotate=-20] (0,-1.3) node{\small{$p$-$1$}};
\draw[rotate=20] (0,-1.1) -- (0,-0.9);
\draw[rotate=20] (0,-1.3) node{\small{$p$+$1$}};
\draw (0,-1.1) -- (0,-0.9);
\node at (0,-1.3) {$p$};
\draw (0,1.1) -- (0,0.9);
\node [above] at (0,1) {$q$};
\draw[rotate=-20] (0,1.1) -- (0,0.9);
\draw[rotate=-20] node at (0,1.2) {\small{$q$-$1$}};
\draw[rotate=20] (0,1.1) -- (0,0.9);
\draw[rotate=20] node at (0,1.2) {\small{$q$+$1$}};
\draw[rotate=-30] (1.3,0) node[transform shape,rotate=90]{$\dots$};
\draw[rotate=30] (1.3,0) node[transform shape,rotate=90]{$\dots$};
\end{tikzpicture}
\)
goes through the degenerate curve 
\(
\begin{tikzpicture} [scale=0.6]
\draw (0,0) circle (1);
\draw (-1.1,0) -- (-0.9,0);
\node[left] at (-1,0) {$0$};
\draw[rotate=20] (-1.1,0) -- (-0.9,0);
\draw[rotate=20] (-1.2,0) node{$1$};
\draw (0,-1.3) node{$\dots$};
\draw[rotate=-30] (1.1,0) -- (0.9,0);
\draw[rotate=-30] (0.65,0) node{\small{$p$-$1$}};
\draw[rotate=30] (1.1,0) -- (0.9,0);
\draw[rotate=30] (0.6,0) node{\small{$q$+$1$}};
\draw (0,1.3) node{$\dots$};
\begin{scope}[shift={(2,0)}]
\draw (0,0) circle (1);
\draw[rotate=30] (-1.1,0) -- (-0.9,0);
\draw[rotate=30] (-0.7,0) node{\small{$p$}};
\draw[rotate=-30] (-1.1,0) -- (-0.9,0);
\draw[rotate=-30] (-0.7,0) node{\small{$q$}};
\draw[rotate=60] (-1.1,0) -- (-0.9,0);
\draw[rotate=60] (-1.3,0) node{\small{$p$+$1$}};
\draw[rotate=-60] (-1.1,0) -- (-0.9,0);
\draw[rotate=-60] (-1.3,0) node{\small{$q$-$1$}};
\draw[rotate=-20] (0,1.3) node[transform shape]{$\dots$};
\draw[rotate=20] (0,-1.3) node[transform shape]{$\dots$};
\end{scope}
\draw (1,0) node{$\bullet$};
\end{tikzpicture}
=
\begin{tikzpicture} [scale=0.6]
\draw (0,0) circle (1);
\draw (-1.1,0) -- (-0.9,0);
\node[left] at (-1,0) {$0$};
\draw[rotate=20] (-1.1,0) -- (-0.9,0);
\draw[rotate=20] (-1.2,0) node{$1$};
\draw (0,-1.3) node{$\dots$};
\draw[rotate=-30] (1.1,0) -- (0.9,0);
\draw[rotate=-30] (0.65,0) node{\small{$p$-$1$}};
\draw[rotate=30] (1.1,0) -- (0.9,0);
\draw[rotate=30] (0.6,0) node{\small{$q$+$1$}};
\draw (0,1.3) node{$\dots$};
\begin{scope}[shift={(2,0)}]
\draw (0,0) circle (1);
\draw[rotate=-30] (-1.1,0) -- (-0.9,0);
\draw[rotate=-30] (-0.7,0) node{\small{$p$}};
\draw[rotate=30] (-1.1,0) -- (-0.9,0);
\draw[rotate=30] (-0.7,0) node{\small{$q$}};
\draw[rotate=-60] (-1.1,0) -- (-0.9,0);
\draw[rotate=-60] (-1.3,0) node{\small{$p$+$1$}};
\draw[rotate=60] (-1.1,0) -- (-0.9,0);
\draw[rotate=60] (-1.3,0) node{\small{$q$-$1$}};
\draw[rotate=-20] (0,1.3) node[transform shape]{$\dots$};
\draw[rotate=20] (0,-1.3) node[transform shape]{$\dots$};
\end{scope}
\draw (1,0) node{$\bullet$};
\end{tikzpicture}
\) with one double point
and ends in the non-degenerate curve 
\(
\begin{tikzpicture}[scale=0.8]
\draw (0,0) circle (1);
\draw (-1.1,0) -- (-0.9,0);
\node[left] at (-1,0) {$0$};
\draw[rotate=20] (-1.1,0) -- (-0.9,0);
\draw[rotate=20] (-1.2,0) node{$1$};
\draw[rotate=45] (-1.3,0) node[transform shape,rotate=90]{$\dots$};
\draw[rotate=-35] (-1.3,0) node[transform shape,rotate=-90]{$\dots$};
\draw[rotate=-20] (0,-1.1) -- (0,-0.9);
\draw[rotate=-20] (0,-1.3) node{\small{$p$-$1$}};
\draw[rotate=20] (0,-1.1) -- (0,-0.9);
\draw[rotate=20] (0,-1.3) node{\small{$q$-$1$}};
\draw (0,-1.1) -- (0,-0.9);
\node at (0,-1.3) {$q$};
\draw (0,1.1) -- (0,0.9);
\node [above] at (0,1) {$p$};
\draw[rotate=-20] (0,1.1) -- (0,0.9);
\draw[rotate=-20] node at (0,1.2) {\small{$p$+$1$}};
\draw[rotate=20] (0,1.1) -- (0,0.9);
\draw[rotate=20] node at (0,1.2) {\small{$q$+$1$}};
\draw[rotate=-30] (1.3,0) node[transform shape,rotate=90]{$\dots$};
\draw[rotate=30] (1.3,0) node[transform shape,rotate=90]{$\dots$};
\end{tikzpicture}
\) that differs from the starting point of the path but belongs to the same $S_n$-orbit.

The fundamental group of $\MonR{n+1}$ is called the \emph{Pure Cactus group} (notation $\PCacti_n$) and coincides with the kernel of the surjection 
\begin{equation}
\begin{array}{rcc}
s_{pq} & \mapsto & {\left(\begin{smallmatrix}
1 &\ldots& p-1& p& p+1& \ldots & q & q+1 & \ldots & n \\
1 &\ldots& p-1& q& q-1& \ldots & p & q+1 & \ldots & n 
\end{smallmatrix}\right) } \\
\Cacti_n & \twoheadrightarrow & S_n
\end{array}
\end{equation}
Unfortunately, we do not know any simple presentation via generators and relations of the group $\PCacti_n$. 
However, let us report some results on this group out of~\cite{Etingof_Rains}.
\begin{remark}
\begin{enumerate}
\setlength\itemsep{-0.4em}	
\item The group $\PCacti_n$ is a finitely presented torsion-free group because the corresponding Eilenberg-MacLane space is a finite-dimensional compact manifold.
\item(\cite{Etingof_Rains}Theorem~3.8) The abelianization $\PCacti_n/(\PCacti_n,\PCacti_n)$ is isomorphic to ${\ZZ}^{r}\oplus E$ where $E$  is a vector space over $\FF_2$.
\end{enumerate}	
\end{remark}
Note that (pure) cactus groups have a lot of common properties with the (pure) braid groups. For example, there is a simple generalization of cactus groups for other Dynkin diagrams (\cite{DJS} Theorem 4.7.2). See also \cite{Rains} and \cite{Loseu_Cacti} for applications and references therein.

\subsection{Lie algebras associated with completions of Pure Cactus groups}
\label{sec::Cacti::Lie}
There are three families of Lie algebras associated with the groups $\PCacti_n$ discussed in detail in~\cite[Section~3]{Etingof_Rains}.
We index the Lie algebras defined over $\ZZ$ (resp. over $\ZZ[\frac{1}{2}]$ or over $\QQ$) by a special superscript $\ZZ$ (resp. $\ZZ[\frac{1}{2}]$ or $\QQ$) if there is no special superscript we work with ordinary Lie algebras defined over $\QQ$ or over any other field of zero characteristics.  
\begin{itemize}
	\item[($\qL_n^{\ZZ}$)] 
 The cohomology of the real locus $\MonR{n+1}$ is shown in~\cite{Etingof_Rains} to be isomorphic to the following quadratic algebra modulo $2$-torsion:
\begin{equation}
\label{eq::H::M0nR::2}
 \ZZ\left[
\begin{array}{c}
\bar\nu_{ijk}, 1\le i,j,k\le n \\
\bar{\nu}_{ijk} = (-1)^{\sigma}\bar{\nu}_{\sigma(i)\sigma(j)\sigma(k)}
\\
\deg(\bar{\nu}_{ijk})=1
\end{array} 
\left|
\begin{array}{c}
\bar{\nu}_{ijk}\bar{\nu}_{ijl}= 0,
\\
\bar{\nu}_{ijk}\bar{\nu}_{klm}
+\bar{\nu}_{jkl}\bar{\nu}_{lmi}
+\bar{\nu}_{klm}\bar{\nu}_{mij} + 
\\
+\bar{\nu}_{lmi}\bar{\nu}_{ijk}
+\bar{\nu}_{mij}\bar{\nu}_{jkl} = 0
\end{array}
\right.
\right]
\end{equation}
 The corresponding quadratic dual Lie algebra (defined over the integers) has the following presentation
	\[\qL_n^{\ZZ}:= 
	\Lie^{\ZZ} \left(
	\begin{array}{c}
	\nu_{ijk}, 1\le i,j,k\le n \\
	\nu_{ijk} = (-1)^{\sigma}\nu_{\sigma(i)\sigma(j)\sigma(k)}
	\\
	\deg(\nu_{ijk})=0
	\end{array} 
	\left|
	\begin{array}{c}
	{[\nu_{ijk},\nu_{pqi}+ \nu_{pqj} + \nu_{pqk}]= 0 }\\
	{ [\nu_{ijk},\nu_{pqr}]=0  } 
	\\
	{\text{for } \#\{i,j,k,p,q,r\}=6.}
	\end{array}
	\right.
	\right)
	\]
	\item[($\Lcs_n^{\ZZ}$)] 
	The commutator in the group $\PCacti_n$ defines a $\ZZ$-Lie algebra structure on the 
	associated graded space to the lower central series filtration: 
	\[\PCacti_n \supset \PCacti_n^{2}:=(\PCacti_n,\PCacti_n)\supset \cdots \supset  \PCacti_n^{p}:=(\PCacti_n,\PCacti_n^{p-1})\supset \cdots \]
	that we denote by $\Lcs_n'^{\ZZ}$.
	The quotient of $\Lcs_n'^{\ZZ}$ by the $2$-torsion is known to be the $\ZZ$-Lie algebra $\Lcs_n^{\ZZ}$ with the same set of generators as $\qL_n^{\ZZ}$.
	Theorem~3.9 of~\cite{Etingof_Rains} states that there exists a natural $S_n$-equivariant surjective map 
	\[\psi_n: \qL_n^{\ZZ} \twoheadrightarrow \Lcs_n^{\ZZ}.
	\]
	\item[($\LCacti_n$)] 
	Let $\widehat{\PCacti_n}$ be the prounipotent (=Malcev) completion of $\PCacti_n$ over $\QQ$. 
	Let $\LCacti_n:=\Lie(\widehat{\PCacti})$ be the $\QQ$-Lie algebra associated to $\widehat{\PCacti_n}$. That is $\LCacti_n$ is the set of primitive elements of the complete Hopf algebra $\QQ[\widehat{\PCacti_n}]$.
	The rational homotopy theory (\cite{Quillen::Rat::Hom}) predicts that the associated graded to the lower central series filtration on $\Lie(\widehat{\PCacti})$ is isomorphic to the rationalization of $\Lcs'_n$:
\[\mathsf{gr}(\LCacti_n):=\mathsf{gr}(\Lie(\widehat{\PCacti_n}))\simeq {\Lcs_n'}^{\ZZ}\otimes \QQ = \Lcs_n^{\ZZ}\otimes\QQ =: \Lcs_n^{\QQ}\] 
\end{itemize}

Let us also denote by $\widehat{\Cacti_n}$ the proalgebraic group 
\begin{equation}
\label{eq::cacti::completion}
\Cacti_n\times_{\PCacti_n} \widehat{\PCacti_n}
\end{equation}
that corresponds to the prounipotent completion of the subgroup of pure cacti elements.

\subsection{(Pure) braid group and completions}
\label{sec::PBn}
Let us recall (after \cite{Kohno},\cite{Drinfeld_Quasi_Hopf},\cite{Drinfeld_Gal_Q}) the analogous known results on the (pure) braid group $(\calP)\B_n$ on $n$ strands:
\[
\B_n:= \left\langle
\begin{array}{c} 
r_{i,i+1} \\
i=1,\ldots, n-1
\end{array} 
\left| 
\begin{array}{c} 
r_{i-1,i} r_{i,i+1} r_{i-1,i} = r_{i,i+1} r_{i-1,i} r_{i,i+1} \\
r_{i,i+1} r_{j,j+1} = r_{j,j+1} r_{i,i+1}, \text{ if } |i-j|>1
\end{array}
\right.
\right\rangle
\]
The pure braid group $\PB_n$ is the kernel of the surjection $\B_n \stackrel{r_{i,i+1}\mapsto (i i+1)}{\longrightarrow} S_n$.
The configuration space of $n$ (numbered) points on $\CC=\RR^{2}$ is the Eilenberg-MacLane space of $\B_n$ (respectively $\PB_n$).

 The prounipotent completion of the pure braid group $\PB_n$ coincides with the rationalization of the lower central series completion and is denoted by $\widehat{\PB_n}$.
 The corresponding Lie algebra is called  Drinfeld-Kohno Lie algebra    (denoted by $\tkd(n)^{\ZZ}$ and $\tkd(n)^{\QQ}$ respectively) and has the following presentation by generators and quadratic relations:
 \[
 \tkd(n):=\Lie\left(
 \begin{array}{c}
 t_{ij}, 1\le i\neq j\le n \\
 t_{ij} = t_{ji}
 \end{array} 
 \left|
 \begin{array}{c}
 {[t_{ij},t_{ik}+ t_{jk}]= 0 } \\
 { [t_{ij},t_{kl}]=0, }
 \\
 {\text{for } \#\{i,j,k,l\}=4. }
 \end{array}
 \right.
 \right)
 \]
The kernel of the map $\tkd(n)\to \tkd(n-1)$ is known to be isomorphic to the free Lie algebra generated by the set $\{t_{in}| 1\leq i\leq n-1\}$. In particular, $\tkd(n)^{\ZZ}$ is the free $\ZZ$-module.
Analogously to~\eqref{eq::cacti::completion} we have:
\[
\widehat{\B_n}:= \B_n\times_{\PB_n} \widehat{\PB_n}
\]

\begin{example}
For any given collection of objects $X_1,\ldots,X_n$ in a braided tensor category $(\calC,\otimes,R,\Phi)$ the operators \[
r_{i,i+1}^{\calC}:= R_{X_i,X_{i+1}}: X_1\otimes\cdots\otimes X_i\otimes X_{i+1}\otimes \cdots \otimes X_n \longrightarrow X_1\otimes\cdots\otimes X_{i+1}\otimes X_{i}\otimes \cdots \otimes X_n
\]
 defines the action of the Braid group on the union of tensor products $\{X_{\sigma(1)}\otimes\cdots\otimes X_{\sigma(n)} | \sigma\in S_n \}$. Respectively, the pure braid group $\PB_n$ admits a natural representation on $X_1\otimes \cdots\otimes X_n$ .  		
\end{example}

\subsection{Coboundary categories and Drinfeld's Unitarization trick}
\label{sec::Drinfeld::unit}
\emph{Cactus} and \emph{Pure Cactus} groups naturally arise in the theory of coboundary monoidal categories introduced by Drinfeld (\cite{Drinfeld_Quasi_Hopf}).
Following \cite{HK_cacti} we recall that a \emph{coboundary monoidal category} is a monoidal category $\calC$ together with a (functorial in $X,Y\in \calO b(\calC)$) commutor morphism:
\[
c_{X,Y}: X\otimes Y \rightarrow Y\otimes X \text{ such that } c_{X,Y}\circ c_{Y,X} = Id_{X\otimes Y},
\]
and yielding the following relations for all $X,Y,Z\in\calO b(\calC)$:
\[
\begin{tikzcd}
&(X\otimes Y) \otimes Z \arrow[r, "\Phi_{X,Y,Z}"] \arrow[dl, "c_{X,Y}\otimes Id_Z"'] & X\otimes (Y\otimes Z) \arrow[dr, "Id_{X}\otimes c_{Y,Z}"]  & \\
(Y\otimes X)\otimes Z \arrow[dr, "c_{Y\otimes X,Z}"'] & & &  X\otimes (Z\otimes Y) \arrow[dl, "c_{X,Z\otimes Y}"] \\
& Z\otimes (Y \otimes X) & (Z\otimes Y) \otimes X \arrow[l, "\Phi_{Z,Y,X}"']
\end{tikzcd}
\]
where $\Phi_{X,Y,Z}: (X\otimes Y)\otimes Z \to X\otimes(Y\otimes Z)$ is the associativity isomorphism.

For each tensor product $X_1\otimes X_2 \otimes \cdots \otimes X_n$ denote by $s_{pq}^\calC$ the following composition of commutors (and associators):
\begin{equation}
\label{eq::cons::commutor}
c_{(X_{q-1}\otimes\cdots\otimes X_{p}),X_{q}}
\circ \cdots
\circ
c_{(X_{p+1}\otimes X_{p}),X_{p+2}}
\circ
c_{X_{p},X_{p+1}}
\end{equation}
that reverse the order of the subsequence of tensor multiples $X_{p}\otimes X_{p+1}\otimes \cdots \otimes X_q$ inside $X_1\otimes \cdots \otimes X_n$:
\[
s_{pq}^{\calC}:
X_1\otimes \cdots \otimes X_p\otimes \cdots \otimes X_q \otimes \cdots \otimes X_n
\longrightarrow
X_1\otimes \cdots \otimes X_q\otimes \cdots \otimes X_p \otimes \cdots \otimes X_n
\]
We omit identity operators, bracketings and associators in presentation~\eqref{eq::cons::commutor} for simplicity. 
\begin{proposition*}
	For any given collection of objects $X_1,\ldots,X_n\in \calC$	
	the operators $s_{pq}^{\calC}$ defines an action of the cactus group on the sum of the tensor products $\oplus_{\sigma\in S_n}X_{\sigma(1)}\otimes\cdots \otimes X_{\sigma(n)}$. Respectively, the pure cactus group $\PCacti_n$ admits a natural representation on $X_1\otimes \ldots\otimes X_n$ .  
\end{proposition*}
\begin{proof}
	 See~\cite[Theorem~7 (Section 3)]{HK_cacti}.
\end{proof}

Drinfeld's motivation of coboundary categories comes from the following unitarization trick:
\begin{example}(\cite{Drinfeld_Quasi_Hopf})
	\label{ex::drinfeld::coboundary}	
	Let $(\calC,\otimes,R,\Phi)$ be a braided monoidal category. Then the operator
	\[
	c_{X,Y}:= R_{X,Y} \circ (R_{Y,X}\circ R_{X,Y})^{-\frac{1}{2}} \in \widehat{\B_2}
	\] 
	is well defined in the completion of the braid group and defines a commutor and a structure of the coboundary category on $\calC$.  
\end{example}	
In particular, one can define a commutor for the category of finite-dimensional representations of the quantum group $U_q(\mathfrak{g})$ for generic $q$. It was shown in~\cite{Kamn_Uq} that the commutor makes sense for the case $q$ tends to $0$ and one has a nice example of a coboundary strict monoidal category with combinatorially defined commutor.   
\begin{example}
	The category of crystals of finite-dimensional representations of a given simple Lie algebra $\mathfrak{g}$ form a coboundary strict monoidal category. There are several different definitions of the commutor \cite{HK_cacti, HK_crystal,Kamn_Uq}. 
\end{example}

Drinfeld's unitarization trick~\ref{ex::drinfeld::coboundary} defines a map of prounipotent completions of nonpure groups and corresponding Lie algebras of prounipotent completions of pure groups for all $n\geq 3$:
\[
\bar{\xi}_n:\widehat{\Cacti_n} \rightarrow \widehat{\B_n}   \quad  \Rightarrow \quad
\tilde{\xi}_n^{\QQ}: \LCacti_n^{\QQ} \rightarrow \tkd(n)^{\QQ}
\]
The associated graded morphism with respect to the lower central series filtration is easy to compute on the level of generators of the corresponding Lie algebras and it happens to be defined  over integers (see Section 3.11 of~\cite{Etingof_Rains}):
\begin{equation}
\label{eq::Drinfeld::map}
\begin{tikzcd}
\xi_n: \ \ \qL_n^{\ZZ} \arrow[rr,"{\nu_{ijk}\mapsto \nu_{ijk}}"] & &
\Lcs_n^{\ZZ} \arrow[rr,"\nu_{ijk}\mapsto{[t_{ij},t_{jk}]}"] & &
\tkd(n)^{\ZZ}
\end{tikzcd}
\end{equation}

\subsection{Overview of results (without operads)}
\label{sec::results}
In this paper we use a lot the language of operads while formulating and proving the results. However, we decided to state the outline of the applications of our results to the \emph{pure cactus groups} and the associated Lie algebras that does not involve so far the word operad.
We prove the following for all $n\geq 3$: 
\begin{enumerate}
	\item The Lie algebras $\qL_n^{\QQ}$ are Koszul
	(Theorem~\ref{thm::HICG});
	\item The morphisms $\xi_n^{\QQ}:\qL_n^{\QQ}\to \tkd(n)^{\QQ}$ and $\tilde{\xi}_n^{\QQ}: \LCacti_n^{\QQ} \rightarrow \tkd(n)^{\QQ}$  are embeddings of Lie algebras (Theorem~\ref{thm::todd::t::emb} and Corollary~\ref{cor::Lcacti::tkd} respectively);
	\item $\MonR{n}$ is a rational $K(\pi,1)$ space
	(Corollary~\ref{cor::Mon::Kp1}) meaning that its $\QQ$-completion does not have higher homotopy groups;
	\item The spaces $\MonR{n}$ are not formal for all $n\geq 6$ as shown in~\cite{Etingof_Rains} (Proposition~3.13). This means, in particular, that 
		\[ \LCacti_n \not\simeq   \mathsf{gr}(\LCacti_n) \simeq  \Lcs_n^{\QQ}
		\]
	We found a huge but purely combinatorial dg-model of the space $\MonR{n+1}$ given by certain graphs and highly nontrivial differential (Theorem~\ref{thm::Mos::Model}). 	
\end{enumerate}
As it is mentioned in Section~3 of~\cite{Etingof_Rains} these results imply in addition the following conclusions (that were stated as conjectures in~\cite{Etingof_Rains}):
\begin{corollary}
\label{cor::qLie}	
\begin{enumerate}
	\item 
	The kernel of the projection $\pi_n:\qL_{n}^{\QQ}\to \qL_{n-1}^{\QQ}$ is a free Lie algebra (on infinitely many generators);
	\item $\qL_n^{\ZZ[\frac{1}{2}]}$ and its universal enveloping algebra $U(\qL_n^{\ZZ[\frac{1}{2}]})$ are free $\ZZ[\frac{1}{2}]$-modules;
	(Note that Conjecture 3.3 of~\cite{Etingof_Rains} is asking whether $\qL_n^{\ZZ}$ is a free $\ZZ$-module.)
	\item The map $\psi_n:\qL_n^{\ZZ[\frac{1}{2}]}\to \Lcs_n^{\ZZ[\frac{1}{2}]}$ is an isomorphism. 
	Thus, the only torsion of the lower central series completion $\Lcs_n'$ of $\PCacti_n$ is the $2$-torsion.  
\end{enumerate}
\end{corollary}
\begin{proof}
We have the following diagram of maps of Lie algebras defined over the integers:
	\[
	\begin{tikzcd}
	0 \arrow[r] & \ker \pi_n \ar[r,hook] \ar[d]  & \qL_n^{\ZZ}  \ar[r,twoheadrightarrow,"\pi_n"] \ar[d,"\xi_n"]   & \qL_{n-1}^{\ZZ} \ar[r] \ar[d,"\xi_{n-1}"] & 0 \\
	0 \arrow[r] & \Lie^{\ZZ}(t_{1n},\ldots,t_{n-1,n}) \ar[r,hook]  & \tkd(n)^{\ZZ}  \ar[r,twoheadrightarrow]   & \tkd(n-1)^{\ZZ} \ar[r]  & 0 
	\end{tikzcd}
	\]
Theorem~\ref{thm::todd::t::emb} predicts that all vertical arrows tensored with $\QQ$ are known to be embeddings of $\QQ$-Lie algebras. Therefore, the kernel of the projection $\pi_n$ tensored with $\QQ$ is a Lie subalgebra of the free Lie algebra on $n-1$ generators. The Shirshov-Witt Theorem states that any Lie subalgebra of the free Lie algebra is also free (see e.g.~\cite{Shirshov}, \cite{Witt}), thus proving the first assertion above.
While looking at the Hilbert series of the Universal enveloping one can easily show that the number of generators of $\ker\pi_n$ is infinite and find degrees of generators. (See~\cite[\S3.2]{Etingof_Rains} for details.)

The second item is an easy corollary of the quadratic dual statement.
It was shown in~\cite{Etingof_Rains} that the integral cohomology of $\MonR{n+1}$ has only $2$-torsion and the cohomology $H^{\udot}(\MonR{n+1},\ZZ[\frac{1}{2}])$ is the $\frac{1}{2}$-localization of the quadratic algebra defined in~\eqref{eq::H::M0nR::2}.  In particular, this implies that $H^{2}(\MonR{n+1},\ZZ[\frac{1}{2}])$ is a free finitely generated $\ZZ[\frac{1}{2}]$-module which we denote by $S$. Let us denote by $S^{\perp}$ the $\ZZ[\frac{1}{2}]$-module spanned by quadratic relations in the cohomology ring and by the commutativity relations:
\[\nu_{ijk}\wedge \nu_{pqr} = - \nu_{pqr} \wedge \nu_{ijk}.\]
In particular, $S$ is the $\ZZ[\frac{1}{2}]$-span of quadratic relations in the Lie algebra $\qL_n$ and in its universal enveloping algebra $U(\qL_n)$.
Moreover, we have the following isomorphism of free $\ZZ[\frac{1}{2}]$-modules:
\[
\ZZ\left[\frac{1}{2}\right]\{\nu_{ijk}\} \otimes_{\ZZ\left[\frac{1}{2}\right]} \ZZ\left[\frac{1}{2}\right]\{\nu_{ijk}\} \simeq S\oplus S^{\perp}
\]
The description of the ideal generated by quadratic algebras predicts the following isomorphisms for the graded component of the universal enveloping algebra:
\begin{multline*}
	U(\qL_n^{\ZZ\left[\frac{1}{2}\right]})_{m} \simeq \frac{ \ZZ\left[\frac{1}{2}\right]\{\nu_{ijk}\}^{\otimes m} }{+_{i=1}^{n-1} \left( \ZZ\left[\frac{1}{2}\right]\{\nu_{ijk}\}^{\otimes i-1} \otimes S \otimes \ZZ\left[\frac{1}{2}\right]\{\nu_{ijk}\}^{\otimes m-i-1} \right) } \simeq \\
	\simeq \bigcap_{i=1}^{n-1} \left(\ZZ\left[\frac{1}{2}\right]\{\nu_{ijk}\}^{\otimes i-1} \otimes S^{\perp} \otimes \ZZ\left[\frac{1}{2}\right]\{\nu_{ijk}\}^{\otimes m-i-1}\right)
\end{multline*}
In particular, since the intersection of free $\ZZ[\frac{1}{2}]$ modules is free we get the second item of Corollary~\ref{cor::qLie}.

Finally, the map $\xi_n^{\ZZ[\frac{1}{2}]}:\qL_n^{\ZZ[\frac{1}{2}]}\to \tkd(n)^{\ZZ[\frac{1}{2}]}$ is the map of free $\ZZ[\frac{1}{2}]$-modules such that tensored with $\QQ$ it is an embedding. Consequently, $\xi_n^{\ZZ[\frac{1}{2}]}$ is also an embedding.  

As mentioned in~\eqref{eq::Drinfeld::map} the map $\xi_n^{\ZZ[\frac{1}{2}]}:\qL_n^{\ZZ[\frac{1}{2}]}\hookrightarrow \tkd(n)^{\ZZ[\frac{1}{2}]}$ factors through the surjection $\psi_n^{\ZZ[\frac{1}{2}]}:\qL_n^{\ZZ[\frac{1}{2}]} \twoheadrightarrow \Lcs_n^{\ZZ[\frac{1}{2}]}$ and, therefore, $\psi_n^{\ZZ[\frac{1}{2}]}$ is an isomorphism.
\end{proof}

\section{Conceptual algebraic model of the mosaic operad}
\label{sec::Mosaic::All}
This chapter is devoted to the understanding of the cell decomposition of the mosaic operad that is compatible with the operad structure.
We do not recall the notion of an operad and Koszul duality for operads and refer to the original paper of Ginzburg and Kapranov~\cite{Ginz_Kapranov} and the modern book on algebraic operads~\cite{LV} and references therein.

\subsection{$\ZZ_2$-action on the associative operad}
\label{sec::As/2}
Recall that the symmetric associative operad $\Ass$ is a quadratic Koszul operad whose space of $n$-ary operations is spanned by operations $\{x_{\sigma(1)}\ldots x_{\sigma(n)}|\sigma\in S_n\}$.
The associative operad $\Ass$ has an automorphism $\tau^{!}$ of order $2$ that flips the order of the multiplication:
\begin{equation}
\label{eq::tau::Ass}
\tau^{!}: x_1 x_2 \rightarrow x_2 x_1, \quad  \tau^{!}: x_{\sigma(1)}\ldots x_{\sigma(n)} \to 
x_{\sigma(n)}\ldots x_{\sigma(1)}.
\end{equation}
The associative operad is Koszul self-dual $\Ass^{!}=\Ass$ and the Koszul dual automorphism $\tau$ of order $2$ has the following presentation:
\begin{equation}
\label{eq::tau!::Ass}
\tau: x_1 x_2 \rightarrow - x_2 x_1, \quad 
\tau: x_{\sigma(1)}\ldots x_{\sigma(n)} \to 
(-1)^{n-1} x_{\sigma(n)}\ldots x_{\sigma(1)}
\end{equation}
The subspace of invariants of the corresponding $\ZZ_2$-action assembles an operad which we denote by $\AsM$.
In particular, we have
\begin{equation}
\label{eq::AsM::dim}
\dim(\AsM(n))=\frac{n!}{2} \quad \forall n\geq 2.
\end{equation}
We will find the presentation of this operad in terms of generators and relations later in Theorem~\ref{thm::AsM::def}.
Let us denote by ${\AsMd}$ the corresponding (linear dual) cooperad in the category of vector spaces. Note that ${\AsMd}$ is the set of $\ZZ_2$-coinvariants of the cooperad $\Ass^{\dual}$.  

\subsection{The Mosaic operad is bar-dual to $\AsM$}
\label{sec::Mos=As}
Let us come back to the geometry of the moduli spaces of curves.
The space $\MonR{n+1}$ admits a stratification by the number of irreducible components of a stable curve.
The codimension one open strata consists of degenerate curves with two components with at least $3$ points on each component. Thus, the codimension one closed strata are isomorphic to $\MonR{|I|+1}\times \MonR{|J|+1}$ and are numbered by decompositions $[0n]=I\sqcup J$ with $|I|,|J|\geq 2$. The stratification defines a cyclic operad structure on the union $\cup_{n\geq 2} \MonR{n+1}$.
This operad named the \emph{mosaic operad} was introduced by 
S.Devadoss in~\cite{Dev} and by Kapranov in~\cite{Kapranov::Moduli}. We will use the slightly shorter notation $\Mos$ for the Mosaic operad.
Note, that as shown by~\cite{DJS} the space $\MonR{n+1}$ is a $K(\PCacti_n,1)$-space, so the groups $\PCacti_n$ assemble an operad in the category of groups.
All open strata of the stratification are contractible and the standard stratification considered by Devadoss defines a cell decomposition of $\MonR{n+1}$ compatible with the operadic compositions.

The underling combinatorics of the latter cell decomposition are very similar to those of the cell decomposition of the Stasheff polytopes. Since we want to use the operadic structure and we want to have analogous constructions in higher dimensions we will use the notation $\FM_1(n)$ for the Fulton-McPherson compactification of the configuration space of poins on a line. We recall the description of $\FM_d$ as a model of $E_d$ in~\ref{sec::FM} and refer to~\cite{Lamb_Vol} for detailed discussions. 

Let $\Tre_n$ be the set of all planar rooted trees with $n$ leaves that are numbered from $1$ to $n$. In particular, $|\Tre_n|=n! a_n$ where $a_n$ are the generalized Catalan numbers, also called the little Schr\"oder numbers (see e.g.~\cite{Gen_Catalan} and see Sloanes OEIS number
A001003 \cite{Sloane}).
The cells in the standard cell decomposition of the Stasheff polytope $\FM_1(n)$ are indexed by elements of $\Tre_n$.
The vertex splitting operation on trees defines the combinatorics of the boundary maps in the cell decomposition. In particular, the codimension of a cell equals the number of inner vertices (see~\cite{Loday_Stasheff} and references therein).

Let $\tau: \Tre_n \rightarrow \Tre_n$ be the reflection of a planar tree around the $y$-axes.\footnote{The line around which the reflection is made does not really matter. We just have to change the orientation of the plane.} 
For each vertex $v\in\vertices(T)$ of a tree $T$ we define the corresponding reflection $\tau_v$ of a maximal subtree of $T$ whose root coincides with $v$. 
We denote by $\eqvr$ the minimal equivalence relation generated by reflections in vertices.
In other words, we say that planar leaf-labeled trees are equivalent if they are connected by a finite composition of reflections $\tau_{v_1}\circ\ldots\circ \tau_{v_k}$.
For example, we have
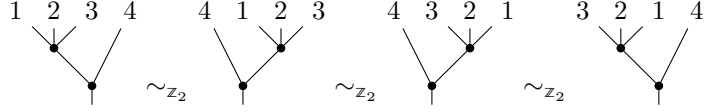
\[
\begin{tikzpicture}[scale=0.5]
\node[int] (v0) at (0,0) {};
\node[int] (v1) at (-1,1) {};
\coordinate (w0) at (0,-.5);
\node (w1) at (-2,2) {$\small{1}$};
\node (w2) at (-1,2) {$\small{2}$};
\node (w3) at (0,2) {$\small{3}$};
\node (w4) at (1,2) {$\small{4}$};
\draw (v0) edge (w0);
\draw (v1) edge (v0);
\draw (w4) edge (v0);
\draw (v1) edge (w1) edge (w2) edge (w3); 
\end{tikzpicture}	
\eqvr
\begin{tikzpicture}[scale=0.5]
\node[int] (v0) at (0,0) {};
\node[int] (v1) at (1,1) {};
\coordinate (w0) at (0,-.5);
\node (w1) at (2,2) {$\small{3}$};
\node (w2) at (1,2) {$\small{2}$};
\node (w3) at (0,2) {$\small{1}$};
\node (w4) at (-1,2) {$\small{4}$};
\draw (v0) edge (w0);
\draw (v1) edge (v0);
\draw (w4) edge (v0);
\draw (v1) edge (w1) edge (w2) edge (w3); 
\end{tikzpicture}	
\eqvr
\begin{tikzpicture}[scale=0.5]
\node[int] (v0) at (0,0) {};
\node[int] (v1) at (1,1) {};
\coordinate (w0) at (0,-.5);
\node (w1) at (2,2) {$\small{1}$};
\node (w2) at (1,2) {$\small{2}$};
\node (w3) at (0,2) {$\small{3}$};
\node (w4) at (-1,2) {$\small{4}$};
\draw (v0) edge (w0);
\draw (v1) edge (v0);
\draw (w4) edge (v0);
\draw (v1) edge (w1) edge (w2) edge (w3); 
\end{tikzpicture}	
\eqvr
\begin{tikzpicture}[scale=0.5]
\node[int] (v0) at (0,0) {};
\node[int] (v1) at (-1,1) {};
\coordinate (w0) at (0,-.5);
\node (w1) at (-2,2) {$\small{3}$};
\node (w2) at (-1,2) {$\small{2}$};
\node (w3) at (0,2) {$\small{1}$};
\node (w4) at (1,2) {$\small{4}$};
\draw (v0) edge (w0);
\draw (v1) edge (v0);
\draw (w4) edge (v0);
\draw (v1) edge (w1) edge (w2) edge (w3); 
\end{tikzpicture}	
\]
The description of the cell decomposition suggested in~\cite{Dev,DJS} can be summarized as follows.
\begin{proposition}
	\label{prp::mosaic::ass}	
	\begin{enumerate}
		\setlength\itemsep{-0.4em}
		\item The cells of the Devadoss's cell decomposition of  $\MonR{n+1}$ are indexed by elements of $\Tre_n/\eqvr$;
		\item There is a surjective cellular map of topological operads:
		\begin{equation}
		\label{eq::E_1::Mosaic}
		p:  \FM_1 \rightarrow \Mos.
		\end{equation}
		such that for all $T\in \Tre_n$ the restriction of the map $p_n$  on the face (cell) $U_{T}$ of the Stasheff polytope $\FM_1(n)$ is a diffeomorphism 
		between $U_T$ and the cell $U_{[T]}\subset\MonR{n+1}$, where $U_{[T]}$ is the cell of $\MonR{n+1}$ assigned to the class of $T$ in $\Tre_n/\eqvr$.
	\end{enumerate}	
\end{proposition}
\begin{proof}
	The proof is contained in~\cite{Dev} but is stated in a bit different form. Let us sketch the main idea.
	The open (top-dimensional) cells of $\FM_1(n)$ are numbered by trees from $\Tre_n$ with the unique inner vertex. In other words, they are numbered by permutations $\sigma\in S_n$ (or cyclic structures on $n+1$ letters) given by the labels on leaves. Respectively, the open cells of $\MonR{n+1}$ are numbered by dihedral structures on $n+1$ letters. Consequently, one identifies the cell associated with $\sigma$ and $\sigma^{op}$. The factorization property of the boundary defines by induction an equivalence $\eqvr$ on the set of leaf-labeled planar trees $\Tre_n$.
\end{proof}
While applying the functor of chains to the cell decomposition of $\Mos$ compatible with the operadic structure we end up with the following simple but curious observation.
\begin{theorem}
	\label{thm::Mosaic::chains}
	The cellular chains of the mosaic operad given by the cell decomposition of $\MonR{n+1}$ due to Devadoss assemble into a quasi-free symmetric operad that is a cobar construction of the cooperad ${\AsMd}$ (defined in the previous Section~\S\ref{sec::As/2}) :
	\[
	\cup_{n\geq 1} \Chains^{\Cell}(\MonR{n+1}) \simeq 
	\Omega({\AsMd})
	\]
	and the map $\Chains(p):\Chains^{\Cell}(\FM_1)\to \Chains^{\Cell}(\Mos)$ is equal to $\Omega(i^{\vee})$ where $i:\AsM \to \Ass$ is an embedding of operads.
\end{theorem}
\begin{proof}
	Thanks to Proposition~\ref{prp::mosaic::ass} we have a collection of isomorphism of vector spaces $\Omega({\AsMd})(n)$ and $Chains(\MonR{n+1})$ compatible with the operadic structure. What remains is to show that the differentials in these complexes coincide as well. The latter follows from the description of codimension one strata in stratification of $\MonR{n+1}$ that was mentioned before and the careful visualization of orientations of cells that is enough to work out for small $n$ thanks to the operadic induction.
\end{proof}

\subsection{Mosaic operad as a homotopy quotient}
\label{sec::Mos::E1/2}
Let us also mention another conceptual homotopical description of the mosaic operad. 
Suppose that a group $G$ acts on a topological operad $E$ by automorphisms. 
In particular $E$ is an operad in the category of $G$-spaces. The category of algebras over $E$ in $G$-spaces is equivalent to the category of algebras over the semidirect product $E\ltimes G$ in spaces. Where the set of operations $E\ltimes G(n)$ is isomorphic to $E(n)\times G^{\times n}$ and the composition is twisted by the action of $G$. (See e.g.~\cite{Ward_Groups} for detailed definitions.)
We claim that there is a homotopy pushout of topological operads that uniquelly defines the topological type of the mosaic operads:
	\[
\begin{tikzcd}
	\Mos \arrow[dr, phantom, "\lrcorner", very near start]
	&   pt \ar{l} \\
	\FM_1\ltimes_{\tau^!} \ZZ_2  \arrow[u,"p"] & \ZZ_2=O(1) \ar{u}\ar{l} 
\end{tikzcd}	
\]
However, we prefer avoiding technical details related to the world of topological operads and state everything on the level of algebraic operads where the meaning of bar and cobar constructions is much more transparent.

\begin{corollary}
\label{cor::Mos=E1/2}	
	The operad of cellular chains $\Chains(\Mos)$ is homotopy equivalent to the homotopy quotient $\frac{\Chains(E_1)\ltimes \QQ[\ZZ_2]}{\QQ[\ZZ_2]}$. 
	I.e. we have the following homotopy pushout square:
	\[
	\begin{tikzcd}
	\Chains^{\Cell}(\Mos) \arrow[dr, phantom, "\lrcorner", very near start]
	&   \QQ \ar{l} \\
	A_{\infty}\ltimes_{\tau^!} \QQ[\ZZ_2]  \ar{u} & \QQ[\ZZ_2] \arrow[u,"\tau^{!}\mapsto 1"'] \ar{l} 
	\end{tikzcd}
	\]
	where the group ring $\QQ[\ZZ_2]$ and $\QQ$ are considered as operads with only unary operations.	
\end{corollary}
\begin{proof}	
	The proof repeats the proof of the similar statement known for the complex moduli space (see e.g.~\cite{BV_Delta}). One may also find a more topological approach to this statement presented in~\cite{Ward_Groups} in a big generality. 
	
	The operad  ${\AsM}$ is  a suboperad of $\Ass$ and, moreover, has a right action of the associative operad $\Ass$:
	\[
	\label{eq::Ass::2comp}
	{\AsM}\circ \Ass \simeq \Hom_{\ZZ_2}(\QQ, \Ass) \circ \Ass \rightarrow 
	\Hom_{\ZZ_2}(\QQ, \Ass \circ \Ass)  \rightarrow \Hom_{\ZZ_2}(\QQ, \Ass) 
	\simeq {\AsM}
	\]
	and since the action of the group $\ZZ_2$ on the space of $n$-ary operations of the associative operad is free for $n\geq 2$ we have an isomorphism of right operadic modules over $\Ass$:
	\begin{equation}
	\label{eq::Ass::2comp}
	\oplus_{n\geq 2}\QQ[\ZZ_2]\otimes {\AsM}(n)  \simeq  \oplus_{n\geq 2}\Ass(n)
	\end{equation}
Let us use this isomorphism to show an isomorphism of dg-operads:
\begin{equation}
\label{eq::free::product}
\mathsf{gr}\calF_{\tau}(\Omega(\Ass^{\dual}))\ltimes\QQ[\ZZ_2] \simeq 	\Omega({\AsMd})*\QQ[\ZZ_2],
\end{equation}
where by $\calF_{\tau}$ we denote the filtration by the number of $\tau$ prescribed by the isomorphism~\eqref{eq::Ass::2comp} and  by $\calP*\calQ$ we denote the free product of operads $\calP$ and $\calQ$, whose algebras are objects that admit both $\calP$ and $\calQ$-structures that does not know anything about each other.
First, let us explain that we have an isomorphism of vector spaces. 
Indeed, the cobar construction $\Omega(\calP^{\vee})$ admits a basis indexed by operadic trees whose vertices are labeled by elements of the cooperad $\calP^{\vee}$.
Therefore, the elements of the cobar construction $\Omega(\QQ[\ZZ_2]\otimes[\Ass]_{\ZZ_2}^{\vee})=\Omega([\Ass]_{\ZZ_2}^{\dual}\oplus \tau\otimes[\Ass]_{\ZZ_2}^{\dual})$ are given by operadic trees whose vertices are at least trivalent and are indexed by elements of either $[\Ass]_{\ZZ_2}^{\vee}$ or $\tau[\Ass]_{\ZZ_2}^{\vee}$. Let us move the indexing $\tau$ from the vertex to the nearby outgoing edge of the corresponding  operadic tree and we get an operadic tree whose vertices are indexed by elements of the cooperad $[\Ass]_{\ZZ_2}^{\vee}$ and some inner or root edges are indexed by $\tau$. The semidirect product with $\QQ[\ZZ_2]$ allows to mark some of the incoming edges by $\tau$.
On the other hand, a basis of the free product $\calP*\calQ$ is given by operadic trees whose  vertices are labeled alternatively by elements of $\calP$ and $\calQ$.
Thus, in our case a basis of  $	\Omega({\AsMd})*\QQ[\ZZ_2]$ is given by the summation of operadic trees whose vertices of valency at least tree  are marked by elements of $\Omega({\AsMd})$ and bivalent vertices are labeled by $\tau$ or $1$. Moreover, bivalent vertices and vertices of higher valencies are going alternatively. Thus, instead of thinking about bivalent vertices that alternates with vertices of higher valency we can imagine an operadic tree with at least trivalent vertices whose internal vertices are labelled by elements of $\Omega({\AsMd})$ and all edges (including leaves) are labelled by $\tau$ or $1$. The latter obviously coincide with the suggested above basis of  $\Omega(\Ass^{\dual}))\ltimes\QQ[\ZZ_2]$. The compatibility with the operadic structure is the similar straightforward comparison of the composition of two operadic trees.
What follows the following isomorphisms:
\begin{multline*}	
\Chains(\Mos) \simeq \frac{\Chains(\Mos)*\QQ[\ZZ_2]}{\QQ[\ZZ_2]} 
\stackrel{\text{Thm.~\ref{thm::Mosaic::chains}}}{\simeq} 
\frac{	\Omega({\AsMd})*\QQ[\ZZ_2] }{\QQ[\ZZ_2]} 
\stackrel{\eqref{eq::free::product}}{\simeq} \\
\simeq
\frac{\mathsf{gr}\calF_{\tau}(\Omega(\Ass^{\dual}))) \ltimes \QQ[\ZZ_2]}{\QQ[\ZZ_2]} 
\simeq \frac{\Omega(\Ass^{\dual})) \ltimes \QQ[\ZZ_2]}{\QQ[\ZZ_2]} 
\stackrel{quis}\twoheadrightarrow \frac{ \Ass \ltimes \QQ[\ZZ_2]}{\QQ[\ZZ_2]}\, .
\end{multline*}
The filtration by $\tau$ drops down while looking for the quotient by the action of $\tau$ and the latter quasi-isomorphism follows from the fact that associative operad is Koszul self-dual.
\end{proof}

Recall that each configuration of intervals in $E_1$ can be considered as the intersection with the equator of a certain configuration of discs in $E_2$. We require that intersections of each disc of a configuration with the equator coincides with the diameter of this disc. We call the corresponding embedding of operads $E_1\to E_2$ by an \emph{equator embedding}.	
Unfortunately, the description of a topological model of the little discs operad $E_2$ and a cell decomposition that is compatible with the operad structure is much more complicated rather than the one known for $E_1$ given by Stasheff polytopes (see e.g.~\cite{Moerdijk::En}.)
Let us denote by $\Chains$ any functor of rational chains, which defines an algebraic model of a topological operad. 

\begin{corollary}
	\label{cor::map::Mos::E_2}
	There exists a homotopy map of operads $\psi:\Chains^{\Cell}(\Mos)\to \Chains(E_2)$ such that 
	the standard equator embedding $\Chains(E_1)\to \Chains(E_2)$ factorises through the composition of the map $\Chains(p):\Chains^{\Cell}(\FM_1)\to \Chains^{\Cell}(\Mos)$ (described in~\eqref{eq::E_1::Mosaic}) and $\psi$:
	\begin{equation}
	\label{eq::Mon::to::E2}
	\begin{tikzcd}
	\Chains^{\Cell}(\FM_1)  \arrow[drr,"\Chains(p)"']&&	\arrow[ll,"\simeq"']
	\Chains(E_1) \arrow[rr,"equator"]  & & \Chains(E_2) \\
	&& \Chains^{\Cell}(\Mos) \arrow[urr,dotted,"\exists \psi" description]
	\end{tikzcd}
	\end{equation}	
\end{corollary}
\begin{proof}
	Consider the automorphism $\bar\tau$ of order $2$ of the little discs operad $E_2$ given by the clockwise rotation over its center by $180^{\circ}$. Note that $\bar{\tau}$ is a homotopically trivial automorphism since one has a family of automorphisms given by rotations over its center by $0^{\circ}\leq t\leq 180^{\circ}$.
    Thus, there exists a model for $\Chains(E_2)$ such that $\tau$ acts trivially on $\Chains(E_2)$.
	Therefore, there is a homotopy equivalence between the semidirect product the operad $\Chains(E_2)$ and the group $\ZZ_2$ twisted by $\bar{\tau}$ and by the identity actions of the group $\ZZ_2$:
	\[
	\Chains(E_2)\ltimes_{\bar\tau}\ZZ_2 \simeq  \Chains(E_2) \ltimes_{\Id} \ZZ_2\, .
	\]
	It is clear, that the automorphism $\bar{\tau}$ of $E_2$ restricts to the automorphism $\tau^{!}$ of $E_1$.
	Consequently we have the following maps of (semi)direct products,
	\[
	\Chains(E_1\ltimes_{\tau^{!}} \ZZ_2 )\rightarrow \Chains(E_2\ltimes_{\bar{\tau}} \ZZ_2)  \simeq \Chains(E_2 \times \ZZ_2), 
	\] 
	whose homotopy quotient gives the desired factorization:
	\[
	\Chains(E_1)\stackrel{p}{\rightarrow} \Chains( \Mos) \simeq \Chains(\frac{E_1\ltimes \ZZ_2}{\ZZ_2} )\longrightarrow \Chains(\frac{E_2\ltimes_{\bar{\tau}}\ZZ_2}{\ZZ_2}) \simeq \Chains(\frac{E_2\times \ZZ_2}{\ZZ_2}) \simeq \Chains(E_2) \, .
	\]
\end{proof}
\begin{remark}
The Drinfeld unitarization trick (Example~\eqref{ex::drinfeld::coboundary}) defines a map of prounipotent completions of the groups $\widehat{\PCacti_n}\to \widehat{\PB_n}$.
An algebraic model of the latter map is the map of the differential forms of the corresponding Eilenberg-Maclane spaces
$\psi^{*}:\Omega(E_2) \to \Omega(\Mos)$ whose dual is denoted by 
 $\psi:\Chains(\Mos)\to \Chains(E_2)$. 
 Thanks to certain computations in deformation theory provided in Section~\S\ref{sec::Deformations} below we conclude that the map $\psi$ can be constructed as a deformation of the map of (Hopf) cooperads
$H(\psi^*):H(E_2) \to H(\Mos)$ whose Koszul-dual map is
 $H(\psi):\qL_n^{\QQ}\to \tkd(n)$ and, moreover, that the map $\psi$ is rigid.
 In other words, there exists a unique (in a proper sense) up to homotopy map of operad of prounipotent completions $\widehat{\PCacti_n}\to \widehat{\PB_n}$.
In particular, this implies that the Drinfeld map coincides with the one obtained in Corollary~\ref{cor::map::Mos::E_2}.
\end{remark}

\section{Computing $\ZZ_2$-invariants of the operads $\Com$, $\Pois$ and $\Ass$}
\label{sec::ZZ_2}
In this section we describe several algebraic operads coming from the action of the finite group on the well known operads and find out the algebraic description of the operad $\AsM$.

\subsection{{Convention on shifts in Koszul duality for operads}}
We deal a lot with Koszul duality for operads and would like to fix certain conventions.
First, we deal with cohomological degree convention of complexes, meaning that our differentials increases the (co)homological degree.
Let $\calP$ be an algebraic operad. 
We say that a structure of an algebra over a homologically shifted ($k$-suspended) operad $\s^k\calP$ on a chain complex $V^{\udot}$ is in one-to-one correspondence with the structure of a $\calP$-algebra on a shifted complex $V^{\udot}[k]$. In particular, the cohomological shift (also called the suspension) increases the homological degree of the space of $n$-ary operations $\calP(n)$ by $(n-1)$ and multiplies with a sign representation $\Sgn_n$:
\[
\s\calP(n) := \calP(n)[1-n]\otimes \Sgn_n
\]
Following the same ideology the cohomological shift (desuspension) of a cooperad shifts the degrees of cogenerators in the other direction. In particular, the cooperad $s^{-k}\Com^{\dual}$ is cogenerated by a single element of degree $-k$ which is skew-symmetric for $k$ odd.

In order to preserve the standard conventions suggested by~\cite{Ginz_Kapranov} that predict the Koszul duality between $\Com$ and $\Lie$ operads we pose the following:
\begin{convention}
Let $\calP$ be an augmented algebraic operad, $\calP^{\dual}$ be the corresponding coaugmented cooperad and $\bar \calP^{\dual}$ the coaugmentation coideal. Then as a chain complex the cobar construction $\Omega(\calP^{\dual})$ is isomorphic to the homological shift of the free operad generated by the shifted symmetric collection:
\[
\Omega(\calP^{\dual}):= \calF(\s^{-1}\bar\calP^{\dual})
\]
Respectively, we use the following degree conventions for the bar-construction:
\[
\calB(\calP):= \calF^{c}(\s\bar \calP)
\]
and we say that a Koszul dual operad $\calP^{!}$ is the homology of the cobar construction $\Omega(\calP^{\dual})$. 
\end{convention}
Note, that if the Koszul operad $\calP$ is generated by a single ternary element of degree $0$ then the operad $\calP^{!}$ is generated by a single ternary element of degree $1-2=-1$. 

\subsection{${\mathbb Z}_2$-invariants of the operad of commutative algebras}
\label{sec::Com::ZZ_2}
Recall that the operad $\Com$ of commutative algebras is the quadratic operad generated by one binary symmetric generator $\mu_2(\ttt,\ttt)$ satisfying the associativity relation $\mu_2(\mu_2(x_1,x_2),x_3) = \mu_2(\mu_2(x_2,x_3),x_1) = \mu_2(\mu_2(x_3,x_1),x_2)$. This operad admits an automorphism $\tau$ of order $2$ such that
\[\tau(\mu_2) = -\mu_2. \]  
\begin{lemma}
The $\ZZ_2$-invariants $[\Com]^{\ZZ_2}$ form a quadratic operad generated by the ternary symmetric operation $\mu_3(x_1,x_2,x_3):=\mu_2(x_1,\mu_2(x_2,x_3))$ subject to the following generalized associativity relation.
\begin{equation}
\label{eq::Pois::Com}
\forall \sigma\in S_5	\  \mu_3(\mu_3(x_1,x_2,x_3),x_4,x_5) = \mu_3(\mu_3(x_{\sigma(1)},x_{\sigma(2)},x_{\sigma(3)}),x_{\sigma(4)},x_{\sigma(5)}) 
\end{equation}
Concretely, this relation states that all quadratic monomials in $[\Com]^{\ZZ_2}$ are the same.

The relations~\eqref{eq::Pois::Com} form a quadratic Gr\"obner basis with respect to any compatible ordering of monomials suggested in~\cite{DK::Grob} and, in particular, this operad is Koszul.
\end{lemma}
\begin{proof}
	The space of $n$-ary operations $\Com(n)$ is one-dimensional and is spanned by $\mu^{\circ n-1}$. Hence $\tau(\mu^{\circ n-1})=(-1)^{n-1}\mu^{\circ n-1}$ and $[\Com]^{\ZZ_2} = \oplus_{k\geq 1} \Com(2k-1)$.
	
	In order to show the Gr\"obner basis property it is enough to notice that quotient of the shuffle operad by the ideal generated by the leading monomials of relations~\eqref{eq::Pois::Com} has the same size as $[\Com]^{\ZZ_2}$.
\end{proof}
We denote by $\Lie^{\odd}$ the operad that is Koszul dual to the operad $[\Com]^{\ZZ_2}$. (The notation refers to the fact that $\Lie^{\odd}$ has nontrivial operations only in odd arities.) The latter operad $\Lie^{\odd}$ is generated by a ternary skewsymmetric generator $\nu_3$ of degree $-1$ subject to the so called generalized Jacobi identity:
\begin{equation}
\label{eq::Pois_odd::Jacobi}
\sum_{\sigma\in S_5}(-1)^{\sigma} \nu_3(x_{\sigma(1)},x_{\sigma(2)}, \nu_3(x_{\sigma(3)},x_{\sigma(4)},x_{\sigma(5)})) =0
\end{equation}

\subsection{${\mathbb Z}_2$-invariants of the Poisson operad and its Koszul-dual operad}
\label{sec::Pois::Z_2}
The Poisson algebras are algebras over the Poisson operad.
The (graded) Poisson operad $\Pois_d$ is the quadratic operad generated by two binary operations: $\mu_2$ -- an associative commutative multiplication and $\nu_2$ -- the Lie bracket of degree $1-d$, such that the Lie bracket with a given element is a derivation of the multiplication:
\begin{equation}
\label{eq::Pois::def}
\Pois_d:= \calF\left(
\begin{array}{c} 
\mu_2(x_1,x_2) = \mu_2(x_2,x_1),\\
\nu_2(x_1,x_2)= (-1)^{d}\nu_2(x_2,x_1) \\
\end{array}
 \left| 
\begin{array}{c}
\mu_2(\mu_2(x_1,x_2),x_3) = \mu_2(x_1,\mu_2(x_2,x_3))
\\
\nu_2(\nu_2(x_1,x_2),x_3) + \nu_2(\nu_2(x_2,x_3),x_1) + \nu_2(\nu_2(x_3,x_1),x_2)=0
\\
\nu_2(\mu_2(x_1,x_2),x_3) = \mu_2(x_1,\nu_2(x_2,x_3)) + \mu_2(\nu_2(x_1,x_3),x_2).
\end{array} 
\right. \right)
\end{equation}
Our convention comes from the one used for the Little discs operad.
In particular, operad of Poisson algebras is denoted $\Pois_1$.
\begin{remark}
\begin{itemize}
\setlength\itemsep{-0.5em}
\item 
The Poisson operad 	$\Pois_d$ is Koszul, admits a quadratic Gr\"obner basis and its Koszul dual coincides with itself upto a homological shift:
\begin{equation}
\label{eq::Pois::Koszul}
(\Pois_d)^{!} = \s^{1-d}\Pois_d, \quad \mu_2 \stackrel{!}{\leftrightarrow}\nu_2
\end{equation}
\item	
For $d\geq 2$ the Poisson operad $\Pois_d$ coincides with the homology of the little discs operad $E_d$ which is known to be formal (\cite{Tamarkin_formality},\cite{Lamb_Vol}).
\item
For $d=1$ the little discs operad $E_1$ is also formal, but the homology coincides with the associative operad $\Ass$ which admits a filtration by commutators, such that the associated graded operad is isomorphic to $\Pois_1$.
\end{itemize}
\end{remark}
The orthogonal group $O(d)$ acts on the operad $E_d$ and,
 in particular, 
there exists an automorphism $\tau$ of order $2$ that changes the orientation of the unit disc.
Respectively, on the level of homology the Poisson operad has an automorphism $\tau^{!}$ of order $2$:
\[
\tau^{!}(\mu_2) = \mu_2, \quad \tau^{!}(\nu_2) = -\nu_2.
\]
The Koszul dual automorphism $\tau$ of order $2$ also defines an automorphism of the Poisson operad
\[
\tau(\mu_2) = -\mu_2, \quad \tau(\nu_2) = \nu_2,
\]
which we shall consider in the following lemma.

\begin{lemma}
\label{lem::Pois::inv}
\begin{enumerate}
\item 
The $\tau$-invariants $[\Pois_d]^{\ZZ_2}$ form the quadratic operad generated by (skew)-symmetric binary Lie bracket  $\nu_2(x_1,x_2)$ and by totally symmetric (commutative) ternary associative multiplication $\mu_3:=\mu_2(\mu_2(x_1,x_2),x_3)$ satisfying relation~\eqref{eq::Pois::Com} and the following analogue of the Leibniz identity:
\begin{equation}
\label{eq::Pois::Leibn}
\nu_2(\mu_3(x_1,x_2,x_3),x_4) = \mu_3(x_1,x_2,\nu_2(x_3,x_4)) + \mu_3(x_1,\nu_2(x_2,x_4),x_3) + \mu_3(\nu_2(x_1,x_4),x_2,x_3)
\end{equation}
\item For $n\geq 2$ we have $\dim([\Pois_d]^{\ZZ_2})=\frac{n!}{2}$.
\label{lem::item::dim}
\item
	The aforementioned relations form a quadratic Gr\"obner basis for the operad of $\tau$-invariants $[\Pois_d]^{\ZZ_2}$ with respect to the appropriate ordering that remembers the convention $\nu_2>\mu_3$ (see~\cite{Dotsenko_Distributive_Law} for the details of the ordering).
In particular, 	the analogue of the Leibniz relation~\eqref{eq::Pois::Leibn} defines a distributive law in the sense of~\cite{Markl_distr_law}.
\end{enumerate}
\end{lemma}
\begin{proof}
The automorphism $\tau$ does not interact with the Lie bracket $\nu_2$ and thus the Leibniz rule in the Poisson operad predicts that we have the following isomorphisms of $\tau$-invariants:
\[
[\Pois_d]^{\ZZ_2} = [\Com \circ \s^{1-d}\Lie]^{\ZZ_2} = [\Com]^{\ZZ_2}\circ \s^{1-d}\Lie 
\]
This implies the following computation of the generating series of dimensions of $[\Pois_d]^{\ZZ_2}$:
\begin{multline}
f_{[\Pois_d]^{\ZZ_2}}(t) = f_{[\Com]^{\ZZ_2}}(t) \circ f_{\Lie}(t) = 
\left(\sum_{k=0}^{\infty} \frac{t^{2k+1}}{(2k+1)!} \right)\circ (-\ln(1-t))= \\
\frac{e^{t} - e^{-t}}{2}\circ (-\ln(1-t)) = 
\frac{1}{2}\left(\frac{1}{1-t} - (1-t)\right) = 
t + \sum_{n\geq 2} \frac{t^{n}}{2}
\end{multline}
Consequently, $\dim([\Pois_d]^{\ZZ_2}(1))= 1$ and  $\dim([\Pois_d]^{\ZZ_2}(n))=\frac{n!}{2}$ for $n\geq 2$.
We proved item~\ref{lem::item::dim} of Lemma~\ref{lem::Pois::inv}.
The relations~\eqref{eq::Pois::Leibn} and \eqref{eq::Pois::Com} are obviously satisfied in $[\Pois_d]^{\ZZ_2}$.

The Koszulness of the operad~$\Pois_d$ is usually addressed to the notion of Distributive law of two Koszul operads $\Com$ and $\Lie$ which nicely defined in~\cite{Markl_distr_law}. Unfortunately the proof of the {\it Distributive Law criterion} suggested in~\cite{Markl_distr_law} is wrong and  the simplest known correct proof (nicely written in~\cite{Dotsenko_Distributive_Law}) uses the Gr\"obner basis machinery for operads (\cite{DK::Grob}). The main idea of~\cite{Dotsenko_Distributive_Law} is to suggest a new compatible ordering of shuffle monomials (based on the compatible ordering of monomials in the skew polynomial algebra) such that the generator $\nu_2$ is considered to be greater than $\mu_2$. 
The automorphism $\tau$ is defined in the free operad generated by $\mu_2$ and $\nu_2$ and multiply a monomial in $\nu_2$ and $\mu_2$ either by $1$ or by $-1$. Therefore $\tau$ is compatible with the ordering of monomials. 
Let us consider the same ordering of monomials for the suboperad $[\Pois_d]^{\ZZ_2}$. It is easy to see that thanks to the relation~\eqref{eq::Pois::Leibn} the set of normal words for the operad defined by binary and ternary generators subject to the other defining relations coincides with the set of normal words of $[\Com]^{\ZZ_2}\circ\s^{1-d}\Lie$, thus, has the same dimension $\frac{n!}{2}$ for all $n\geq 2$. Consequently,  
the defining relations of the operad $[\Pois_d]^{\ZZ_2}$ form a quadratic Gr\"obner basis with $\nu_2>\mu_2$.
\end{proof}
\begin{corollary}
	\label{thm::Pois_d_odd::koszul}
	The operad $[\Pois_d]^{\ZZ_2}$ is Koszul.
	The Cobar-construction  $\Omega([\Pois_d^{\dual}]_{\ZZ_2})$ considered as a quasi-free dg-operad generated by the dual cooperad $[\Pois_d^{\dual}]_{\ZZ_2}$ homologically shifted by $1$ is quasiisomorphic to the quadratic operad generated by the ternary operation $\bar\mu_3$ of degree $-1$ and the binary operation $\bar\nu_2$ of degree $d-1$.
\end{corollary}
\begin{notation}
	We denote the shifted Koszul dual operad $\s^{1-d} H(\Omega([\Pois_d]^{\dual}_{\ZZ_2}))$ by $\Pois_d^{\odd}$.
\end{notation}
We will use the same duality in notations between $\mu$ and $\nu$ as well as the operadic homological shift $\{1-d\}$  as one has for the Koszul duality for the Poisson operad~\eqref{eq::Pois::Koszul}.
\begin{corollary}	\label{cor:Poisdodd generators}
The operad $\Pois_d^{\odd}$ is generated by a binary operation  of degree $0$ (denoted $\mu_2:=\bar\nu_2\{d\}$) and a ternary operation $\nu_3:=\bar\mu_3\{d\}$ of homological degree $1-2d$ 
subject to the following symmetry conditions:
	 \[
	 \begin{array}{c} 
	 \mu_2(x_1,x_2) = \mu_2(x_2,x_1),\\
	 \nu_3(x_1,x_2,x_3)= (-1)^{d|\sigma|}\nu_3(x_{\sigma(1)},x_{\sigma(2)},x_{\sigma(3)}) \\
	 \end{array}
	 \]	
and the following quadratic relations:
\begin{gather}
\label{eq::Pois_odd::assoc}
\mu_2(\mu_2(x_1,x_2),x_3) = \mu_2(x_1,\mu_2(x_2,x_3))
\\
\label{eq::Pois_odd::Leibniz}
\nu_3(\mu_2(x_1,x_2),x_3,x_4) = \mu_2(x_1,\nu_3(x_2,x_3,x_4)) + \mu_2(\nu_3(x_1,x_3,x_4),x_2).
\\
\label{eq::Pois_odd::Jacobi}
\sum_{\begin{smallmatrix}
	\sigma\in S_5/S_2\times S_3, \\
	\sigma \text{ is even}
	\end{smallmatrix}
	} 
	\nu_3(x_{\sigma(1)},x_{\sigma(2)}, \nu_3(x_{\sigma(3)},x_{\sigma(4)},x_{\sigma(5)})) =0
\end{gather}
\end{corollary}
In other words the operad $\Pois_d^{\odd}$ is generated by commutative associative multiplication $\mu_2$ of degree $0$ and a ternary operation $\nu_3$ of homological degree $1-2d$ obeying the Leibniz rule \eqref{eq::Pois_odd::Leibniz}
and the generalized Jacobi identity~\eqref{eq::Pois_odd::Jacobi}.

\begin{remark} 
It is worth mentioning that the operads $\Pois_d$ as well as the operads $\Pois_d^{\odd}$ are Hopf operads, i. e., operads in the category of (counital) cocommutative coalgebras. 
In particular the space of $n$-ary operations of the cooperad $(\Pois_d^{\odd})^{\dual}$ is a commutative graded algebra with unit. This observation will be extensively used later (Sections~\ref{sec::Hopf::E_d},\ref{sec::Hopf::Mosaic}).
\end{remark}
\begin{remark}	
For $d=1$ the operad $\Pois_1^{\odd}$ coincides with the operad called \emph{operad of $2$-Gerstenhaber algebras} in~\cite{Etingof_Rains}.
\end{remark}
\begin{corollary}
The Leibniz rule~\eqref{eq::Pois_odd::Leibniz} defines a distributive law in the sense of~\cite{Markl_distr_law,Dotsenko_Distributive_Law} between the operad of commutative algebras generated by $\mu_2$ and the suboperad generated by the ternary bracket $\nu_3$ called $\Lie^{\odd}$. In particular, we have an isomorphism of symmetric collections:
\[
 \Pois_d^{\odd} \simeq  \Com \circ (\s^{1-d}\Lie^{\odd}) \text{ where } \Lie^{\odd}:= ([\Com]^{\ZZ_2})^{!}
\]
\end{corollary}
Note, that the operad $[\Com]^{\ZZ_2}$ is generated by a ternary operation of homological degree $0$. Respectively, the Koszul dual operad $\Lie^{\odd}$ is generated by a ternary operation of homological degree $-1$.
The standard relation for generating series for Koszul dual operads (even for non binary generated) leads to the equation
\[
f_{[\Com]^{\ZZ_2}}(-t)\circ f_{\Lie^{\odd}}(-t) = t 
\]
As mentioned earlier 
\[
f_{[\Com]^{\ZZ_2}}(t)= \sum_{n\geq 1} \frac{t^{2n-1}}{(2n-1)!} = \sinh(t)
\]
Consequently
\[ f_{\Lie^{\odd}}(t) := (-f_{[\Com]^{\ZZ_2}}(-t))^{\circ-1} = \mathsf{arcsinh}(t) = 
 \sum_{n\geq 0} \frac{(2n-1)!!}{(2n)!!}\frac{t^{2n+1}}{2n+1} = 
 \sum_{n\geq 0} \frac{\binom{2n}{n}}{4^n(2n+1)} t^{2n+1}. 
 \] 
The space of $n$-ary operation $\Lie^{\odd}(n)$ differs from zero only for $n$ odd and has homological degree $\frac{1-n}{2}$. Hence, we can write a generating series with additional parameter $z$ that corresponds to the grading:
\[
f_{\Lie^{\odd}}(t,z):=\sum_{n\geq 0} \frac{\dim\Lie^{\odd}(2n+1) t^{2n+1} z^{n}}{(2n+1)!} = \frac{\arcsin(t\sqrt{z})}{\sqrt{z}}
\]
The isomorphism of graded symmetric collections $\Pois_d^{\odd}\simeq \Com\circ(\s^{1-d}\Lie^{\odd})$ leads to the following presentation of the 
generating series of $\Pois_d^{\odd}$:
\[
f_{\Pois_d^{odd}}(t,z):= f_{\Com}(t)\circ f_{\Lie^{odd}}(t,z) =  \exp\left(\frac{\arcsin(t \sqrt{z})}{\sqrt{z}}\right).
\]
Here the parameter $z$ corresponds to the grading of a ternary operation.
If we multiply the coefficient of $t^{n}$ by $n!$ we obtain the generating series of the space of $n$-ary operations $\Pois_d^{\odd}(n)$
\begin{equation}
\label{eq::Pois::odd::series}
\prod_{1\leq k < \frac{n}{2}} (1+(n-2k)^2 z)\, .
\end{equation}
The homological degree of $z$ is equal to $1-2d$.
\begin{remark}
A little bit more advanced computations with the generating series of symmetric functions given by $S_n$-characters on the space of $n$-ary operations for $\Com$ and $[\Com]^{\ZZ_2}$ leads to the computation of $S_n$ characters of $\Lie^{\odd}(n)$, $\TwoPois_d(n)$. 
The details for the main case $d=1$ can be found in~\cite{Rains_Sn}.
\end{remark}

\subsection{${\mathbb Z}_2$-invariants of the associative operad}
\label{sec::Ass::Inv}
We already discussed the $\ZZ_2$-action and its Koszul dual action on the operad $\Ass$ of associative algebras in~\eqref{eq::tau::Ass} and~\eqref{eq::tau!::Ass} given by flipping the order of the multiplication and changing the sign in the Koszul dual case.
The aim of this section is to give the algebraic presentation of the suboperad of $\tau$-invariants $\AsM$ in terms of generators and relations. The key observation is that $\tau$ preserves filtration by commutators. Consequently, the space of $\tau$-invariants of the associative operad and the space of $\tau$-invariants of the associated graded operad (which coincides with the Poisson operad) have the same dimension.  
\begin{theorem}
\label{thm::AsM::def}
The operad $\AsM$ is generated by 

-- the binary generator $\nu_2:= x_1 x_2-x_2 x_1$,

-- the ternary generator $\mu_3:= \sum_{\sigma\in S_3} x_{\sigma(1)} x_{\sigma(2)} x_{\sigma(3)}$,
subject to the following 
relations\footnote{not all the relations are invariant under the action of symmetric group, so first one has enlarge the space of quadratic relations using permutations of inputs}:
\begin{gather}
\label{eq::As_Jacobi}
\nu_2(\nu_2(x_1,x_2),x_3) + \nu_2(\nu_2(x_2,x_3),x_1) 
+ \nu_2(\nu_2(x_3,x_1),x_2) = 0
\\
\label{eq::As::Leibn}
\nu_2(\mu_3(x_1,x_2,x_3),x_4) = \mu_3(x_1,x_2,\nu_2(x_3,x_4)) + \mu_3(x_1,\nu_2(x_2,x_4),x_3) + \mu_3(\nu_2(x_1,x_4),x_2,x_3)
\\
\label{eq::As_Com}
\begin{array}{c}
{
\mu_3(\mu_3(x_1,x_2,x_4),x_3,x_5) - \mu_3(\mu_3(x_1,x_2,x_3),x_4,x_5) =
B(x_1,x_2,x_3,x_4,x_5) - B(x_1,x_2,x_4,x_3,x_5)
}
\\
\text{ with } 
{
B(a,a,a,b,b):= 
3 \mu_3(\nu_2(a,b),\nu_2(a,b),a)
+6 \mu_3(\nu_2(a,\nu_2(a,b)),a,b) + 
}
\\
{
+\frac{6}{5} \nu_2( b,\nu_2(a,\nu_2(a,\nu_2(a,b))))
+ \frac{18}{5} \nu_2( a,\nu_2(b,\nu_2(a,\nu_2(b,a)))).
}
\end{array}
\end{gather}
These relations form a Gr\"obner basis  
with respect to the degree-path-lexicographic ordering of monomials and convention $\nu_2>\mu_3$.
\end{theorem}
\begin{proof}
Let us explain that the aforementioned relations are indeed satisfied.
First, notice that the commutator 
$\nu_2$ is a Lie bracket in any associative algebra, hence the Jacobi identity \eqref{eq::As_Jacobi} follows.
Second, the Lie bracket is a derivation of the associative multiplication
\[
[x_1x_2,x_3] = x_1x_2 x_3 - x_3 x_1 x_2 = x_1x_2 x_3 -x_1 x_3 x_2 +x_1 x_3 x_2 - x_3 x_1 x_2 = x_1 [x_2,x_3] + [x_1,x_3]x_2,
\]
which implies the Leibniz rule~\eqref{eq::As::Leibn} for the iterated symmetrized multiplication $\mu_3$. 
Note that if we denote by $\mu_2$ the symmetrization of multiplication $x_1 x_2 + x_2 x_1$, then 
\begin{equation}
\label{eq::iterated::multiplication}
\mu_3(x_1,x_2,x_3) = \mu_2(x_1,\mu_2(x_2,x_3)) +  \mu_2(x_2,\mu_2(x_3,x_1)) + \mu_2(x_3,\mu_2(x_1,x_2))
\end{equation}
We have the following relation together with the Leibniz rule:
\begin{equation}
\label{eq::pois::to::assoc}
\mu_2(x_1,\mu_2(x_2,x_3)) - \mu_2(x_2,\mu_2(x_3,x_1)) = 
\nu_2(\nu_2(x_1,x_2),x_3) 
\end{equation}
The remaining relation~\eqref{eq::As_Com} was found and checked first using the computer and we do not provide these computations.
However, one can derive it out of presentation~\eqref{eq::iterated::multiplication} and relation~\eqref{eq::pois::to::assoc}.

In fact the precise form of the right-hand side of relation~\eqref{eq::As_Com} will not be relevant, and that such a relation exists may be shown without direct computations. Indeed, the associative operad $\Ass$ admits the filtration by commutators and the associated graded operad coincides with the Poisson operad $\Pois_1$. 
The automorphism $\tau$ preserves the filtration by commutators and the corresponding automorphism of the associate graded operad coincides with the automorphism $\tau$ considered in the previous Section~\ref{sec::Pois::Z_2}. Consequently, the operad $\AsM$ also admits a filtration by commutators and the associated graded coincides with the Koszul operad $[\Pois_1]^{\ZZ_2}$.
The operations $\mu_3$ and $\nu_2$ generate the associated graded operad and, consequently, their preimages generate the initial operad $\AsM$. As we already mentioned the Jacobi identity for $\nu_2$ as well as the Leibniz identity~\eqref{eq::Pois::Leibn} for generators of $[\Pois_1]^{\ZZ_2}$ remain valid in $\AsM$ (relations~\eqref{eq::As_Jacobi} and ~\eqref{eq::As::Leibn}). However, the right-hand side of relation~\eqref{eq::Pois::Com} has to be replaced by an expression in $\mu_3$ and $\nu_2$ that contains at least one $\nu_2$. Notice, that an operadic monomial of arity $5$ in the binary operation $\nu_2$ and the ternary operation $\mu_3$ should be at least cubic in generators if the degree in $\nu_2$ is greater than zero.  Thus, the quadratic Gr\"obner basis of relations for $[\Pois_1]^{\ZZ_2}$ is replaced by a nonhomogeneous Gr\"obner basis of the relations for $\AsM$ but with the same homogeneous component of degree $2$.
\end{proof}

\begin{corollary}
\label{cor::Ass::Pois}	
	The homology of the cobar construction of the cooperad ${\AsMd}$ dual to the operad $\AsM$ is equal to the operad  $\Pois_1^{\odd}$. However, the cobar construction is not formal and the nonhomogeneous relation~\eqref{eq::As_Com} yields the existence of nontrivial $\infty$-products on $\Pois_1^{\odd}$ that make it equivalent to the mosaic operad $\Mos$.
\end{corollary}	
\begin{proof}
Consider, the Anick-type resolution of a (shuffle) operad (discovered in~\cite{DK::Anick}) whose generators are constructed out of intersections of leading monomials of the given Gr\"obner basis of an operad. This resolution is minimal if the leading monomials are quadratic. Therefore, the space of generators for the minimal resolutions for the operad $\AsM$ and its associated graded are the same:
\[
\Omega((\Pois_1^{\odd})^{\dual}) := \calF(\s(\Pois_1^{\odd})^{\dual}, d) \stackrel{\text{quis}}{\rightarrow} [\Pois_1]^{\ZZ_2} \quad
\calF(\s(\Pois_1^{\odd})^{\dual}, d+ d_{\geq 3}) \stackrel{\text{quis}}{\rightarrow} \AsM
\]
The differential in the Anick-type resolution replaces leading monomials by the corresponding relations in the Gr\"obner basis. Therefore, the differential can be split into part of homogeneity $2$ and the part of homogeneity greater than $2$.
The double bar construction explains that the generators of the minimal resolution coincide with the cohomology of the (co)bar construction. 
Consequently, the cohomology of the cobar constructions of ${\AsMd}$ and of its associated graded are the same vector spaces, meaning that the corresponding spectral sequence degenerates in the first term. 

To see the non-formality, we first note that operads $\AsM$ and $[\Pois_1]^{\ZZ_2}$ are not isomorphic. This is because the generators are singled out uniquely, up to scale, by their symmetries, and they satisfy different relations.
But then the operads $\AsM$ and $[\Pois_1]^{\ZZ_2}$ are also not quasi-isomorphic, since they have no differential, and any quasi-isomorphism would in particular induce an isomorphism on the level of the cohomology operads.
Finally it follows that the cobar construction of $\AsMd$ cannot be formal. 
Otherwise, the bar-cobar construction of $\AsMd$ would be quasi-isomorphic to the bar construction of $\Pois_1^{\odd}$, i.e., to $[\Pois_1]_{\ZZ_2}^\dual$. Of course, bar-cobar construction of $\AsMd$ is also quasi-isomorphic to $\AsMd$, and hence, dualizing, $[\Pois_1]^{\ZZ_2}$ and $\AsM$ were quasi-isomorphic, a contradiction.
\end{proof}

As an easy consequence of Theorem~\ref{thm::Mosaic::chains} and Corollary~\ref{cor::Ass::Pois} we recover Theorem 2.1.4 of~\cite{Etingof_Rains} on the rational cohomology of $\Mos$:
\begin{corollary}
	\label{thm::Mos::Pois}
	\begin{enumerate}
		\setlength\itemsep{-0.4em}		
		\item The $\QQ$-homology of the mosaic operad is equal to $\Pois_1^{\odd}$.
		\item The mosaic operad is not formal.
		\item The Poincar\'e polynomial $\sum_{k\geq 0}\dim H^{k}(\MonR{n+1};\QQ) {t^{k}}$ equals $\prod_{1\leq k < \frac{n}{2}} (1+z(n-2k)^2)$
	\end{enumerate}
\end{corollary}
\begin{proof}
	Items 1 and 2 follow from Corollary~\ref{cor::Ass::Pois} and the Poincar\'e polynomial was computed for $\Pois_d^{\odd}(n)$ in~\eqref{eq::Pois::odd::series}.
\end{proof}
The proof in~\cite{Etingof_Rains} is based on the comparison of the combinatorics of cycles in $\MonR{n+1}$ and the combinatorics of operadic monomials in $\Pois_1^{\odd}$. Our proof is based on the cell decomposition of $\MonR{n+1}$ and the koszulness of the operad $\Pois_1^{\odd}$ and does not really require any deep combinatorics.

\begin{remark}
	In order to cover the information of the mosaic operad over integers and, in particular, the description of the torsion homology groups of $\MonR{n}$ one has to deal with the presentation of the cooperad ${\AsMd}$ over integers in terms of (co)generators and (co)relations. Unfortunately, we do not know a simple description of the latter cooperad over the integers. 
\end{remark}

\section{Hopf models for $\Pois_d$ and $E_d$}
\label{sec::Hopf::E_d}
This section does not contain any new material.
We give a short overview of certain known combinatorial algebraic models of the little discs operad and the Poisson operad considered as Hopf operads in terms of graphs. We refer to e.g.~\cite{Lamb_Vol},\cite{Willwacher_grt},\cite{Willwacher_oriented} for more detailed expositions to these operads. The most complicated model of oriented graphs is upgraded in Section~\ref{sec::Hopf::Mosaic} in order to produce a model of $\TwoPois_d$. 

By a Hopf operad we simply mean an operad in the category of  differential graded cocommutative coalgebras.
The comultiplication on the space of $n$-ary operations of a Hopf operad $\calP$ will be denoted by $\Delta_{\calP}:\calP \to \calP\otimes \calP$.

\subsection{Simplest Hopf models for $\Pois_d$}
\label{sec::Pois::Hopf}
The standard presentation~\eqref{eq::Pois::def} of the operad $\Pois_d$ admits the following very simple formula for the Hopf coproduct: 
\begin{equation}
\label{eq::Hopf::Pois}
\Delta_{\Pois_d}(\mu_2) = \mu_2\otimes \mu_2, \quad 
\Delta_{\Pois_d}(\nu_2) = \mu_2 \otimes \nu_2 + 
\nu_2 \otimes \mu_2 \text{ with } \nu_2,\mu_2\in \Pois_d(2).
\end{equation}
In particular, for all $d\geq 2$ the corresponding spaces of $n$-ary (co)operations of the cooperad 
$$\Pois_d^{\dual}(n):= H^{\udot}(E_d(n)) = H^{\udot}(Conf(n,\RR^{d}))$$ is known to be a quadratic Koszul algebra called Orlik-Solomon algebra (see e.g.~\cite{Arnold},\cite{Orlik}):
\begin{equation}
\label{eq::OS::def}	
OS_d(n):= H^{\udot}(E_d(n);\QQ)\simeq
\QQ\left[
\begin{array}{c}
\nu_{ij}, 1\le i\neq j\le n \\
\nu_{ij} = (-1)^{d}\nu_{ji}
\\
\deg(\nu_{ij})=d-1
\end{array} 
\left|
\begin{array}{c}
\nu_{ij}\nu_{ij},
\\
\nu_{ij}\nu_{jk}
+\nu_{jk}\nu_{ki}
+\nu_{ki}\nu_{ij}
\end{array}
\right.
\right].
\end{equation}
The Lie bracket $\nu_2\in \Pois_d(2)$ generates a Koszul suboperad $\Lie\{1-d\}$ of Lie algebras (with a Lie bracket of degree $(1-d)$). The Koszul dual operad $\Lie\{1-d\}^{!}$ is an operad of shifted commutative algebras and one has a simple resolution $L_{\infty}\{1-d\}$ of $\Lie\{1-d\}$. This replacement defines another simple Hopf operad called $\ho \Pois_d$ that is quasi-isomorphic to $\Pois_d$:
\begin{equation}
\label{eq::hoEd}
\begin{array}{c}
\calF\left(
\begin{array}{c}
\mu_2\in \ho \Pois_d(2) ,\\
 \nu_{k} \in \ho \Pois_d(k), k\geq 2 
\\
\deg(\nu_{k})= 1-(k-1)d
\end{array}
\left| 
\begin{array}{c}
\mu_2(\mu_2(x_1,x_2),x_3) = \mu_2(x_1,\mu_2(x_2,x_3));
\\
\nu_k(\mu_2(x_0,x_1),\ldots) = \mu_2(x_0,\nu_k(x_1,\ldots)) + \mu_2(x_1,\nu_k(x_0,\ldots)); \\ 
d(\nu_{k})=\sum_{i+j=k+1}\sum_{\sigma\in S_{k+1}} (-1)^{\sigma}\sigma\cdot \nu_{i}\circ_1\nu_{j}.
\end{array}
\right.\right)
\\
{\phantom{1}}\\
\Delta(\mu_2) = \mu_2\otimes \mu_2, \quad 
\Delta(\nu_{k}) = \underbrace{\mu_2\circ\mu_2\circ\ldots\circ \mu_2}_{k} \otimes \nu_{k} + 
\nu_{k}\otimes \underbrace{\mu_2\circ\mu_2\circ\ldots\circ \mu_2}_{k}.
\end{array}
\end{equation}

The case $d=1$ is a bit exceptional meaning that the operad $E_1$ is equivalent to the operad $\Ass$ of associative algebras. However, the operad $\ho \Pois_1$ is also well defined and is a model for the operad $\Pois_1$ of Poisson algebras.
There exist a standard filtration by the number of commutators on the operad of associative algebras whose associated graded is $\Pois_1$. Respectively, the operad $\Ass$ is considered as a deformation of the operad $\Pois_1$ and we will specify a particular element in the deformation complex of $\Pois_1$ that is responsible for this deformation in Section~\ref{sec::Graphs::orient}.

\subsection{Fulton-McPherson operad $FM_d$}
\label{sec::FM}
In~\cite{FM} Fulton and McPherson defined a compactification of the space of configurations of points in a variety $X$ given as a consecutive composition of blowups of all the diagonals in $X^{n}$. For $X=\RR^{d}$ one has to consider real (spherical) blowups (bordifications), as has been done by Axelrod--Singer \cite{AxelrodSinger}. It is well known that the corresponding compactifications $\FM_{d}(n)$ together with their natural stratifications define manifolds with corners which assemble into a model of the little $d$-dimensional discs operad $E_d$. 
In particular, $\FM_d(2)\simeq S^{d-1}$ and the forgetful maps $\FM_d(n)\to \FM_d(2)$ forgetting all but $i$-th and $j$-th points is denoted by $\pi_{ij}$.

These operads were used by Kontsevich to obtain fibrant differential graded commutative algebra ({\bf dgca} for short) models of $\Pois_d$ (respectively $\FM_d$) given by a certain class of differential forms on $\FM_d$. The drawback of the dgca models $\ho\Pois_d$ and $\Pois_d$ of the operad $\Pois_d$ presented in the previous section~\ref{sec::Pois::Hopf} is that they are not fibrant, (essentially) due to the fact that the corresponding dg coalgebras are not cofree.

\subsection{Kontsevich's operads $\Graphs_d$ of graphs}
\label{sec::Graphs::E_d}
\subsubsection{Definition of the (co)operad}
In~\cite{Kont_Motives} M.\,Kontsevich defined the combinatorial model $\Graphs_d$ for the little cubes operad $E_d$, that he used to show the (real) formality of $E_d$.
Many details of this construction have later been worked out and clarified by Lambrechts and Volic~\cite{Lamb_Vol}.
Since we need a variant of the construction below, we shall recall the main ideas here, in the language and version we need. On the other hand, we will leave some technical details and a more comprehensive exposition to~\cite{Lamb_Vol}.

An admissible graph is a finite undirected graph with two sorts of vertices, external and internal, such that the following holds:
\begin{itemize}
		\setlength{\itemsep}{0.3em}
    \item The external vertices are numbered $1,\dots,n$. The internal vertices are not numbered.
    \item Each internal vertex has valence at least 3.
    \item Each connected component contains at least one external vertex.
\end{itemize}
Here is an example for $n=5$
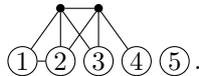
\[
\begin{tikzpicture}
\node[ext] (v1) at (0,0) {1};
\node[ext] (v2) at (0.5,0) {2};
\node[ext] (v3) at (1,0) {3};
\node[ext] (v4) at (1.5,0) {4};
\node[ext] (v5) at (2,0) {5};
\node[int] (w1) at (0.5,.7) {};
\node[int] (w2) at (1.0,.7) {};
\draw (v1) edge (v2) edge (w1) (v2) edge (w1) edge (w2) (v3) edge (w1) edge (w2) (v4) edge (w2) (w1) edge (w2);
\end{tikzpicture}\, .
\]
The cohomological degree of an internal vertex is $-d$ and the cohomological degree of an edge is $d-1$, so that a graph with $k$ internal vertices and $e$ edges has degree $e(d-1)-kd$.

One defines an orientation of an admissible graph to be the following data, depending on the parity of $d$:
\begin{itemize}
		\setlength{\itemsep}{0.3em}
    \item If $d$ is even an orientation is an ordering of the edges up to even permutations.
    \item If $d$ is odd an orientation is an ordering of the internal vertices and half-edges, up to even permutations.
\end{itemize}
The space $\Graphs_d^{\dual}(n)$ is defined to be the space of linear combinations of pairs $(\Gamma,o)$, with $\Gamma$ an admissible graph with $n$ external vertices and $o$ an orientation, modulo the following relations:
\begin{itemize}
		\setlength{\itemsep}{0.3em}
    \item For $\phi : \Gamma\to \Gamma'$ an isomorphism of admissible graphs one identifies $(\Gamma,o)\sim (\Gamma',\phi(o))$, with $\phi(o)$ the orientation naturally induced on $\Gamma'$ via $\phi$.
    \item We identify $(\Gamma,o)\sim -(\Gamma, -o)$, where $-o$ is the opposite orientation. (Note that there only two orientations.)
\end{itemize}
In the following we will (ab)use the notation and speak of generators of $\Graphs_d^{\dual}(n)$ as graphs, assuming implicitly that the graph is accompanied by an orientation. 

The spaces $\Graphs_d^{\dual}(n)$ assemble into a Hopf cooperad, i.e., a cooperad object in the category of dg commutative algebras.
Combinatorially the cocomposition is given by subgraph contraction and the differential is given by edge contraction. The commutative multiplication is defined by fusing two graphs at the external vertices, e.g.,
  \begin{equation}
  \label{equ:ext gluing}
\begin{tikzpicture}[baseline=-.65ex]
\node[ext] (v1) at (0,0) {1};
\node[ext] (v2) at (0.5,0) {2};
\node[ext] (v3) at (1,0) {3};
\node[ext] (v4) at (1.5,0) {4};
\node[int] (w1) at (0.5,.7) {};
\draw (v1) edge (v2) edge (w1) (v2) edge (w1) (v3) edge (w1) ;
\end{tikzpicture}
\wedge
\begin{tikzpicture}[baseline=-.65ex]
\node[ext] (v1) at (0,0) {1};
\node[ext] (v2) at (0.5,0) {2};
\node[ext] (v3) at (1,0) {3};
\node[ext] (v4) at (1.5,0) {4};
\node[int] (w2) at (1.0,.7) {};
\draw   (v2)  edge (w2) (v3)  edge (w2) (v4) edge (w2);
\end{tikzpicture}
=
\begin{tikzpicture}[baseline=-.65ex]
\node[ext] (v1) at (0,0) {1};
\node[ext] (v2) at (0.5,0) {2};
\node[ext] (v3) at (1,0) {3};
\node[ext] (v4) at (1.5,0) {4};
\node[int] (w1) at (0.5,.7) {};
\node[int] (w2) at (1.0,.7) {};
\draw (v1) edge (v2) edge (w1) (v2) edge (w1) edge (w2) (v3) edge (w1) edge (w2) (v4) edge (w2);
\end{tikzpicture}
\end{equation}

We also consider the operad $\Graphs_d$ dual to $\Graphs_d^\dual$. (Mind that we slightly abuse the notation ${}^\dual$ here.)
Concretely, elements of $\Graphs_d^\dual$ are finite linear combinations of graphs, while elements of $\Graphs_d$ can be identified with series of graphs.
The operadic composition in the operad $\Graphs_d$  is given by  insertion  at external vertices, and the differential is vertex splitting. 

The dg operad $\Graphs_d$ is generally not a dg Hopf operad because of completion issues. It is however a complete dg Hopf operad, in the sense that to define the cocomposition one uses a completed tensor product. (The completion issue is furthermore inexistent if $d\geq 3$, since then it is of finite type in each arity.)

Also note that the space $\Graphs_d^{\dual}(n)$ is isomorphic to the free graded commutative algebra generated by the space of internally connected graphs, meaning that the graphs are connected if one erases all external vertices.
The dual of the span of internally connected graphs defines an operad in the category of $\LL_\infty$-algebras denoted $\ICG_d(n)$ (see~\cite{Severa_Willwacher} for detailed combinatorial description of the $\LL_{\infty}$-structure on $\ICG_d(n)$). 

\subsubsection{The Kontsevich graph complex $\GC_d$ and action}
Note that we required above that our graphs have at least one external vertex.
We may similarly consider the graded vector space $\GC_d^\dual$ of linear combinations of connected graphs with $\geq 3$-valent vertices, together with an orientation datum. Comparing to the definition of $\Graphs_d^\dual$, think of all vertices as internal vertices. Generally, we call a graph with only one sort of vertices admissible if it is $\geq 3$-valent and connected. 
\[
\begin{tikzpicture}
    \node[int] (v) at (0,0) {};
    \node[int] (v1) at (0:.5) {};
    \node[int] (v2) at (120:.5) {};
    \node[int] (v3) at (-120:.5) {};
    \draw (v) edge (v1) edge (v2) edge (v3) 
    (v1) edge (v2) edge (v3) 
    (v2) edge (v3);
\end{tikzpicture}
\]
We define the cohomological degree of such a graph with $v$ vertices and $e$ edges as 
\[
(d-1)e-d (v-1).
\]
The graded vector space $\GC_d^\dual$ carries the structure of a dg Lie coalgebra, with the differential given by edge contraction, and the Lie cobracket by subgraph contraction.

The dg Lie algebra $\GC_d$ dual to $\GC_d^\dual$ is the well known Kontsevich graph complex with the vertex splitting differential.
The important fact for us is that there is a dg Lie coaction of $\GC_d^\dual$ (and hence also a dg Lie action of $\GC_d$) on the Hopf cooperad $\Graphs_d^\dual$.
This was first described in \cite{Willwacher_grt}, with a more general and detailed account in \cite{Dolgu_Will}.
Leaving signs and a more careful discussion to loc. cit., the coaction is a morphism 
\[
\Graphs_d^\dual(n) \to \GC_d^\dual \otimes \Graphs_d^\dual(n) 
\]
that assigns to a graph $\Gamma$ the linear combination
\begin{equation}\label{equ:G coaction}
\underbrace{\sum_{\gamma\subset \Gamma\atop \text{no ext.} } \pm \gamma\otimes (\Gamma/\gamma)}_{(I)}
+
\underbrace{
\sum_{\gamma\subset \Gamma
\atop \text{one ext.}} \pm \gamma_{int} \otimes \Gamma
}_{(II)}
+
\underbrace{
\sum_{\Gamma'\subset \Gamma
\atop \text{all ext.} } \pm (\Gamma/\Gamma') \otimes \Gamma'
}_{(III)}.
\end{equation}
The first sum runs over connected subgraphs $\gamma$ with only internal vertices, and $(\Gamma/\gamma)$ is obtained by contracting the subgraph $\gamma$.
The second sum runs over subgraphs $\gamma$ with exactly one internal vertex. Such a $\gamma$ is considered a generator of $\GC_d^\dual$ by declaring the external vertex internal, and we write $\gamma_{int}$ for the graph thus obtained.
Finally, the last sum is over all subgraphs that contain all external vertices.

Next, let $\gamma\in \GC_d$ be a Maurer-Cartan element, i.e., $\delta\gamma+\frac12[\gamma, \gamma]=0$.
Then we may twist the dg Hopf cooperad $\Graphs_d^\dual$ by $\Gamma$. The twisted dg Hopf cooperad $(\Graphs_d^{\dual})^\gamma$ has the same graded Hopf cooperad structure, but differential $d+\gamma\cdot$ obtained by adding a term that is the action of $\gamma$.

\subsubsection{Quasi-isomorphism with differential forms on $\FM_d$}
For technical reasons the Kontsevich morphism requires the use semi-algebraic differential forms $\Omega_{triv}(-)$, respectively $\Omega_{PA}(-)$.
We refer to \cite{HLTV} for the definition of semi-algebraic differential forms and \cite[section 4]{Lamb_Vol} for a discussion in the context of the formality of the little disks operads.
For this paper we will treat the semi-algebraic differential forms as a "blackbox" -- the precise technical definition will not be important.

Now choose a unit volume top form $\omega\in \Omega^{d-1}_{triv}(\FM_d(2))$, that we call the propagator.
We require that for $\tau:\FM_d(2)\to \FM_d(2)$ the transposition of the two points we have 
\begin{equation}\label{equ:prop symm}
\tau^* \omega =(-1)^d \omega.
\end{equation}

We assign to each graph $\Gamma\in \Graphs_d^\dual(n)$ with $n$ external and $k$ internal vertices the differential form
given by the product of pullbacks of propagators: 
\[\tilde{\omega}_\Gamma:=\bigwedge_{e\in \edges(\Gamma)} \pi_{in(e) out(e)}^{*}\omega \ \in  \ \Omega_{PA}^{(d-1)(\# edges)} (\FM_d(n+k)) \]
the direct image of this differential form with respect to the map $\FM_d(n+k)\to \FM_d(n)$ of forgetting the points defines a map:
\begin{equation}
\label{equ:omega Gamma}
\begin{array}{rcl}
{\Gamma} & \mapsto & \omega_{\Gamma}:=\int_{\FM_d(n+k)\to \FM_d(n)}\tilde{\omega}_{\Gamma} 
\\
\Graphs^{\dual}_{d}(n) & \longrightarrow & \Omega_{PA}(\FM_d(n)).
\end{array}
\end{equation}
The same construction (with $n=0$) applied to elements of $\GC_d^\dual$ defines a map 
\[
\GC_d^\dual \to \mathbb R[-1],
\]
or equivalently a degree 1 element $\gamma_\omega\in \GC_d$.
Concretely, this element may be written as 
\[
\gamma_\omega = \sum_{\Gamma} c_\Gamma \Gamma^*
\]
with $\Gamma$ running over a basis of $\GC_d^\dual$,  $\Gamma^*\in \GC_d$ being the dual basis elements, and the weights $c_\Gamma$ being defined by  
\begin{equation}
\label{eq::Konts::vanish}
c_{\Gamma}:=\int_{\FM_d(n)} \tilde{\omega}_{\Gamma}.
\end{equation}

Now we may rephrase the main results of \cite{Kont_Motives} and \cite{Lamb_Vol} as follows.\footnote{This is a slight reformulation, as the MC element is not explicitly mentioned in loc. cit.}
\begin{theorem}[Kontsevich, Lambrechts-Volic]
\label{thm:LambVol}
Let $d\geq 2$.
\begin{enumerate}
    \item 
    The element $\gamma_\omega\in \GC_d$ is a Maurer-Cartan element.
    \item The assignment $\Gamma\mapsto \omega_\Gamma$ above yields quasi-isomorphisms of dg commutative algebras 
    \[
    (\Graphs^{\dual}_{d})^{\gamma_\omega} (n)  \to  \Omega_{PA}(\FM_d(n)),
    \]
    that are furthermore compatible with the cooperad structure.
    \item In the special cases of $d\geq 3$, or if $d=2$ and $\omega$ is the round (i.e., rotation invariant) volume form we have $\gamma_\omega=0$, and the twist by $\gamma_\omega$ above can be omitted.  
    \end{enumerate}
\end{theorem}
\begin{proof}[Sketch of proof]
For (1) we apply Stokes' Theorem to the integrals defining the numbers $c_\Gamma$ above:
\[
0 = \int_{\FM_d(n)} d\tilde{\omega}_{\Gamma}
=
\int_{\partial \FM_d(n)} \tilde{\omega}_{\Gamma}
\]
Now the (codimension one) boundary strata of $\FM_d(n)$ correspond to subsets $S$ of the $n$ points in the configuration approaching a point in $\mathbb R^d$.
\[
\begin{tikzpicture}
    \draw[dashed] (0,0) circle (.3cm);
    \node at (.4,.4) {$S$};
    \node[int] at (0,.2) {};
    \node[int] at (.2,-.1) {};
    \node[int] at (-.1,-.2) {};
    \node[int] at (-1,1) {};
    \node[int] at (-1,-1) {};
    \node[int] at (1.3,0) {};
\end{tikzpicture}
\]
The integral along such a boundary stratum factorizes by Fubini's Theorem and hence we have that
\begin{equation}\label{equ:MC on int}
0
=
\int_{\partial \FM_d(n)} \tilde{\omega}_{\Gamma}
=
\sum_{\graph\subset \Gamma }
\left(
\int_{\partial \FM_d(n')} \tilde{\omega}_{\graph}
\right)
\left(
\int_{\partial \FM_d(n'')} \tilde{\omega}_{\Gamma/\graph}
\right)
=
\sum_{\graph\subset \Gamma }
c_{\graph} c_{\Gamma/\graph}.
\end{equation}
Here $\graph$ is the full subgraph of $\Gamma$ on the subset $S$ of the vertex set corresponding to the collapsing points on one boundary stratum. Note that $\graph$ might not be admissible, i.e., it might be disconnected and/or have vertices of valence $\leq 3$ -- still we are using the notation $\tilde \omega_\graph$ and $c_\graph$, for which admissibility plays no role.
If $\graph=
\begin{tikzpicture}
\node[int] (v) at (0,0) {};
\node[int] (w) at (0.5,0) {};
\draw (v) edge (w);
\end{tikzpicture}$ is a graph with two vertices and one edge then $c_\graph=1$ by the assumption that the propagator $\omega$ is of unit volume.
It is part of technical verifications of Kontsevich and Lambrechts-Volic to check that for all other non-admissible $\graph$ the integral along the respective boundary stratum vanishes, i.e., $c_\graph=0$, see \cite[section 9.3]{Lamb_Vol}.
We just recall here the verification that the integral vanishes if $\graph$ has a bivalent vertex. It is sufficient to consider the integral over the point $z\in \mathbb R^d$ in the configuration corresponding to the bivalent vertex.
\[
\begin{tikzpicture}
    \node[ext,label=90:{$x$}] (v) at (0,0) {}; 
    \node[ext,label=90:{$y$}] (x) at (1.4,0) {}; 
    \node[int,label=90:{$z$}] (z) at (.7,0) {};
    \draw (v) edge (z) (z) edge (x);
\end{tikzpicture}
\]
Let us call the points corresponding to the adjacent vertices $x$ and $y$. Then the part of the integral over $w$ reads 
\[
I_2(x,y):=\int_{z\in \mathbb R^d}\omega\left(\frac{x-z}{\|x-z\|}\right) \omega\left(\frac{z-y}{\|z-y\|}\right).
\]
Performing the change of variables $z\mapsto x+y-z$ yields 
\begin{equation}\label{equ:I2 vanishing}
I_2(x,y) =
(-1)^d
\int_{z\in \mathbb R^d}\omega\left(\frac{z-y}{\|z-y\|}\right) \omega\left(\frac{x-z}{\|x-z\|}\right)
=
-I_2(x,y),
\end{equation}
hence $I_2(x,y)=0$ as desired.
Here the first sign $(-1)^d$ is due to our change of variables changing the orientation and the second sign $(-1)^{d-1}$ is due to commuting the two $d-1$-forms.

Having the vanishing in the non-admissible case, the above equation \eqref{equ:MC on int} is the equivalent on the coefficients of the Maurer-Cartan equation for $\gamma_\omega$.
The differential term $\delta\gamma_\omega$ corresponds to those summands for which $\graph$ is the graph with two vertices connected by one edge.

Statement (2) of the theorem consists of three separate statements:
\begin{enumerate}
    \item[(2a)] The morphism respects the differential.
    \item[(2b)] The morphism respects the commutative algebra structure and cooperad structure.
    \item[(2c)] The morphism is a quasi-isomorphism.
\end{enumerate}

To show (2a) one again applies Stokes Theorem to the fiber integrals defining $\omega_\Gamma$:
\[
0 = 
\int_{\FM_d(n+k)\to \FM_d(n)}\tilde{\omega}_{\Gamma} 
=
\int_{\partial_{fiber}\FM_d(n+k)}\tilde{\omega}_{\Gamma}.
\]
On the right-hand side one has to integrate over the fiberwise boundary.
Now there are three different types of codimension 1 fiberwise boundary strata:
\begin{enumerate}
    \item[(S1)] Some subset of the $k$ points corresponding to internal vertices of $\Gamma$ come close to each other.
    \item[(S2)] Some subset of the $k$ points comes close to one of the $n$ points corresponding to the external vertices of $\Gamma$.
    \item[(S3)] Some subset of the $k$ points escapes to $\infty$.
\end{enumerate}
In each case, the integral over the respective stratum factorizes into a product of an integral of the form $c_\graph$ and one of the form $\omega_{\Gamma'}$.
The terms arising from the strata of type S1 then can be seen to contribute the terms (I) of \eqref{equ:G coaction}, those from (S2) the terms (II), and those from (S3) the terms (III).
Here again one needs to check that the integrals vanish in case the collapsing subgraph $\gamma$ is not admissible, but this is done as before.
We conclude that \eqref{equ:omega Gamma} is indeed a morphism of cochain complexes.

The compatibility of \eqref{equ:omega Gamma} with the commutative algebra and cooperadic structures (i.e., (2b) above) follows directly from the definition.
Hence to show (2) one is left with showing that \eqref{equ:omega Gamma} induces an isomorphism on cohomology.
This can be done by explicitly computing the cohomology of the left-hand side to be 
\[
H\left( (\Graphs_d^{\dual})^{\gamma_\omega} \right)\cong H(\FM_d),
\]
and then checking that \eqref{equ:omega Gamma} sends (commutative algebra) generators of the cohomology to generators.

The proof of the third part of the Theorem depends on the dimension $d$. For $d\geq 3$ it follows merely by dimension counting that $c_\Gamma=0$.
For $d=2$ there is a more intricate argument that shows the statement, see \cite[Theorem 6.6.1]{KDef}.
\end{proof}

\subsection{Operads $\dGraphs_d$ of directed graphs}
\label{sec::dGraphs}
The construction of Kontsevich and Lambrechts-Volic outlined in the previous subsection allows for a minor variation, using directed graphs instead of undirected graphs in all constructions.
For technical reasons we also allow bivalent vertices in our graphs.

Concretely, we call a directed graph with external and internal vertices admissible, if it satisfies the following conditions.
\begin{itemize}
    \item The external vertices are numbered $1,\dots,n$. The internal vertices are not numbered.
    \item Each internal vertex has valence at least 2, and if it has valence 2 it has either two incoming or two outgoing edges attached.
    \item Each connected component contains at least one external vertex.
\end{itemize}
In other words, we do not allow "passing vertices" 
$\begin{tikzpicture}
    \node (v) at (0,0) {}; 
    \node (x) at (1.4,0) {}; 
    \node[int](w) at (.7,0) {};
    \draw (v) edge [->] (w) (w) edge[->] (x);
\end{tikzpicture}$
but other bivalent vertices are allowed.

Here is an example admissible graph with $n=5$ external vertices.
\[
\begin{tikzpicture}[every edge/.style={draw, ->}]
\node[ext] (v1) at (0,0) {1};
\node[ext] (v2) at (0.5,0) {2};
\node[ext] (v3) at (1,0) {3};
\node[ext] (v4) at (1.5,0) {4};
\node[ext] (v5) at (2,0) {5};
\node[int] (w1) at (0.5,.7) {};
\node[int] (w2) at (1.0,.7) {};
\node[int] (w3) at (1.75,.7) {};
\draw (v1) edge (v2) (v2) edge (w1)  (w1) edge (w2)
(w3) edge (v4) edge (v5)
(w1) edge (v1) edge (v3)
(w2) edge (v2) edge (v3) edge (v4);
\end{tikzpicture}\, .
\]
An orientation datum for admissible directed graphs is the same as for undirected graphs. We define $\dfGraphs_d^\dual(n)$ to be the space of linear combinations of pairs $(\Gamma,o)$ consisting of a directed graph and orientation datum, modulo the same relations (graph isomorphism and orientation change)  as in the definition of $\Graphs_d^\dual(n)$ above.

Again $\Graphs_d^\dual$ is a dg Hopf cooperad. The commutative algebra structure is once again given by the graph gluing through external vertices as in \eqref{equ:ext gluing}. Similarly, let $\dGC_d^\dual$ be the directed graphs version of $\GC_d^\dual$, whose elements are linear combinations of connected directed graphs with $\geq 2$-valent vertices, but no passing vertices. As above, $\dGC_d^\dual$ is a dg Lie coalgebra, coacting on the dg Hopf cooperad $\dfGraphs_d^\dual$. Hence the dual dg Lie algebra $\dGC_d$ to $\dGC_d^\dual$ acts on $\dfGraphs_d^\dual$.

The definition of the map $\omega_\Gamma$ \eqref{equ:omega Gamma} and the Maurer-Cartan element $\gamma_\omega$ readily extend to directed graphs, using the same formulas. The only difference is that we do not need to require that our propagator $\omega$ satisfies the symmetry condition \eqref{equ:prop symm}. Fix such a (possibly non-symmetric) propagator $\omega$. 
From the construction we then obtain the Maurer-Cartan element 
\[
\gamma_\omega = \sum_{\Gamma} c_\Gamma \Gamma \in \dGC_d
\]
with 
\[
c_{\Gamma}=\int_{\FM_d(n)} \tilde{\omega}_{\Gamma}.
\]
Furthermore, we obtain maps of dg commutative algebras 
\begin{equation}
    \label{equ:omega Gamma dir}
\begin{gathered}
(\dfGraphs_d^{\dual})^{\gamma_\omega}(n)
\to 
\Omega_{PA}(\FM_d(n)) \\
\Gamma \mapsto \omega_\Gamma =\int_{\FM_d(n+k)\to \FM_d(n)}
\tilde{\omega}_{\Gamma}
\end{gathered}
\end{equation}
with 
\[
\tilde{\omega}_{\Gamma}= \bigwedge_{e\in \edges(\Gamma)} \pi_{in(e) out(e)}^{*}\omega
\]
The proof that $\gamma_\omega\in\dGC_d$ is a Maurer-Cartan element and that the map \eqref{equ:omega Gamma dir} respects the differential and also the dg Hopf cooperad structure is identical to the proof of Theorem \ref{thm:LambVol} above.
In particular, the vanishing argument \eqref{equ:I2 vanishing} continues to hold in the directed setting. However, it only shows that integrals of diagrams with bivalent passing vertices vanish, but it does not show that the integrals for diagrams with other bivalent vertices also vanish:
\begin{align*}
\begin{tikzpicture}
    \node (v) at (0,0) {}; 
    \node (x) at (1.4,0) {}; 
    \node[int](w) at (.7,0) {};
    \draw (v) edge [->] (w) (w) edge[->] (x);
    \draw[red] (-0.2,0.3) -- (1.5,-0.3);
    \draw[red] (-0.2,-0.3) -- (1.5,0.3);
\end{tikzpicture}
&
&
\begin{tikzpicture}
    \node (v) at (0,0) {}; 
    \node (x) at (1.4,0) {}; 
    \node[int](w) at (.7,0) {};
    \draw (v) edge [->] (w) (w) edge[<-] (x);
\end{tikzpicture}
&
&
\begin{tikzpicture}
    \node (v) at (0,0) {}; 
    \node (x) at (1.4,0) {}; 
    \node[int](w) at (.7,0) {};
    \draw (v) edge [<-] (w) (w) edge[->] (x);
\end{tikzpicture}
\end{align*}
This is the reason why we need to allow those vertices in the definition of admissible directed graphs.
In fact one can also check that \eqref{equ:omega Gamma dir} remains a quasi-isomorphism, but we will not use or further discuss this point in this paper.

Let $\dfGraphs_d$ be the dual dg Hopf operad to $\dfGraphs_d^\dual$. 
 The space $\dGraphs_d(n)$ is a (complete) cofree cocommutative coalgebra cogenerated by internally connected directed graphs with $n$ external vertices.
We denote the corresponding operad in the category of $\LL_{\infty}$-algebras by $\dICG_d$.

\subsection{Shoikhet cocycle}
\label{sec::Shoikhet}
We note that if $\Gamma, \Gamma'\in \dGC_d$ are two graphs that do not have directed cycles, then the differentials $\delta\Gamma$, $\delta\Gamma'$ and the Lie bracket $[\Gamma, \Gamma']$ are linear combinations of graphs that have no directed cycles either.
This means that the graded subspace 
\[
\dGC_d^\orient \subset \dGC_d
\]
of series of graphs that do not not have directed cycles forms a dg Lie subalgebra, that we call the oriented graph complex. The definition goes back to S. Merkulov, a detailed discussion of the object $\dGC_d^\orient$ can be found in \cite{Willwacher_oriented}.

We now specialize to $d=2$.
Consider a unit volume propagator $\omega^{\orient}\in \Omega^{1}_{triv}(S^{1})$ whose support belongs to the upper 
semi-circle. In other words $\omega^{\orient}$ is zero at the point $((x_1,y_1), (x_2,y_2))\in \FM_2(2)$ if $y_1\leq y_2$.  

Suppose that $\Gamma\in \dGC_2^\dual$ is a directed graph that has a directed cycle. Then it is clear that the weight $c_{\Gamma}^{\orient}:=c_\Gamma$ computed with the propagator $\omega^{\orient}$ vanishes, since the integrand is zero on the integration domain.
Hence we find that the Maurer-Cartan element 
\begin{equation}
\label{eq::MC::Shoikhet}
\gamma^{\orient}:= \sum_{\Gamma\in \text{directed graphs}} c_{\Gamma}^{\orient}\cdot \Gamma \ \in \dGC_2^\orient, \text{ with } c_{\Gamma}^{\orient}:=\int_{\FM_2(n)} \tilde \omega^{\orient}_{\Gamma}
\end{equation}
resides in the dg Lie subalgebra $\dGC_d^\orient \subset \dGC_d$.
We call this element the Shoikhet Maurer-Cartan element after \cite{Shoikhet}.

We do not know how to compute all weights $c_{\Gamma}^{\orient}$. However, one can choose the propagator $\omega^{\orient}\in \Omega_{triv}(S^1)$ anti-symmetric under reflection at the $y$-axis:
\begin{equation}
\label{eq::orient::symmetry}
\omega^{\orient}(-\bar{z}) = \omega^{\orient}(\pi-\varphi) = -\omega^{\orient}(\varphi) = -\omega^{\orient}(z).
\end{equation}
The nonzero summands of $\gamma^{\orient}\in\dGC_2$ of the smallest loop order contain graphs with $4$ vertices and look as follows:
\begin{equation}
\label{eq::SHoikhet}
\gamma_0^{\orient}:=
\begin{tikzpicture}[baseline={(current bounding box.center)},every edge/.style={draw, -triangle 60},scale=0.7]
\node[int] (v1) at (-1,1.5) {};
\node[int] (v2) at (-1,-0.5) {};
\node[int] (v4) at (-2,0.5) {};
\node[int] (v3) at (0,0.5) {};
\draw  (v1) edge (v2);
\draw  (v3) edge (v2);
\draw  (v3) edge (v1);
\draw  (v4) edge (v1);
\draw  (v4) edge (v2);
\end{tikzpicture}
\pm
\begin{tikzpicture}[baseline={(current bounding box.center)},every edge/.style={draw, -triangle 60},scale=0.7]
\node[int] (v1) at (-1,1.5) {};
\node[int] (v2) at (-1,-0.5) {};
\node[int] (v4) at (-2,0.5) {};
\node[int] (v3) at (0,0.5) {};
\draw  (v1) edge (v2);
\draw  (v2) edge (v3);
\draw  (v1) edge (v3);
\draw  (v1) edge (v4);
\draw  (v2) edge (v4);
\end{tikzpicture}
\pm 2 \,
\begin{tikzpicture}[baseline={(current bounding box.center)},every edge/.style={draw, -triangle 60 },scale=0.7]
\node[int] (v1) at (-1,1.5) {};
\node[int] (v2) at (-1,-0.5) {};
\node[int] (v4) at (-2,0.5) {};
\node[int] (v3) at (0,0.5) {};
\draw  (v1) edge (v2);
\draw  (v1) edge (v3);
\draw  (v2) edge (v3);
\draw  (v4) edge (v1);
\draw  (v4) edge (v2);
\end{tikzpicture}
\end{equation}

The additional symmetry properties imply certain vanishing of weights $c_{\Gamma}^{\orient}$ discovered by Shoikhet:
\begin{lemma}(\cite{Shoikhet} Lemma 2.2)
	\label{lem:SMCevenlooporder}
	If a directed graph $\Gamma$ contains a directed loop or its Euler characteristic is even then the corresponding integral $c_{\Gamma}^{\orient}:=\int_{\FM_2(n)} \omega^{\orient}_{\Gamma}=0$. 
	In other words, the Shoikhet element $\gamma^{\orient}$ is concentrated in even loop orders.
\end{lemma}
\begin{proof}
	If the graph $\Gamma$ with $n$ vertices contains a directed loop then the corresponding form ${\omega}^{\orient}_\Gamma\in\Omega(\FM_2(n))$ is already zero before integration.

	Consider the reflection of the plane $\CC=\RR^2$ around the $y$-axis
	\[\tau: z=x+iy\mapsto -\bar{z}=-x+iy.\]
	Denote by the same letter the corresponding reflection of a configuration of $n$ points in $\RR^2$.
	Since points in $\FM_2(n)$  are considered modulo joint translations, thus we may suppose that the first point of a configuration coincides with the origin. Therefore, the determinant of the Jacobian of $\tau$ on $\FM_2$ is equal to $(-1)^{n-1}$. On the other hand, the symmetry~\eqref{eq::orient::symmetry} of the propagator implies that for each graph $\Gamma$ with $n$ vertices we have
	\begin{multline*}
	(-1)^{\# \edges(\Gamma)} 
	c^{\orient}_\Gamma = (-1)^{\# \edges}   \int_{FM_2(n)}\bigwedge_{e\in \edges(\Gamma)}\omega^{\orient}_e =
	\int_{FM_2(n)}\bigwedge_{e\in \edges(\Gamma)}\tau(\omega^{\orient}_e)  = \\
	\int_{\tau(FM_2(n))}\bigwedge_{e\in \edges(\Gamma)}\omega^{\orient}_e
	= (-1)^{n-1} \int_{FM_2(n)}\bigwedge_{e\in \edges(\Gamma)}\omega^{\orient}_e = (-1)^{\# \vertices(\Gamma)-1} c^{\orient}_\Gamma
	\end{multline*}
	Consequently, if the number of edges and the number of vertices of a connected directed graph $\Gamma$ are of the same parity then $c^{\orient}_\Gamma=0$.
\end{proof}

\begin{remark}
	The MC element $\gamma^{\orient}$ is gauge trivial as an element of $\dGC_d$ since there exists a family of interpolating 
	propagators between the round propagator (yielding Kontsevich's vanishing property $c_\Gamma=0$) and the directed one $\omega^{\orient}$.
 It is however not gauge trivial as an element of the smaller dg Lie algebra $\dGC_2^\orient$, see \cite{Willwacher_oriented}.
\end{remark}

\subsection{Merkulov's operads $\Graphs_d^{\orient}$ of oriented graphs}
\label{sec::Graphs::orient}
Suppose that $d\geq 2$.
The Hopf cooperad $\dGraphs_d^\dual$ has a dg Hopf sub-cooperad 
\[
\dGraphs_d^{\nosource,\dual}\subset \dGraphs_d^\dual
\]
that is spanned by graphs such that each internal vertex has at least one outgoing edge. Indeed, it is clear that this subspace is closed under the differential, commutative product and cooperad structure since neither of these operations can create an internal vertex without outgoing edge if there was none before.

Next, following \cite{Willwacher_oriented} we define the quotient dg Hopf cooperad 
\[
\dGraphs_d^{\orient,\dual} := \dGraphs_d^{\nosource,\dual} /I
\]
by setting to zero all graphs that have either (i) a directed cycle or (ii) an external vertex with an outgoing edge.
(In the formula we denote by $I$ all linear combinations of such graphs.)
One easily checks that this quotient construction yields a well-defined dg Hopf cooperad: $I$ is clearly an ideal with respect to the commutative product and cooperad structure. We also have $dI\subset I$ by the following argument. Suppose $\Gamma$ is a graph with an external vertex $j$ with an outgoing edge $e$. Then contracting the edge $e$ might produce a graph not in $I$. However, since the internal vertex $v$ at the other end of the edge $e$ has at least one outgoing edge, and $e$ is incoming to $v$, the graph $\Gamma/e$ still has an outgoing edge at $j$ and is hence in $I$.

Furthermore, the action of $\dGC_d$ on $\dGraphs_d^\dual$ descends to an action of $\dGC_d^\orient$ on $\dGraphs_d^{\orient,\dual}$.
Restricting to $d=2$, we claim that also the morphism \eqref{equ:omega Gamma dir} descends:

\begin{lemma}\label{lem:omega orient}
    The compositions of the morphisms \eqref{equ:omega Gamma dir} with the pullback map $\iota^*:\Omega_{PA}(\FM_2)\to \Omega_{PA}(\FM_1)$ induced by the inclusion $\iota: \FM_1\to \FM_2$ descend to morphisms of dg commutative algebras 
    \begin{equation}\label{equ:omega orient}
    \begin{gathered}
    (\dGraphs_d^{\orient,\dual})^{ \gamma^\orient}(n) \to \Omega_{PA}(\FM_1(n))
    \\
    \Gamma \mapsto \omega_\Gamma^\orient :=
    \int_{\FM_2(k,n) \to \FM_1(n)} \tilde \omega^\orient_\Gamma
    \end{gathered}
    \end{equation}
   that are compatible with the cooperadic structure.
   Here $\FM_2(k,n) \to \FM_1(n)$ is the restriction of the semi-algebraic bundle $\FM_2(k+n)\to \FM_2(n)$ to $\FM_1(n)\subset \FM_2(n)$.
\end{lemma}
In fact, one can show that Morphism~\eqref{equ:omega orient} is a quasi-isomorphisms, however, we omit the details because this will not be used in this paper.
Thanks to choice of the propagator $\omega^{\orient}$
the essential integration is over the subspace of the fiber of $\FM_2(k,n) \to \FM_1(n)$ for which all of the $k$ points have non-negative $y$-coordinate.
These subspaces assemble into a model for the Swiss-Cheese operad.
\begin{proof}
    First, we may of course compose the morphism \eqref{equ:omega Gamma dir}, for the oriented propagator $\omega^\orient$ of section \ref{sec::Shoikhet}, with the inclusion and $\iota^*$ to obtain the morphism 
    \[
    f: 
    \dGraphs_d^{\nosource,\dual,\gamma^\orient}
    \hookrightarrow 
    \dGraphs_d^{\dual,\gamma^\orient}
    \to 
    \Omega_{PA}(\FM_2)
    \xrightarrow{\iota^*} 
    \Omega_{PA}(\FM_1).
    \]
    Then Lemma~\ref{lem:omega orient} is shown if we can check that $f(I)=0$, or equivalently $f(\Gamma)=0$ for each graph $\Gamma$ that has either an oriented cycle or an external vertex with outgoing edges.
    First, let $\Gamma$ have an oriented cycle.
    Then the integrand in \eqref{equ:omega orient} is zero, since the propagator $\omega^\orient$ is supported on the upper semi-circle.
    Next, suppose that $\Gamma$ does not have a directed cycle, but has an outgoing edge $e$ at the external vertex $j$. Let $v$ be the vertex at the other end of $e$. If $v$ is external, then the integrand $\tilde \omega_\Gamma$ factorizes through ($\pi_{jv}^*$ of) the restriction of $\omega^\orient$ to $\FM_1(2)$, which is zero, hence $\omega_\Gamma^\orient=0$.
    Hence assume that $v$ is an internal vertex. Therefore, there exists a chain of directed edges connecting $v$ so some external vertex $k$ (i.e., directed from $v$ to $k$), since every internal vertex has an outgoing edge and $\Gamma$ does not have directed cycles by assumption. 
    But since the points of the configuration corresponding to the external vertices are on the real line $\mathbb R\subset \mathbb R^2$, the integrand in \eqref{equ:omega orient} again vanishes due to $\omega^\orient$ being supported on the upper semicircle. Hence $\omega^\orient_\Gamma=0$ as desired.
\end{proof}

We will also consider the dual dg Hopf cooperad $\dGraphs_d^\orient$ to $\dGraphs_d^{\orient,\dual}$.
This is related to the operad $\dGraphs_d$ as follows:
The operad $\dGraphs_d$ of directed graphs 
contains an acyclic ideal spanned by directed graphs 
with at least one \emph{sink} (i.e. a vertex without outgoing edges).
Let $\dfGraphs_d^{\nosource}$ be the corresponding quotient.
Then the suboperad 
$\dGraphs_d^{\orient}\subset \dfGraphs_d^{\nosource}$ is spanned by directed graphs such that
\begin{itemize}
	\setlength\itemsep{-0.4em}
	\item There are no directed cycles.
	\item There are no edges starting at the external vertices.
\end{itemize}

We cite the following result from the literature:
\begin{theorem}
	(\cite{Willwacher_oriented}) 
	\label{thm::Wilw::orient}
		The following assignment of a graph to each generator of $\ho \Pois_d$
		\begin{equation}
		\label{eq::map::e_d::Graphs::or}
		\ho \Pois_d(2)\ni \mu_2 \mapsto  
		\begin{tikzpicture}[scale=.5]
		\node[ext] at (0,0) {1};
		\node[ext] at (1.5,0) {2};
		\end{tikzpicture},
		\quad
		\ho \Pois_d(n)\ni \nu_n \mapsto  
		{
			\begin{tikzpicture}[scale=1]
			\node[int] (v) at (0,1) {};
			\node[ext] (v0) at (-1.3,0) {1};
			\node[ext] (v4) at (1.3,0) {n};
			\node[ext] (v1) at (-.7,0) {2};
			\node (v2) at (0,0) {$\cdots$};
			\node[ext] (v3) at (.7,0) {\tiny{n-1}};
			\draw[-triangle 60] (v) edge (v1)  edge (v3) edge (v0) edge (v4);
			\end{tikzpicture}}
		\end{equation}
		extends to a quasiisomorphism of Hopf operads $\ho \Pois_d \to \Graphs_{d+1}^{\orient}$.
\end{theorem}

It was also proved in~\cite{Willwacher_oriented} that the oriented graph complex computes (essentially) the full set of deformations of the Hopf operad $\Pois_d$, and as a consequence is (essentially) quasi-isomorphic  to the ordinary (non-oriented) graph complex
	\begin{equation}\label{equ:GCGCor}
	H(\GC_{d+1}^{\orient})=H(\GC_{d})  \oplus \bigoplus_{\begin{smallmatrix}
		j\geq 1,\\ j\equiv 2d+1 \mod 4  
		\end{smallmatrix}}\kk[d-j]. 
	\end{equation} 
See also \cite{Zivkovic} for a more direct proof of this statement.
In particular, we have $\dim H^{1}(\GC^{\orient}_2)=\dim H^{1}(\GC_1)=1$ and the corresponding generator in $\GC^{\orient}_2$ can be presented as cycle~\eqref{eq::SHoikhet}.

\section{Recognizing the Hopf structure of $\Pois_d^{\odd}$ and of $\Mos$}
\label{sec::Hopf::Mosaic}
Note that the cell decomposition of the mosaic operad is not compatible with the diagonal embedding $\Delta_{\Mos}:\Mos\to\Mos\times \Mos$. Nevertheless, $\Mos$ is a topological operad and one should expect its model in the category of differential graded cocommutative coalgebras. 
First, we describe several Hopf models of the homology operad $H_{\ldot}(\Mos)$ that thanks to Corollary~\ref{thm::Mos::Pois} is equal to $\Pois_1^{\odd}$. Second, we describe the corresponding deformation complex and, third, give a presentation of the free dgca model for the mosaic operad as a deformation of the cofree model of $\Pois_1^{\odd}$ with a given MC (Shoikhet's) element of deformation complex.

\subsection{Simplest Hopf models for $\Pois_d^{\odd}$}
\label{sec::2-gerst}
One can verify that the operad $\Pois_d^{\odd}$ is a Hopf operad 
whose comultiplication $\Delta_{\Pois_d^{\odd}}: \Pois_d^{\odd}\rightarrow \Pois_d^{\odd}\otimes \Pois_d^{\odd}$ is easily defined on the generators of the operad $\Pois_d^{\odd}$:
\begin{equation}
\label{eq::Delta::Pois::odd}
\Delta(\mu_2) = \mu_2\otimes \mu_2, \quad 
\Delta(\nu_3(\ttt,\ttt,\ttt)) = \mu_2(\mu_2(\ttt,\ttt),\ttt) \otimes \nu_3(\ttt,\ttt,\ttt) + 
\nu_3(\ttt,\ttt,\ttt)\otimes \mu_2(\mu_2(\ttt,\ttt),\ttt).
\end{equation}
Note that the latter formulas~\eqref{eq::Delta::Pois::odd} resemble the description~\eqref{eq::Hopf::Pois} of the comultiplication for the $\Pois_d$ operad.

The latter presentation~\eqref{eq::Delta::Pois::odd} leads to the following description of the graded commutative algebra structure on the dual space $(\Pois_d^{\odd})^{\dual}(n)$ discovered in~\cite{Etingof_Rains} for $d=1$: 
\begin{equation}
\label{eq::H_MonR::generators}
 (\Pois_d^{\odd})^{\dual}(n) \simeq
\QQ\left[
\begin{array}{c}
\nu_{ijk}, 1\le i,j,k\le n \\
\nu_{ijk} = (-1)^{d\sigma}\nu_{\sigma(i)\sigma(j)\sigma(k)}
\\
\deg(\nu_{ijk})=2d-1
\end{array} 
\left|
\begin{array}{c}
\nu_{ijk}\nu_{ijl}= 0,
\\
\nu_{ijk}\nu_{klm}
+\nu_{jkl}\nu_{lmi}
+\nu_{klm}\nu_{mij} + 
\\
+\nu_{lmi}\nu_{ijk}
+\nu_{mij}\nu_{jkl} = 0
\end{array}
\right.
\right]
\end{equation}
\begin{remark}
In~\cite{Etingof_Rains} combinatorics of certain posets was used to prove the description of algebras~\eqref{eq::H_MonR::generators}. 
We postpone the alternative proof of this result to the forthcoming paper~\cite{Khor::Hopf}.
\end{remark}

Following the comparison with the models we discussed for operads $\Pois_d$ in Section~\ref{sec::Pois::Hopf} we may consider the following homotopy replacement of the operad $\TwoPois_d$:
\begin{gather*}
\ho\Pois_d^{\odd}:= \calF\left(
\begin{array}{c}
\mu_2, \nu_{2k+1}, k\geq 1 
\\
\deg(\nu_{2k+1})= 1-2kd
\end{array}
\left| 
\begin{array}{c}
\mu_2(\mu_2(x_1,x_2),x_3) = \mu_2(x_1,\mu_2(x_2,x_3)); \\
\nu_{2k+1}(\mu_2(x_0,x_1),x_2,\ldots,x_{2k+1}) = \phantom{........} \\
\quad = \mu_2(x_0,\nu_{2k+1}(x_1,\ldots,x_{2k+1})) + \phantom{....} \\
\quad \quad + \mu_2(x_1,\nu_{2k+1}(x_0,x_2,\ldots,x_{2k+1}));
\\
d(\nu_{2k+1})=\sum_{i+j=k}\sum_{\sigma\in S_{2k+1}} (-1)^{\sigma}\sigma\cdot \nu_{2i+1}\circ_1\nu_{2j+1};
\end{array}
 \right.\right)
\\
\Delta(\mu_2) = \mu_2\otimes \mu_2, \quad 
\Delta(\nu_{2k+1}) = \underbrace{\mu_2\circ\mu_2\circ\ldots\circ \mu_2}_{2k} \otimes \nu_{2k+1} + 
\nu_{2k+1}\otimes \underbrace{\mu_2\circ\mu_2\circ\ldots\circ \mu_2}_{2k}.
\end{gather*}
where the generators $\{\nu_{2k+1}| k\geq 1\}$ generate the free resolution $\LL_{\infty}^{\odd}=\Omega([\Com]_{\ZZ_2}^{\dual})$ of the suboperad $\Lie^{\odd}\subset \TwoPois_d$. The latter suboperad $\Lie^{\odd}$ is the suboperad generated by the ternary operation $\nu_3$.

\subsection{Fibrant dgca model for $\TwoPois_d$ via "odd" graphs}
\label{sec::Graphs::odd}
We say that a vertex in a directed graph is \emph{odd} (respectively \emph{even}) if the number of outgoing edges is odd (resp. even). 
An oriented graph from $\Graphs_{d}^{\orient}$ is 
called \emph{odd} if for each internal vertex the number of outgoing edges is odd.
We denote by $\Graphs_d^{\orient,\odd,\dual}\subset \Graphs_d^{\orient,\dual}$ the dg Hopf sub-cooperad spanned by odd graphs.
To check that this is indeed a dg Hopf sub-cooperad
note that the contraction of an edge connecting two odd vertices necessarily produces an odd vertex, hence $\Graphs_d^{\orient,\odd,\dual}$ is closed under the differential. Closedness under the cooperad structure and commutative product is clear since these operations do not change the out-valence of internal vertices.

We denote by $\Graphs_d^{\orient,\odd}$ the dual dg Hopf operad to $\Graphs_d^{\orient,\odd,\dual}$.
This is a quotient of $\Graphs_d^{\orient}$ by the operadic ideal spanned by oriented graphs containing at least one \emph{even} internal vertex.

\begin{theorem}
\label{thm::H:Graphsodd}
	\label{thm::H:Graph:1}
	The following assignment of a graph to each generator of $\ho \TwoPois_{d}$
	\begin{equation}
	\label{eq::map::e_d::Graphs::odd}
	\ho\TwoPois_{d}(2)\ni \mu_2 \mapsto  
	\begin{tikzpicture}[scale=.5]
	\node[ext] at (0,0) {\small{1}};
	\node[ext] at (1.5,0) {\small{2}};
	\end{tikzpicture},
	\quad
	\ho\TwoPois_{d}(2k+1)\ni \nu_{2k+1} \mapsto  
		\begin{tikzpicture}[scale=1]
		\node[int] (v) at (0,1) {};
		\node[ext] (v0) at (-1.3,0) {1};
		\node[ext] (v4) at (1.3,0) {\tiny{2k+1}};
		\node[ext] (v1) at (-.7,0) {2};
		\node (v2) at (0,0) {$\cdots$};
		\node[ext] (v3) at (.7,0) {\tiny{2k}};
		\draw[-triangle 60] (v) edge (v1)  edge (v3) edge (v0) edge (v4);
		\end{tikzpicture}
	\end{equation}
	extends to a quasi-isomorphism of (complete) Hopf operads $F:\ho \TwoPois_{d} \to \Graphs_{d+1}^{\orient,\odd}$.
\end{theorem}
\begin{proof}
Let us compute the cohomology of $\Graphs_{d+1}^{\orient,\odd}$.
To this end we endow $\Graphs_{d+1}^{\orient,\odd}$ with the descending complete filtration by the number of top (source) vertices in graphs, i.e., the number of internal vertices with no inputs.
We consider the associated graded complexes $\mathrm{gr}^t \Graphs_{d+1}^{\orient,\odd}$, where the "$t$" shall remind that this is via the filtration by the number of top vertices.
The differential consists of just those splittings of vertices that do not increase the number of top vertices.

Next, we filter $\mathrm{gr}^t \Graphs_{d+1}^{\orient,\odd}$ again by the number of directed paths between top vertices and external vertices. The differential on the path-associated graded $\mathrm{gr}^p \mathrm{gr}^t \Graphs_{d+1}^{\orient,\odd}$
consists of only those splittings of vertices that do not decrease the number of paths. The cohomology can be evaluated as in \cite{Markl_Voronov}. 
To this end we take yet a third filtration as follows.
We temporarily call an edge from a vertex $v$ to a vertex $w$ separating, such that the edge is the only outgoing edge at $v$, and the only incoming edge at $w$, pictorially
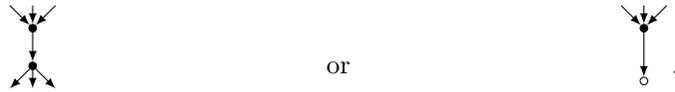
\begin{align*}
&
\begin{tikzpicture}
\node[int] (v) at (0,.5) {};
\node[int] (w) at (0,0) {};
\draw (v) edge[de] (w) 
(v) edge[ed] +(-.3,.3) edge[ed] +(0,.3) edge[ed] +(.3,.3)
(w) edge[de] +(-.3,-.3) edge[de] +(0,-.3) edge[de] +(.3,-.3); 
\end{tikzpicture}
&\text{or}
& &
\begin{tikzpicture}
\node[int] (v) at (0,.5) {};
\node[ext] (w) at (0,-.2) {};
\draw (v) edge[de] (w) 
(v) edge[ed] +(-.3,.3) edge[ed] +(0,.3) edge[ed] +(.3,.3);
\end{tikzpicture}\, .
\end{align*}
We filter by the number of non-separating edges, so that the differential on the associated graded complex $\mathrm{gr}^{ns}\mathrm{gr}^p \mathrm{gr}^t \Graphs_{d+1}^{\orient,\odd}$
consist only of those terms creating a separating edge, while also not decreasing the number of paths or creating top vertices.
Pictorially it is given by the following operations
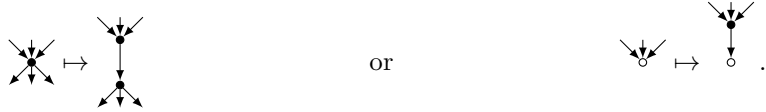
\begin{align*}
\begin{tikzpicture}[baseline=-.65ex]
\node[int] (v) at (0,0) {};
\draw 
(v) edge[ed] +(-.3,.3) edge[ed] +(0,.3) edge[ed] +(.3,.3)
(v) edge[de] +(-.3,-.3) edge[de] +(0,-.3) edge[de] +(.3,-.3); 
\end{tikzpicture}
&\mapsto
\begin{tikzpicture}[baseline=-.65ex]
\node[int] (v) at (0,.3) {};
\node[int] (w) at (0,-.3) {};
\draw (v) edge[de] (w) 
(v) edge[ed] +(-.3,.3) edge[ed] +(0,.3) edge[ed] +(.3,.3)
(w) edge[de] +(-.3,-.3) edge[de] +(0,-.3) edge[de] +(.3,-.3); 
\end{tikzpicture}
&\text{or}
 && 
\begin{tikzpicture}[baseline=-.65ex]
\node[ext] (w) at (0,0) {};
\draw
(w) edge[ed] +(-.3,.3) edge[ed] +(0,.3) edge[ed] +(.3,.3);
\end{tikzpicture}
&\mapsto
\begin{tikzpicture}[baseline=-.65ex]
\node[int] (v) at (0,.5) {};
\node[ext] (w) at (0,0) {};
\draw (v) edge[de] (w) 
(v) edge[ed] +(-.3,.3) edge[ed] +(0,.3) edge[ed] +(.3,.3);
\end{tikzpicture}\, .
\end{align*}
This differential has an obvious homotopy by contracting a separating edge.
We can hence identify the cohomology of $\mathrm{gr}^{ns}\mathrm{gr}^p \mathrm{gr}^t \Graphs_{d+1}^{\orient,\odd}$ with the subquotient consisting of graphs with no internal edge of valency $(\geq 2,\geq 2)$ and no external vertices of valency $\geq 2$, modulo all graphs with separating edges. Combinatorially in all nonzero such diagrams all vertices (internal and external) have either zero or one incoming edge.
Now one quickly verifies that the space of such diagrams is identified precisely with $\ho \TwoPois_{d}$, and the identification is done via our map $F$ (i.e., the leading order terms with respect to our filtrations).
Hence by standard spectral sequence arguments, the map $F$ is a quasi-isomorphism.
(Here we note that $\Graphs_{d+1}^{\orient,\odd}$ splits into a direct product of subcomplexes according to the loop order of graphs, and each subcomplex is finite dimensional. Thus convergence of the spectral sequences considered is automatic.)
\end{proof}

We remark that the exact same argument also gives a proof of Theorem~\ref{thm::Wilw::orient}, alternative to the one given in~\cite{Willwacher_oriented}.

\subsection{Hopf model of the mosaic operad $\Mos$ via odd graphs}
\label{sec::Graphs::MonR}
Let $\GC_d^{\orient,\odd}\subset \GC_d^\orient$ be the Lie subalgebra of directed graphs whose loop order is even. In other words, the Euler characteristic is odd.
\begin{proposition}
The action of the dg Lie algebra $\GC_d^\orient$ on $\Graphs_d^\orient$ (resp. on $\Graphs_d^{\orient,\dual}$) descends to an action of $\GC_d^{\orient,\odd}$ on $\Graphs_d^{\orient,odd}$ (resp. on $\Graphs_d^{\orient,\odd,\dual}$).
\end{proposition}
\begin{proof}
It suffices to check the case of the action of $\GC_d^{\orient,\odd}$ on $\Graphs_d^{\orient,\odd,\dual}$, the case of $\Graphs_d^{\orient,\odd}$ follows by duality.
To this end, we need to show that the subspace
\[
\Graphs_d^{\orient,\odd,\dual}(n)\subset \Graphs_d^{\orient,\dual}(n)
\]
spanned by graphs all of whose internal vertices are odd, is closed under the action of any $\gamma\in \GC_d^{\orient}$ of even loop order.
Let $\Gamma\in \Graphs_d^{\orient,\odd,\dual}(n)$ be a graph with only odd internal vertices. Then the action of $\gamma$ on $\Gamma$ has three terms, see \eqref{equ:G coaction}, that we need to consider in turn.
Note that in the second and third term, either a subgraph is contracted to an external vertex, or the graph $\Gamma$ is reduced to a subgraph $\Gamma'$. These operations do not affect all remaining internal vertices and hence the odd-ness condition on the internal vertices is not violated. 
The only critical term in \eqref{equ:G coaction} is the first summand,
where a subgraph $\gamma'$ (isomorphic to $\gamma$) without external vertices is contracted to an internal vertex. We need to check that the newly created internal vertex $v$ has an odd number of outgoing edges.
Say $\gamma'$ (and hence also $\gamma$) has $m$ vertices and $k$ edges, with even loop order $l=k-m+1$.
The total number of outgoing edges from vertices of $\gamma$ hence has the same parity as $m$, while the number of those edges that are part of $\gamma'$ has the opposite parity since $l=k-m+1$ is assumed to be even.
Hence the number of remaining outgoing edges at the new internal vertex $v$ is necessarily odd, showing the proposition.
\end{proof}

Thanks to Lemma~\ref{lem:SMCevenlooporder} the Shoikhet MC element $\gamma^{\orient}$ recalled in~\eqref{eq::MC::Shoikhet} belongs to the Lie subalgebra $\GC_2^{\orient,\odd}$.	
In particular, one can consider the twisted operad $(\Graphs_{2}^{\orient,\odd})^{\gamma^{\orient}}$ twisted by the Shoikhet MC element $\gamma^{\orient}$. The twisting affects the differential, but does not change the operadic composition. 
In fact, the twisting does also not affect the cohomology (although it changes the homotopy type).
\begin{proposition}\label{prop:H Graphsodd tw}
    The cohomology of the twisted cooperad $(\Graphs_{2}^{\orient,\odd})^{\gamma^{\orient}}$ is identified with the odd Poisson cooperad
    \[
    H\left((\Graphs_{2}^{\orient,\odd, \dual})^{\gamma^{\orient}}\right)
    \cong \Pois_1^{\odd,\dual}.
    \]
\end{proposition}
\begin{proof}
    Let $\Gamma\in \Graphs_{2}^{\orient,\odd, \dual}$ be a graph with $k$ internal vertices and $e$ edges. Then we say that $\Gamma$ is of \emph{complexity} $e-k$. Equivalently, the complexity is the loop order of the graph obtained by fusing all external vertices together.
    The complexity grading is preserved by the untwisted differential $d$, and it is compatible (additive) with the Hopf cooperad structure.
    It is also compatible with the action of $\dGC_2^{\orient,\odd}$ in the sense that if $\gamma\in \dGC_2^{\orient,\odd}$ is of loop order $\ell$, then $\gamma\cdot \Gamma$ is of complexity $e-k-\ell$.
    
    We may hence endow $(\Graphs_{2}^{\orient,\odd, \dual})^{\gamma^{\orient}}$ with the ascending exhaustive filtration by complexity, and consider the associated spectral sequence, which converges to cohomology.
    Since $\gamma^{\orient}$ is of loop order $\geq 2$ we have that the first page of the spectral sequence is 
    \[
    (\Graphs_{2}^{\orient,\odd, \dual}, d),
    \]
    with the untwisted differential $d$. By Theorem \ref{thm::H:Graph:1} we hence have that the next page is
    $E^1= H(\Graphs_{2}^{\orient,\odd, \dual}, d) \cong \Pois_2^{\odd,\dual}$.
    By Corollary \ref{cor:Poisdodd generators} $\Pois_2^{\odd,\dual}$ is cogenerated in arities 2 and 3. The corresponding cocycles in $\Graphs_{2}^{\orient,\odd, \dual}$ are the graphs 
    \begin{equation}\label{equ:Gamma 0 1}
    \begin{aligned}
    \Gamma_0 &= \begin{tikzpicture}
            \node[ext] (v) at (0,0) {$\scriptstyle 1$};
            \node[ext] (v) at (.7,0) {$\scriptstyle 2$};
        \end{tikzpicture}
        &
        &\text{and} 
        &
        \Gamma_1 &= 
        \begin{tikzpicture}
            \node[ext] (v1) at (0,0) {$\scriptstyle 1$};
            \node[ext] (v2) at (.7,0) {$\scriptstyle 2$};
            \node[ext] (v3) at (1.4,0) {$\scriptstyle 3$};
            \node[int] (w) at (.7,.7) {};
            \draw (w) edge[->] (v1) edge[->] (v2) edge[->] (v3);
        \end{tikzpicture}.        
    \end{aligned}
    \end{equation}
    which live in degree 0 and complexity 0, and degree 1 and complexity 2 respectively. In particular, this means that $E^1$ is concentrated on a line in the degree-complexity-plane on which the complexity is twice the degree. The higher differentials in our spectral sequence increase the degree by one, but reduce the complexity by at least two. Hence all higher differentials vanish, and the proposition follows.
\end{proof}

We may then compose the morphism \eqref{equ:omega orient} of Lemma \ref{lem:omega orient} with the inclusion to obtain a morphism of dg Hopf collections compatible with the cooperad structures
\begin{equation}\label{equ:omega odd}
	\begin{gathered}
	(\Graphs_{2}^{\orient,\odd, \dual})^{\gamma^{\orient}}
	\hookrightarrow 
	(\Graphs_{2}^{\orient, \dual})^{\gamma^{\orient}}
	\to
	\Omega_{PA}(\FM_1) \\
	\Gamma\mapsto \omega^{\orient}_{\Gamma}:=\int_{\SC(m,n)\to \FM_1(n)}\bigwedge_{e\in \Gamma}\omega^{\orient}_{e}
	\end{gathered}.
\end{equation}

Our next goal is to show that the morphism \eqref{equ:omega odd} takes values in $\Omega_{PA}(\Mos)\subset \Omega_{PA}(\FM_1)$. This is established by the following Lemma.

\begin{lemma}
	\label{lm::corr::wG}
	Let $T_1\eqvr T_2$ be a pair of leaf-labeled rooted trees from $\Tre_n$ that are equivalent under the equivalence relation introduced in Section~\ref{sec::Mos=As}. Let $U_{T_1}$ (respectively $U_{T_2}$) be the corresponding cells in $\FM_1(n)$. Then the pushforwards of the restrictions of the form $\omega^{\orient}_{\Gamma}$ on the corresponding cells coincides:
	\[
	 \left({p_n}\Bigr|_{\substack{U_{T_1}}}\right)_*\left(\omega_{\Gamma}^{\orient} \Bigr|_{\substack{U_{T_1}}}\right) =
	\left({p_n}\Bigr|_{\substack{U_{T_2}}}\right)_*\left(\omega_{\Gamma}^{\orient} \Bigr|_{\substack{U_{T_2}}}\right)
	\]
	In particular, the morphism \eqref{equ:omega odd} takes values in $\Omega_{PA}(\Mos)$.
\end{lemma}
\begin{proof}
	The statement of the Lemma is obvious for the top-dimensional cells. For example, as predicted by the flip $\tau:\Ass\to \Ass$ defined in Section~\ref{sec::Ass::Inv} the orientation of the cell $U_{\sigma}$ and $U_{\sigma^{op}}$ differs by the sign $(-1)^{n-1}$.
	Consequently, the forms ${p_n}_*\left(\omega_{\Gamma}^{\orient}\Bigr|_{\substack{U_{\sigma}}}\right)$ and 
	${p_n}_*\left(\omega_{\Gamma}^{\orient}\Bigr|_{\substack{U_{\sigma^{op}}}}\right)$ coincides for the simplest directed graph 
	$\Gamma=\begin{tikzpicture}[scale=0.5]
	\node[int] (v) at (0,1) {};
	\node[ext] (v0) at (-1.3,0) {1};
	\node[ext] (v4) at (1.3,0) {\small{n}};
	\node (v2) at (0,-0.3) {$\cdots$};
	\draw[-triangle 60] (v)  edge (v0) edge (v2) edge (v4);
	\end{tikzpicture}$ (with a unique inner vertex) if and only if $n$ is odd.
	
	The coincidence for general strata follows from the factorization property of the strata:
	\[ U_T:= \bigtimes\limits_{v\in\vertices{T}} U_{\sigma_v}, \ \text{ with } \sigma_v\in S_{|in(v)|}
	\]
	 and the factorization property of fibered integrals: 	
\begin{equation}
\label{eq::omega::factor}
	\omega^{\orient}_{\Gamma}\Bigr|_{\substack{U_{T}}} = \bigotimes_{v\in\vertices(T)} \omega^{\orient}_{\Gamma_v}\Bigr|_{\substack{U_{\sigma_v}}},
\end{equation}
	where the graphs $\Gamma_v$ are uniquely determined by the cooperad structure and cooperadic cocomposition prescribed by $T$. \end{proof}

	\begin{theorem}
		\label{thm::Mos::Model}	
			The fiber-integral assignment $\Gamma\mapsto \omega^{\orient}_{\Gamma}$ defines a quasi-isomorphism between the twisted Hopf cooperad $(\Graphs_{2}^{\orient,\odd, \dual})^{\gamma^{\orient}}$ and the differential forms on the mosaic operad $\Omega_{PA}(\Mos)$.
			 
			Equivalently, the twisted operad $(\Graphs_{2}^{\orient,\odd}, d+[{\gamma^{\orient}},\ttt])$ defines a real  algebraic model of the Mosaic operad in the category of (complete) dg-coalgebras.
		\end{theorem}

		\begin{remark}
			The functor of differential forms is not a comonoidal functor, thus the dgca's of differential forms of a topological operad is not really a cooperad but is very close to be.
			Following \S3 of~\cite{KW::Framed} and Definition~3.1 of~\cite{Lamb_Vol} we show that the following diagrams are commutative for all pairs $(m,n)$:
		\begin{equation}
		\label{eq::diag::coproduct}
		\begin{tikzcd}
		((\Graphs_{2}^{\orient,\odd})^{\gamma^{\orient}})^{\dual}(m+n)
		\arrow[dd] \arrow[rr,"\Gamma\mapsto \omega_{\Gamma}^{\orient}"] &  & \Omega_{PA}(\MonR{m+n+1})
		\arrow[d] \\
		 & &  \Omega_{PA}(\MonR{m+2}\times \MonR{n+1}) \\
		((\Graphs_{2}^{\orient,\odd})^{\gamma^{\orient}})^{\dual}(m+1) \otimes
		((\Graphs_{2}^{\orient,\odd})^{\gamma^{\orient}})^{\dual}(n)
		\arrow[rr,"\Gamma_1\otimes\Gamma_2\mapsto \omega_{\Gamma_1}^{\orient}\otimes \omega_{\Gamma_2}^{\orient}"] & &
		\Omega_{PA}(\MonR{m+2})\otimes\Omega_{PA}(\MonR{n+1}) 
		\arrow[u]
		\end{tikzcd}.
		\end{equation}
		Where the left downarrow is given by the cooperad structure of $(\Graphs_{2}^{\orient,\odd})^{\gamma^{\orient}}$ and the right vertical arrows is what we have instead of cooperadic cocomposition for the collection of commutative algebras $\Omega_{DR}(\Mos)$. \end{remark}

		\begin{proof}
		From the construction we already know that our morphism \eqref{equ:omega odd} is a morphism of dg Hopf collections and compatible with the cooperad structures in the above sense. It only remains to check that our morphism is a quasi-isomorphism into $\Omega_{PA}(\Mos)$.

        To this end we note that by Proposition \ref{prop:H Graphsodd tw} and by Corollary \ref{thm::Mos::Pois} we have isomorphisms of cooperads
        \[
        H\left((\Graphs_{2}^{\orient,\odd, \dual})^{\gamma^{\orient}}\right)
        \cong \Pois_1^{\odd,\dual} \cong 
        H(\Mos),
        \]
        hence we only have to check that the map $\Gamma\to \omega_\Gamma^\orient$ induces an isomorphism. By arity-wise finite dimensionality it is sufficient to show that the cohomology map $[\omega^\orient]$ is injective. But since $\Pois_1^{\odd,\dual}$ is cogenerated in arities 2 and degree 0 and arity 3 and degree 1 by Corollary \ref{cor:Poisdodd generators}, we just have to check that $[\omega^\orient]$ is an isomorphism in arities 2 and degree 0 and arity 3 and degree 1.
        
        In arity 2 we have $H^0\left((\Graphs_{2}^{\orient,\odd})^{\gamma^{\orient}}(2)\right)\cong \mathbb R$, and the representative is the graph $\Gamma_0$ of \eqref{equ:Gamma 0 1}.
        We trivially have that $\omega_{\Gamma_0}=1$ is the constant form, and this spans $H^0(\Mos(2))$.
        Similarly in arity 3, $H^1\left((\Graphs_{2}^{\orient,\odd})^{\gamma^{\orient}}(3)\right)\cong \mathbb R$, and the representative is the graph $\Gamma_1$ of \eqref{equ:Gamma 0 1}.
        The resulting form $\omega_{\Gamma_1}^\orient$ is a top form on $\FM_1(3)$ that descends to $\Mos(3)$. To check that this indeed represents a non-trivial cohomology class (on the circle $\Mos(3)\simeq S^1$) we may just integrate our form over the fundamental cycle and show that the result is nonzero.
        Indeed, using the definition of $\omega_{\Gamma_1}^\orient$ and that the integral factorizes,
        \[
            \int_{\FM_1(3)} \omega_{\Gamma_1}^\orient 
            =
            \int_{\SC(1,3)} \pi_{12}\omega^\orient\wedge \pi_{13}\omega^\orient\wedge \pi_{14}\omega^\orient
            =
            \left( \int_{\SC(1,1)} \omega^\orient \right)^3
            =
            \left( \int_{\FM_2(2)} \omega^\orient \right)^3
            =1^3=1.
        \]
        Hence the theorem is shown.
		\end{proof}

\section{Rational homotopy type of $\MonR{n+1}$}
\label{sec::Rational::MONR::all}

\subsection{Internally connected graphs as (ho)Lie models of configuration space}
As we mentioned in~\S\ref{sec::FM} the configuration space $\Conf_n(\RR^d)$ on the $d$-dimensional space is homotopy equivalent to $\FM_{d}(n)$ which is one of the nice models of the little discs operad. The homology operad $H_{\ldot}(E_d,\QQ)$ is equal to $\Pois_d$ and the space of $n$-ary operations of the corresponding cohomology cooperad is the quadratic algebra known under the name Orlik-Solomon algebra $OS_d(n)$ whose presentation was given in~\eqref{eq::OS::def}:
$$
OS_d(n)\simeq H^{\udot}(E_d(n))= H^{\udot}(\FM_d(n))= H^{\udot}(\Conf_n(\RR^d)).
$$
The quadratic algebra $OS_d(n)$ is Koszul and the easiest way to show it is that it admits a quadratic Gr\"obner bases. 
However, we recall the main ideas of a different proof of the koszulness property suggested by the second author in~\cite{Severa_Willwacher}. 
In the subsequent section~\S\ref{sec::ICGodd} we show how one can generalize these ideas to show the koszulness of the  algebra $({\TwoPois_{d}}(n))^{\dual}\simeq H^{\udot}(\MonR{n+1};\QQ)$ (Corollary~\ref{cor::koszul::Twopois} below).
Note, that it was verified by E.\,Rains that the latter quadratic algebra does not admit a quadratic Gr\"obner basis for $n\geq 6$ and, therefore, we have to use the complicated combinatorics of graphs to show the Koszul property.
We recalled in~\S\ref{sec::Graphs::E_d} the following quasi-isomorphisms of commutative algebras, that are compatible with the operad structure:
$$
\begin{tikzcd}
OS_d(n) & \arrow[l,"quis"',two heads] \Graphs_{d}^{\dual}(n) \arrow[r,equal] & C^{\udot}_{CE}(\ICG_d).
\end{tikzcd}
$$
The Lie algebra that is quadratic dual to the Orlik-Solomon algebra is known under the name Drinfeld-Kohno Lie algebra $\tkd_d(n)$\footnote{The lower index $d$ corresponds to the dimension of the unit ball in the little $d$-discs operad $E_d\simeq\FM_d$. The classical Drinfeld-Kohno Lie algebra corresponds to the case $d=2$ when $E_2(n)$ is an aspherical space with $\tkd_2(n)$ equals the Malcev completion of the pure braid group.}
 has the following presentation by generators and relations:
\[
\tkd_d(n):=\Lie\left(
\begin{array}{c}
	t_{ij}, 1\le i\neq j\le n \\
	t_{ij} = (-1)^{d} t_{ji} \\
	\deg(t_{ij}) = 2 - d
\end{array} 
\left|
\begin{array}{c}
	{[t_{ij},t_{ik}+ t_{jk}]= 0 } \\
	{ [t_{ij},t_{kl}]=0, }
	\\
	{\text{for } \#\{i,j,k,l\}=4. }
\end{array}
\right.
\right)
\]
\begin{definition}
\label{def::complexity}
The loop order of an internally connected graph $\Gamma$ with all external vertices colied is called the \emph{complexity} of $\Gamma$ and is equal to 
$$
\#\text{edges}(\Gamma) - \#\text{internal vertices}(\Gamma).$$
\end{definition}
Note that complexity defines a grading on the $\LL_{\infty}$-algebra $\ICG(n)$.
It was shown in~\cite{Severa_Willwacher} using appropriate filtration that each homology class of the $L_{\infty}$-algebra of internally connected graphs $\ICG_d(n)$ has a leading term that consists of the sum of internally trivalent trees. What follows that the homology of the graded component $\ICG_d(n)_{(m)}$ (consisting of graphs of complexity $m$) is concentrated in a unique homological degree $m(2-d)$ and coincides with the $m$-th graded component of the Lie algebra $\tkd_{d}(n)$. 
Therefore, we can define a truncated $L_{\infty}$-subalgebra of $\ICG_d(n)$:
\begin{equation}
	\label{eq::truncation::ICG}
	\left(\TICG_{2}(n)_{(m)}\right)^{s}:=\left\{
	\begin{array}{c}
		{\left({\ICG_{d}}(n)_{(m)}\right)^{s}, \text{ if } s< m(2-d), }\\
		{\left({\ICG_{d}}(n)_{(m)}\right)^{s}_{\text{closed}}, \text{ if } s= m(2-d), } \\
		{0, \text{ if } s>m(2-d), }
	\end{array}
	\right.
\end{equation}
Here $\ICG_{2}(n)_{(m)}$ denotes the graded component spanned by graphs whose internal loop order is equal to $m$ and $	\left(\TICG_{2}(n)_{(m)}\right)^{s}$ is the subset of chains of homological degree $s$.
It follows that we have a collection of quasi-isomorphisms of $L_{\infty}$-algebras that are compatible with the operad structure:
\begin{equation}
\label{eq::trunc::i}
\begin{tikzcd}
	\tkd_d(n) & \arrow[l,two heads] \TICG_{d}(n) \arrow[r,hook] & \ICG_d(n).
\end{tikzcd}
\end{equation}
This finishes the proof of the Koszul property for the Orlik-Solomon algebra $OS_d(n)$ and its quadratic dual Drinfeld-Kohno Lie algebra $\tkd_{d}(n)$.
In the next subsection~\ref{sec::ICGodd} we will try to follow the same strategy to show that the quadratic algebra $\TwoPois_{d}(n)$ is Koszul.

\subsection{The rational homotopy type of $\TwoPois_d(n)$}	
\label{sec::ICGodd}
We proved in Theorem~\ref{thm::H:Graphsodd} that the operad of oriented odd graphs $\Graphs_{d+1}^{\orient,\odd}$ is a dg-model of the Hopf operad $\TwoPois_d$.
Recall that $\Graphs_{d+1}^{\orient,\odd}$ is spanned by oriented graphs, such that each internal vertex is at least trivalent, there are no directed cycles, each external vertices have only incoming edges and moreover each internal vertex has odd number of outgoing edges.
 One of the main features of this model is that the space of $n$-ary operations $\Graphs_{d+1}^{\orient,\odd}(n)$ is the cofree cocommutative coalgebra whose cogenerators are spanned by internally connected odd graphs. I.e. the cogenerators correspond to the set of graphs that remains to be connected after deleting all external vertices. We call this $\LL_\infty$-algebra "\emph{internally connected odd graphs}" and denote it by $\ICG^{\orient,\odd}_{d+1}$ following the notation of~\cite{Severa_Willwacher} for the case of the operad of $\Graphs_{d+1}$.
This space is endowed with the vertex splitting differential and the $\LL_\infty$-structure that can be read of the isomorphism between the space of all possible graphs and the Chevalley-Eilenberg complex of internally connected graphs  
\begin{equation}
\label{eq::ICG::def}
\Graphs_d^{\orient,\odd}(n) = C^{CE}_{\ldot}(\ICG^{\orient,\odd}_d(n))
\end{equation}
The part of the vertex-splitting differential that decreases the number of internally connected components remembers the desired $\LL_\infty$-structure. The latter may happen only while splitting an external vertex.

In this section we compute the homology of the $\LL_{\infty}$-algebra $\ICG^{\orient,\odd}_{d+1}(n)$ for all $n$. In particular, we show that the cohomology of these $\LL_{\infty}$-algebras are quadratic Koszul Lie algebras, which assemble into a model of $\TwoPois_d$ in the category of Lie algebras.
\begin{lemma}
	The quadratic dual Lie algebra to the commutative algebra $(\TwoPois_{d}(n))^{\dual}$ admits the following presentation:
\begin{equation}
	\label{eq::Lie::MonR::generators}
	\todd_d(n):=
	\Lie \left(
	\begin{array}{c}
		\nu_{ijk}, 1\le i,j,k\le n \\
		\nu_{ijk} = (-1)^{d\sigma}\nu_{\sigma(i)\sigma(j)\sigma(k)}
		\\
		\deg(\nu_{ijk})=2-2d
	\end{array} 
	\left|
	\begin{array}{c}
		{[\nu_{ijk},\nu_{pqi}+ \nu_{pqj} + \nu_{pqk}]= 0 }\\
		{ [\nu_{ijk},\nu_{pqr}]=0  } 
		\\
		{\text{for } \#\{i,j,k,p,q,r\}=6.}
	\end{array}
	\right.
	\right).
\end{equation}
The composition maps $\circ_{1}:\todd_d(n)\oplus \todd_d(1+m)\to \todd_{d}(n+m)$ in the operad $\todd_d$ are defined by the following maps of generators:
\begin{equation}
	\label{eq::todd::operad}
\begin{array}{c}
\todd_d(n)\ni\nu_{ijk} \mapsto \nu_{ijk}\in\todd_d(n+m)	 \ (1\leq i,j,k\leq n), \\
\todd_d(m+1)\ni\nu_{ijk} \mapsto \begin{cases}
	\sum_{l=1}^{n} \nu_{l,j+n,k+n}, \text{ if } i=1, \\
	\nu_{i+n,j+n,k+n} \text{ if } i,j,k>1.
\end{cases}	
\end{array}
\end{equation}
Finally, the assignment \(\nu_{ijk} \mapsto 
	\left[
	\begin{tikzpicture}[scale=0.5]
		\node[int] (v) at (0,1) {};
		\node[ext] (v1) at (-1,0) {i};
		\node[ext] (v0) at (0,0) {j};
		\node[ext] (v2) at (1,0) {k};
		\draw[-triangle 60] (v) edge (v1)  edge (v0) edge (v2);
	\end{tikzpicture}
	\right]
	\) 
	extends to an embedding of operads in the category of Lie algebras 
	\[\psi:\todd_d(n) \rightarrow H^{\udot}(\ICGod{d+1}(n)).\]
\end{lemma}
\begin{proof}
The presentation via generators and relations extends the one found in~\cite{Etingof_Rains}. However, these relations can be found either from the operad structure~\eqref{eq::todd::operad} (see~\cite{Khor::Hopf} where it is explained in details in bigger generality) or from the pictorial relations with internally connected graphs.
Indeed, the tripod graph $\left[
\begin{tikzpicture}[scale=0.5]
	\node[int] (v) at (0,1) {};
	\node[ext] (v1) at (-1,0) {i};
	\node[ext] (v0) at (0,0) {j};
	\node[ext] (v2) at (1,0) {k};
	\draw[-triangle 60] (v) edge (v1)  edge (v0) edge (v2);
\end{tikzpicture}
\right]$ 
is the unique graph of the minimal complexity $2$ connected to vertices $i$, $j$ and $k$ and is a generator of $H^{\udot}(\ICGod{d+1}(n))$.
On the other hand, the relations are verified for the graphs of internal loop order $4$.
Indeed, if $|\{i,j,k,p,q,r\}|=6$ then $\psi(\nu_{ijk})$ and $\psi(\nu_{pqr})$ commute, because the subsets of external vertices they are connected to do not intersect.
However, for a subset $\{i,j,h,p,q\}$ of indices of cardinality $5$ we have the following relations for the cohomological classes:
\begin{align*}
0=	
\left[d\left(
\begin{tikzpicture}[scale=0.5]
\node[int] (v) at (-1,1.2) {};
\node[int] (w) at (1,1.2) {};
\node[ext] (v1) at (-2,0) {i};
\node[ext] (v2) at (-1,0) {\small{j}};
\node[ext] (v3) at (0,0) {k};
\node[ext] (v4) at (1,0) {\small{p}};
\node[ext] (v5) at (2,0) {\small{q}};
\draw[-triangle 60] (v) edge (v1)  edge (v2) edge (v3);
\draw[-triangle 60] (w) edge (v)  edge (v4) edge (v5);
\end{tikzpicture}
\right)\right]
 & =  &
\left[\begin{tikzpicture}[scale=0.5]
\node[int] (v) at (-1,1.2) {};
\node[int] (w) at (1,1.2) {};
\node[int] (w1) at (0,1) {};
\node[ext] (v1) at (-2,0) {i};
\node[ext] (v2) at (-1,0) {\small{j}};
\node[ext] (v3) at (0,0) {k};
\node[ext] (v4) at (1,0) {\small{p}};
\node[ext] (v5) at (2,0) {\small{q}};
\draw[-triangle 60] (v) edge (v1)  edge (v2) edge (w1);
\draw[-triangle 60] (w) edge (w1)  edge (v4) edge (v5);
\draw[-triangle 60] (w1) edge (v3);
\end{tikzpicture}\right]
\ & + & \
\left[\begin{tikzpicture}[scale=0.5]
\node[int] (v) at (-1,1.2) {};
\node[int] (w) at (1,1.2) {};
\node[int] (w1) at (0,1) {};
\node[ext] (v1) at (-2,0) {\small{j}};
\node[ext] (v2) at (-1,0) {k};
\node[ext] (v3) at (0,0) {i};
\node[ext] (v4) at (1,0) {\small{p}};
\node[ext] (v5) at (2,0) {\small{q}};
\draw[-triangle 60] (v) edge (v1)  edge (v2) edge (w1);
\draw[-triangle 60] (w) edge (w1)  edge (v4) edge (v5);
\draw[-triangle 60] (w1) edge (v3);
\end{tikzpicture}\right]
\ & + & \
\left[\begin{tikzpicture}[scale=0.5]
\node[int] (v) at (-1,1.2) {};
\node[int] (w) at (1,1.2) {};
\node[int] (w1) at (0,1) {};
\node[ext] (v1) at (-2,0) {k};
\node[ext] (v2) at (-1,0) {i};
\node[ext] (v3) at (0,0) {\small{j}};
\node[ext] (v4) at (1,0) {\small{p}};
\node[ext] (v5) at (2,0) {\small{q}};
\draw[-triangle 60] (v) edge (v1)  edge (v2) edge (w1);
\draw[-triangle 60] (w) edge (w1)  edge (v4) edge (v5);
\draw[-triangle 60] (w1) edge (v3);
\end{tikzpicture}\right]
= \\
& = &
[\psi(\nu_{ijk}),\psi(\nu_{kpq})] 
\ & + & \
[\psi(\nu_{ijk}),\psi(\nu_{ipq})]
\ & + & \
[\psi(\nu_{ijk}),\psi(\nu_{jpq})]. 
\end{align*}
Thus, the relations~\eqref{eq::Lie::MonR::generators} are verified on the level of homology $H^{\udot}(\ICGod{d}(n))$.
\end{proof}

The complexity of a graph defines a grading on the $\LL_\infty$-algebra $\ICGod{d+1}(n)$ which is compatible with differential, $\LL_\infty$-maps and the operadic compositions:
$$
\ICGod{d+1}(n) = \oplus_{m}\ICGod{d+1}(n)_{(2m)}.
$$
Note that if each vertex of a graph $\Gamma$ has odd number of outgoing edges then the complexity of this graph is even.
\begin{theorem}
	\label{thm::HICG}
The cohomology of the graded component of $\ICGod{d+1}(n)_{(2m)}$ is concentrated in a unique homolofical degree $2m(1-d)$.
\end{theorem}
\begin{proof}
This theorem is a particular case $S=\emptyset$ of Lemma~\ref{lem::extS::trivalent} proved below.
\end{proof}

Before going to the cumbersome proof of Lemma~\ref{lem::extS::trivalent} let us formulate the main corollary of Theorem~\ref{thm::HICG}, based on the truncated $\LL_\infty$-algebra $\TICG_{d}(n)$ whose bigraded component $\left(\TICG_{d+1}^{\orient,\odd}(n)_{(2m)}\right)^{s}$ of loop order $2m$ and homological degree $s$ is defined by the following {\it truncation} rule:
$$
\left(\TICG_{d+1}^{\orient,\odd}(n)_{(2m)}\right)^{s}:=\begin{cases}
	(\ICGod{d+1}(n)_{(2m)})^s, \text{ if } s < 2m(1-d); \\
		(\ICGod{d+1}(n)_{(2m)})^s_{closed}, \text{ if } s=2m(1-d); \\
		0, \text{ if } s>2m(1-d).
\end{cases}
$$
\begin{corollary}
\label{cor::koszul::Twopois}
There is a sequence of quasi-isomorphisms of $\LL_\infty$-algebras compatible with the operadic structure:
\begin{equation}
\label{eq::trunc::quasi-iso}
\begin{tikzcd}
	\todd_d(n) & \arrow[l,two heads] \TICG_{d+1}^{\orient,\odd}(n) \arrow[r,hook] & \ICGod{d+1}(n).
\end{tikzcd}
\end{equation}
It follows that the quadratic algebra $(\TwoPois_{d}(n))^{\dual}$ and its quadratic dual Lie algebra $\todd_d(n)$ are Koszul for all $d$ and all $n$.
\end{corollary}
\begin{proof}
Thanks to Theorem~\ref{thm::H:Graphsodd} we know that the cohomology of the Chevalley-Eilenberg complex of $\ICGod{d+1}(n)$ is isomorphic to the quadratic commutative algebra $(\TwoPois_{d}(n))^{\dual}$.
What follows, that $\ICGod{d+1}(n)$ computes the Harrison cohomology of the latter quadratic algebra and coincides with the primitive elements in the extension groups of the algebra $(\TwoPois_{d}(n))^{\dual}$. In particular, the basics of Koszul theory predicts that if the cohomology are concentrated in a unique given degree then the considered algebra is Koszul and its graded components are equal to graded components of the quadratic dual algebra. What follows that the cohomology of $\ICGod{d+1}(n)$ is equal to $\todd_{d}(n)$ and the latter quadratic Lie algebra is Koszul. 

The verification of the quasi-isomorphism property of~\eqref{eq::trunc::quasi-iso} repeats the standard one suggested in~\cite{Severa_Willwacher} for~\eqref{eq::trunc::i}.
\end{proof}
\label{sec::proof::koszul}
Let us start to explain the most technical part of this paper which is necessary for the proof of Theorem~\ref{thm::HICG}. 
We will prove a generalization suggested in Lemma~\ref{lem::extS::trivalent} whose proof is based on a collection of consecutive spectral sequence arguments such that the associated graded differential for each particular spectral sequence is the edge contraction that does not break certain symmetries defined combinatorially in terms of graphs. 
Maschke's theorem is one of the key argument that helps us to provide inductive arguments in this type of computation.

\begin{remark}
	For each given pair of integers $n\geq 1$ and $m\geq 0$ 
	the graded component $\ICGod{d}(n)_{(m)}$ spanned by 
 graphs with a given complexity $m=\#(\text{edges})-\#(\text{int. vertices})$ is finite because the number of such graphs is finite. Consequently, all spectral sequences we are dealing with converge because all complexes split into direct sums of finite-dimensional ones.
\end{remark}

\begin{definition}
\label{def::ICG_S}	
For each subset $S\subset [1n]$ consider the subspace $\ICGS{S}\subset \ICGod{d}(n)$ spanned by internally connected graphs $\Gamma$ with the following properties:
\begin{itemize}
		\setlength{\itemsep}{0.4em}
	\item[($\imath$)] for all $s\in S$ the graph $\Gamma$ has at most one edge ending in the external vertex $\extv{s}$;
		\item[($\imath\imath$)] each vertex $v$ connected by an edge to an external vertex $\extv{s}$ with $s\in S$ has more than one (and thus at least $3$) outgoing edges:
		\(	\begin{tikzpicture}[scale=0.5]
			\node[int] (v) at (0,1) {};
			\node[ext] (v1) at (-1,0) {\small{$s$}};
			\node (v0) at (0,-0.2) {\small{$\ldots$}};
			\node (v2) at (1,0) {};
			\draw[-triangle 60] (v) edge (v1)  edge (v0) edge (v2);
		\end{tikzpicture}
		\);
	\item[($\imath\imath\imath$)] for all pairs of different elements $s,t\in S$ the length of the minimal (nondirected) path between $\extv{s}$ and $\extv{t}$ is greater than $2$. In other words, we do not allow an internal vertex with two outgoing edges ending in external vertices from the subset $S$: 
	\(
  	\begin{tikzpicture}[scale=0.5]
  	\node[int] (v) at (0,1) {};
  	\node[ext] (v1) at (-1,0) {\small{$s$}};
  	\node[ext] (v0) at (0,-0.2) {\small{$t$}};
  	\node (v2) at (1,0) {};
  	\draw[-triangle 60] (v) edge (v1)  edge (v0) edge (v2);
  	\draw[red] (-2,1.5) -- (2,-0.5);
  	\draw[red] (-2,-0.5) -- (2,1.5);
  	\end{tikzpicture}.
  	\) 
\end{itemize}
We define the standard vertex splitting differential on $\ICGS{d,S}$ saying that we consider only those vertex splittings that belong to the subspace $\ICGS{d,S}$.
\end{definition}
Let us make the description of a differential a bit more precise.
Properties ($\imath$) and ($\imath\imath\imath$) are preserved by the vertex-splitting differential for $\ICGod{d}$ since the number of incoming edges to the external vertex may not decrease and the length of the minimal path between $\extv{s}$ and $\extv{t}$ may not decrease while one splits a vertex into two connected by an edge.
Therefore, the subspace spanned by graphs that have properties ($\imath$) and ($\imath\imath\imath$) forms a subcomplex of $\ICGod{d}$ called $C$.
 On the other hand, the graphs that do not satisfy the property~($\imath\imath$) are preserved by the vertex splitting differential and assemble into a subcomplex in $C$. The quotient complex is spanned by graphs having all properties ($\imath$)-($\imath\imath\imath$) and is isomorphic to $\ICGS{S}$. In other words, $\ICGS{S}$ is a subquotient of $\ICGod{d}$.
 However, when the set $S$ is empty the properties $(\imath)$-$(\imath\imath\imath)$ do not give any restrictions and $\ICGod{d,\emptyset}=\ICGod{d}$. Thus, it is enough for our purposes to prove the following lemma.
\begin{lemma}
\label{lem::extS::trivalent}	
	The homology of the graded component $\ICGS{S}_{(2m)}$ spanned by graphs of complexity $2m$ is concentrated in a unique homological degree $2m(1-d)$.
\end{lemma}
\begin{proof}	
Note, that if $\Gamma\in\ICGod{d}$ is an internally trivalent tree with $m$ inner vertices then its complexity is equal to $2m$ and the homological degree is equal to $2m(1-d)$.
So the proof consists of consecutive collection of spectral sequences that degenerates in the second terms, such that all cycles that appear on the second pages are given by internally trivalent trees. Unfortunately, the detailed description of cycles leads to deep combinatorics, which we do not want to touch.
However, the proof is the simultaneous (increasing) induction on the number $n$ of externally connected vertices, (decreasing) induction on the cardinality of the subset $S$, and (increasing) induction on the complexity of a graph.

For \emph{the base of induction} ($|S|=n$) we notice that there are no oriented graphs (with no oriented loops) in $\ICGS{[1n]}$. Indeed, each internal vertex of a graph $\Gamma\in \ICGod{d}$ has at least one outgoing edge. Moreover, if this edge ends in an external vertex belonging to $S=[1n]$ then there should exist one more outcome that does not end in an external vertex. 
Thus, a directed path of infinite length that avoids external vertices exist.
This is a contradiction because graphs we are considering have a finite number of vertices and do not contain directed loops.

\emph{(Induction step).}
Suppose $m\in [1n]\setminus S$ and thanks to the induction on $n$ we may suppose that we are dealing with the ideal of graphs connected with $\extv{m}$. For each graph $\Gamma\in \ICGS{S}$ we can define the full subgraph $T_m:=T_m(\Gamma)\subset \Gamma$ containing $\extv{m}$ and the maximal subset of internal vertices of $\Gamma$ satisfying the following list of assumptions for all $v\in T_m\setminus\{m\}$:
\begin{itemize}
	\setlength{\itemsep}{0.4em}
	\item[$(a0)$] there exists a unique (undirected) path $p(v)\subset\Gamma$ without self-intersections connecting internal vertex $v$ and $\extv{m}$;
	\item[$(a1)$] the direction of edges on the path $p(v)$ coming from $\Gamma$ should coincide with the natural direction on the path from $v$ to $\extv{m}$;
	\item[$(a2)$] the number of outgoing edges in $\Gamma$ that start in an internal vertex $v\in T_m$ is equal to $1$.
\end{itemize}

Picture~\eqref{pic::tree::orient} suggests an example of a graph $\Gamma$ with $T_m(\Gamma)$ drawn in red and the connected components of the complementary graphs $\Gamma\setminus T_m$ are drawn in different colors (green, yellow and blue).
We also draw the tree $T_m^{\ddorient}$ that equals $T_m$ together with the incoming half-edges that correspond to connected components that are designed by corresponding coloring and its internally trivalent subtree $T_m^{\orient}$:
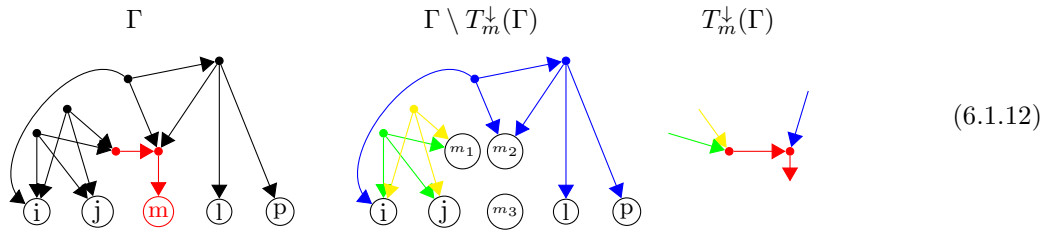
\begin{equation}
	\label{pic::tree::orient}
	\begin{array}{cccc}
		{
			\begin{tikzpicture}[scale=0.8]
				\node[int,red] (w1) at (0,1) {};
				\node[int,red] (w2) at (-0.7,1) {};
				\node[int] (w3) at (-2,1.3) {};
				\node[int] (w4) at (-0.5,2.8) {};
				\node[int] (w5) at (1,2.5) {}; 
				\node[int] (w6) at (-1.5,1.7) {};
				\node[int,red] (w7) at (-0.3,1.8) {};
				\node[ext] (v0) at (2,0) {\small{p}};
				\node[ext] (v1) at (-2,0) {i};
				\node[ext] (v2) at (-1,0) {j};
				\node[ext,red] (v3) at (0,0) {\small{m}};
				\node[ext] (v4) at (1,0) {\small{l}};
				\draw[-triangle 60] (w1) edge[red] (v3);
				\draw[-triangle 60] (w2) edge[red] (w1);
				\draw[-triangle 60] (w3) edge (w2);
				\draw[-triangle 60] (w3) edge (v1) edge (v2);
				\draw[-triangle 60] (w4) edge (w5);
				\draw[-triangle 60] (w4) edge[bend right=90] (v1);
				\draw[-triangle 60] (w4) edge (w7);
				\draw[-triangle 60] (w5) edge (w7);
				\draw[-triangle 60] (w5) edge (v4);
				\draw[-triangle 60] (w5) edge (v0);
				\draw[-triangle 60] (w7) edge[red] (w1);
				\draw[-triangle 60] (w6) edge (v1) edge (v2) edge (w2);
			\end{tikzpicture}
		}
		&
		{
			\begin{tikzpicture}[scale=0.8]
				\node[ext] (w1) at (-0.3,1.8) {\tiny{$m_2$}};
				\node[ext] (w2) at (-0.7,1) {\tiny{$m_1$}};
				\node[int,green] (w3) at (-2,1.3) {};
				\node[int,blue] (w4) at (-0.5,2.8) {};
				\node[int,blue] (w5) at (1,2.5) {}; 
				\node[int,yellow] (w6) at (-1.5,1.7) {};
				\node[ext] (v0) at (2,0) {\small{p}};
				\node[ext] (v1) at (-2,0) {i};
				\node[ext] (v2) at (-1,0) {j};
				\node[ext] (v3) at (0,0) {\tiny{$m_3$}};
				\node[ext] (v4) at (1,0) {\small{l}};
				\draw[-triangle 60,green] (w3) edge (w2);
				\draw[-triangle 60,green] (w3) edge (v1) edge (v2);
				\draw[-triangle 60,blue] (w4) edge (w5);
				\draw[-triangle 60,blue] (w4) edge[bend right=90] (v1);
				\draw[-triangle 60,blue] (w4) edge (w1);
				\draw[-triangle 60,blue] (w5) edge (w1);
				\draw[-triangle 60,blue] (w5) edge (v4);
				\draw[-triangle 60,blue] (w5) edge (v0);
				\draw[-triangle 60,yellow] (w6) edge (v1) edge (v2) edge (w2);
			\end{tikzpicture}
		}
		&
		{
			\begin{tikzpicture}[scale=0.8]
				\node[int,red] (w1) at (0,1) {};
				\node[int,red] (w2) at (-1,1) {};
				\node[int,red] (w7) at (-0.3,1.8) {};
				\coordinate (w3) at (-2,1.3);
				\coordinate (w6) at (-1.5,1.7);
				\coordinate (w5) at (0.3,2.5); 
				\coordinate (v3) at (0,0.5);
				\node[ext,red] (v3) at (0,0) {\small{m}};
				\draw[-triangle 60] (w1) edge[red] (v3);
				\draw[-triangle 60] (w2) edge[red] (w1);
				\draw[-triangle 60] (w7) edge[red] (w1);
				\draw (w6) edge[-triangle 60,yellow] (w2);
				\draw (w3) edge[-triangle 60,green] (w2);
				\draw (w5) edge[-triangle 60,blue] (w7);
			\end{tikzpicture}
		} 
	&
		{
			\begin{tikzpicture}[scale=0.8]
				\node[int,red] (w1) at (0,1) {};
				\node[int,red] (w2) at (-1,1) {};
				\coordinate (w3) at (-2,1.3);
				\coordinate (w6) at (-1.5,1.7);
				\coordinate (w5) at (0.3,2); 
				\coordinate (v3) at (0,0.5);
				\node[ext,red] (v3) at (0,0) {\small{m}};
				\draw (w1) edge[-triangle 60,red] (v3);
				\draw (w2) edge[-triangle 60,red] (w1);
				\draw (w6) edge[-triangle 60,yellow] (w2);
				\draw (w3) edge[-triangle 60,green] (w2);
				\draw (w5) edge[-triangle 60,blue] (w1);
			\end{tikzpicture}
		}
	 \\
	&  &  &  \\
		\Gamma  & \Gamma\setminus T_m(\Gamma) & T_m^{\ddorient}(\Gamma) & T_m^{\orient}(\Gamma)
	\end{array}
\end{equation}
Let us work out some properties of $T_m$ that immediately follow from the definition:\\
First, we notice that $T_m$ is a directed tree with root $\extv{m}$ because each internal vertex has a unique (directed) path coming to the root and this property is satisfied by all vertices in this path.
\\
Second, there are no edges in $\Gamma$ that go from $T_m$ to the complement thanks to assumption~($a2$).  
\\
Third, we claim that each internally connected component $\Gamma'$ of the complement subgraph $\Gamma\setminus T_m$ is connected with exactly one vertex of $T_m$.
Indeed, suppose that there are two different vertices $v$ and $v'$ connected with $\Gamma'$, then there exists an (undirected) path $p(v,v')$ without selfintersections between $v$ and $v'$ that goes through internal vertices of $\Gamma'$. On the other hand, we have paths $p(v)$ and $p(v')$ connecting $v$ (respectively $v'$) with $\extv{m}$. Without loss of generality we suppose that $p(v')$ does not contain $v$ (if this happens then vice versa $p(v)$ does not contain $v'$).
The concatenation of paths $p(v,v')$ and $p(v')$ defines a path (without selfintersections) that starts at $v$ and ends in $\extv{m}$ and differs from $p(v)$. So we end up with a contradiction to our assumption~($a0$).
\\
Fourth, the vertex-splitting differential may not increase the number of connected components of $\Gamma\setminus T_m$. We see that in Example~\eqref{pic::tree::orient} vertex splitting in  $\extv{j}$ will connect green and yellow components, and vertex splitting in the red vertex that is denoted by $\extv{m_2}$ in the middle picture may either keep the number of components to be the same or decrease them to one.

Let us denote by $T_m^{\ddorient}(\Gamma)$ the tree $T_m$ together with the half-edges coming to it that correspond to the connected components of the complement $\Gamma\setminus T_m$. As we explained each connected component is connected to a unique vertex of $T_m$ and thus, the notion of a half-edge is well-defined even when there are several edges of $\Gamma$ coming to one vertex of $T_m$. (See the pictorial example~\ref{pic::tree::orient}.) 
Each inner vertex of $\Gamma$ has valency at least $3$.
Therefore, each inner (nonroot) vertex of $T_m^{\ddorient}$ may have less then $3$ vertices in a very specific case -- when it has a unique incoming half-edge that corresponds to several edges coming from a particular connected component $\Gamma'\subset(\Gamma\setminus T_m)$. Let us remove the following vertex $v$ from $T_m^{\ddorient}$ and attach it to the connected component $\Gamma'$. 
Note that this operation does not change the number of connected components of $\Gamma\setminus T_m$ and decreases the number of bivalent vertices in $T_m^{\ddorient}$. If we do this sufficiently many time we end up with a decomposition of $\Gamma$ into the disjoint union of $k$ internally connected graphs $\Gamma_1,\ldots,\Gamma_k$ and the maximal tree $T_m^{\orient}\subset T_m^{\ddorient}$ with nonbivalent vertices and with $k$ incoming half-edges marked by these graphs.

Consider the increasing filtration $\calF^k$ of $\ICGS{S}$ by the number $k$ of internally connected components of the graph $\Gamma\setminus T_m(\Gamma)$. 
Below we will show by induction that the  cohomology of the associated graded complex are represented by internally trivalent trees.

First, let us work out carefully the homology of the subcomplex $\calF^1\ICGS{S}$ spanned by those $\Gamma$ such that $\Gamma\setminus T_m$ has a unique connected component $\Gamma'$.
Suppose that $v$ is  a leaf of $T_m$ it has at most one outgoing edge and thus it should have a nonzero amount of incoming edges. These edges should belong to the unique connected component $\Gamma'$ of $\Gamma\setminus T_m$. However, as we explained earlier graph $\Gamma'$ is connected with exactly one vertex of $T_m$.  Thus $T_m$ has only one leaf $v$. 
Therefore, the vertices that differ from $v$ and $\extv{m}$ are bivalent in $T_m$. However, there are no other edges coming or outgoing to them in the full graph $\Gamma$ and they remain bivalent in $\Gamma$. Since bivalent vertices are forbidden in $\ICGS{d}$ we see that either $T_m$ is equal to $\extv{m}$ or consists of one edge $\oneedge{v}{m}$.
  
This observation induces the decomposition of the complex $\calF^1\ICGS{S}$ into two subspaces  
$$\calF^1\ICGS{S} = \calF^1\ICGS{S}_{0} \oplus \calF^1\ICGS{S}_{1}$$ 
where the extra rightmost lower index corresponds to the number of edges in $T_m$. 
Consider the homotopy $h:\calF^1\ICGS{S}_{1}\twoheadrightarrow \calF^1\ICGS{S}_{0}$ to the  first differential in the corresponding spectral sequence given by contraction of the edge $\oneedge{v}{m}$ if allowed. 
$$
\begin{tikzcd}
\begin{tikzpicture}[scale =0.5]
\node[ext] (u) at (0,0) {\small{m}};
\coordinate (u1) at (-1,2.5);
\coordinate (u2) at (0,2);
\coordinate (u3) at (1,2.5);
\draw (u1) edge[-triangle 60] (u);
\draw (u3) edge[-triangle 60] (u);
\node (w) at (0,3) {$\Gamma$};
\draw[dotted] (w) circle  [radius=1];
\end{tikzpicture}
 \arrow[r,"d_0",mapsto,shift left=1ex]
&
\begin{tikzpicture}[scale =0.5]
\node[ext] (v) at (0,-1) {\small{m}};
\node[int] (u) at (0,0.5) {};
\coordinate (u1) at (-1,2.5);
\coordinate (u2) at (0,2);
\coordinate (u3) at (1,2.5);
\draw (u1) edge[-triangle 60] (u);
\draw (u3) edge[-triangle 60] (u);
\node (w) at (0,3) {$\Gamma$};
\draw[dotted] (w) circle  [radius=1.1];
\draw (u) edge[-triangle 60] (v);
\node (a) at (0.5,0.5) {$v$};
\end{tikzpicture};
 \arrow[l,"h",mapsto,shift left=1]
& 
0 & 
\begin{tikzpicture}[scale =0.5]
\node[ext] (v) at (0,-1) {\small{m}};
\node[int] (u) at (0,0.5) {};
\coordinate (u1) at (-1,2.5);
\coordinate (u2) at (0,2);
\coordinate (u3) at (1,2.5);
\draw (u1) edge[-triangle 60] (u);
\draw (u3) edge[-triangle 60] (u);
\coordinate (b1) at (-1.5,0);
\coordinate (b2) at (1.5,0);
\draw (u) edge[-triangle 60] (b1);
\draw (u) edge[-triangle 60] (b2);
\node (w) at (0,3) {$\Gamma$};
\draw[dotted] (w) circle  [radius=1.1];
\draw (u) edge[-triangle 60] (v);
\node (a) at (0.5,0.5) {$v$};
\end{tikzpicture}
 \arrow[l,"h",mapsto]
\end{tikzcd}.
$$

 The kernel of this surjection is spanned by graphs having more than one outgoing edge from the unique internal vertex $v\in T_m^{\orient}$. 
 There are two possibilities with these graphs. Either there is an edge starting in $v$ and ending in one of the external vertices $\extv{s}$ with $s\in S$ or there are no such outgoing edges.
 The absence of such an edge is precisely the assumption $(\imath\imath\imath)$ from Definition~\ref{def::ICG_S} verified for the external vertex $\extv{m}$. Two other assumptions are already verified. What follows, that 
 the latter subset spans the subcomplex isomorphic to $\ICGS{{S\sqcup\{m\}}}$ and we denote the quotient complex by $K$. 
 The homology of $\ICGS{{S\sqcup\{m\}}}_{(l)}$ with a given complexity $l$ are concentrated in a unique homological degree thanks to the decreasing induction with respect to the cardinality of $S$. So it remains to describe the cohomology of the quotient complex $K$.
 
 Consider the filtration $K=\oplus_{s\in S} K_s$ by the minimal number of the external vertex $\extv{s}$ connected by an edge with $v$. The associated graded complex $K_s$ admits an additional decomposition as a bicomplex $K_s=K_s^{3}\oplus K_s^{>3}$, where $K_s^{3}$ is spanned by graphs with the vertex $v$ trivalent. Note that if $v$ is trivalent then $v$ has $3$ outgoing edges, one ending in $\extv{s}$ another ending in $\extv{m}$ and the latter connects $v$ with the remaining part of the graph: $\ttriple{v}{s}{m}{\ldots}.$ 
Consider the homotopy $h'$ to the associated graded differential given by contraction of the unique edge connecting $v$ and the remaining part of the graph:
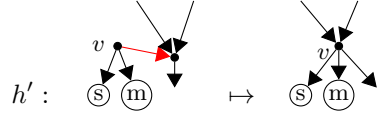
\[
h': \quad
\begin{tikzpicture}[scale =0.5]
\node (w) at (-1,1.3) {\small{$v$}};
\node[int] (w1) at (-0.5,1.3) {};
\node[int] (w2) at (1,1) {};
\node[ext] (v1) at (-1,0) {\small{s}};
\node[ext] (v2) at (0,0) {\small{m}};
\coordinate (u1) at (0,2.5);
\coordinate (u2) at (1.5,2.5);
\coordinate (u3) at (1,0.2);
\draw[-triangle 60] (w1) edge (v1) edge (v2) edge[red] (w2);
\draw (w2) edge[-triangle 60] (u3);
\draw (u1) edge[-triangle 60] (w2);
\draw (u2) edge[-triangle 60] (w2);
\end{tikzpicture}
\quad \mapsto \quad
\begin{tikzpicture}[scale =0.5]
\node (w) at (-0.4,1.1) {\small{$v$}};
\node[int] (w1) at (0,1.3) {};
\node[ext] (v1) at (-1,0) {\small{s}};
\node[ext] (v2) at (0,0) {\small{m}};
\coordinate (u1) at (-1,2.5);
\coordinate (u2) at (0.5,2.5);
\coordinate (u3) at (1,0.2);
\draw[-triangle 60] (w1) edge (v1) edge (v2);
\draw (w1) edge[-triangle 60] (u3);
\draw (u1) edge[-triangle 60] (w1);
\draw (u2) edge[-triangle 60] (w1);
\end{tikzpicture}
\]
The homotopy $h'$ defines a bijection between graphs spanning $K_s^{3}$ and $K_s^{>3}$ except one particular case when the remaining part of the graph consists of one external vertex $\extv{t}$ and  the cohomology of $K_s$ is spanned by the simplest trivalent graph 
$		\begin{tikzpicture}[scale=0.5]
\node (w) at (-0.4,1) {\small{$v$}};
\node[int] (v) at (0,1) {};
\node[ext] (v1) at (-1,0) {\small{$s$}};
\node[ext] (v0) at (0,-0.2) {\small{$m$}};
\node[ext] (v2) at (1,0) {\small{$t$}};
\draw[-triangle 60] (v) edge (v1)  edge (v0) edge (v2);
\end{tikzpicture}
$. 
What follows, in particular, is that the homology of $K$ is represented by trivalent graphs and, consequently, the cohomology of $\calF^{1}\ICGS{S}_{(l)}$ are concentrated in a unique homological degree.

Next, we must work out the case when $\Gamma\setminus T_m$ has more than one connected component. 
We claim that we have an isomorphism of complexes:
\begin{equation}
\label{eq::grF::ICGS}
\calF^k\ICGod{d,S}/\calF^{k-1} \ICGod{d,S} \simeq \LL_{\infty}(k)\otimes_{S_k} \left(\calF^1\ICGod{d,S}\right)^{\otimes k} 
\end{equation}
that we already visualized on the level of graphs while describing the internal subtree $T_m^{\orient}(\Gamma)$ and the connected components $\Gamma_1,\ldots,\Gamma_k$ of the complement $\Gamma\setminus T_m$. 
The tensor factor $\LL_{\infty}$ in~\eqref{eq::grF::ICGS} corresponds to the tree $\bar{T}_m^{\orient}$ and the remainig tensor multiples correspond to the graphs $\Gamma_1$,\ldots,$\Gamma_k$.

The homology of $\LL_{\infty}$ are represented by trivalent trees because the operad $\Lie$ is Koszul. For each given collection of complexities the homology of the remaining tensor multiples are concentrated in a unique homological degree, mimicking the fact that there exists a filtration such that the cycles for the associated graded are represented by internally trivalent trees as we worked out the case of one connected component.
Note also that Isomorphism~\eqref{eq::grF::ICGS} preserves the complexity grading. It follows that for $k>1$ the complexity of each graph $\Gamma_i$  is strictly less than the complexity of the corresponding graph $\Gamma$ which is also relevant for particular inductive computations with cocycles.
 \end{proof}

\begin{remark}
Let us denote by $\LICGod{d}(n)$ the $\LL_\infty$-ideal of $\ICGod{d}(n)$ spanned by graphs connected by an edge with the latter external vertex $\extv{n}$. Following the aforementioned proof one can define a collection of subcomplexes $\LICGod{{d,S}}\subset \LICGod{d}$, describe a filtration such that the homology for the associated graded is equal to the homology of $\LICGod{{d,S}}$ but the cycles are represented by internally trivalent trees with a unique edge connected with the last external vertex $\extv{n}$.
\end{remark}

\subsection{Why the map $\todd_{d}(n)\hookrightarrow \tkd_{d+1}(n)$ is an embedding}
\begin{lemma}
	The map that forgets the orientation of edges in a graph defines a map of dg-operads in the category of $\LL_{\infty}$-algebras:
\begin{equation}
	\label{eq::forget::orient}
	\begin{tikzcd}
		f:\ICGod{d} \arrow[r] & \ICG_{d}
	\end{tikzcd}.
\end{equation}	
\end{lemma}
\begin{proof}
	First, let us explain that $f$ is a map of complexes. Pictorially, the differential is the vertex-splitting differential in both complexes.
	Consider an oriented graph $\Gamma\in\ICGod{d}$ and a 
 vertex $v$ that has $2l+1$ outgoing edges. Suppose we are looking a component of the differential that splits $v$ into two vertices $v'$ and $v''$. We may suppose that $v'$ has $2l'$ outputs and $v''$ has $2l''+1$ outputs that come from the $2l+1$ outputs of $v$ ($l'+l''=l$). Therefore, the orientation of the new edge between $v'$ and $v''$ is uniquely defined -- it should go from $v'$ to $v''$ in order to satisfy the condition that all vertices have an odd number of outputs. Hence it is clear that the map $f$ forgetting edge directions intertwines the differentials. The compatibility with the operad structure is obvious. While the compatibility with the $\LL_\infty$ structure is equivalent to having the same orientation-forgetting map between their homological Chevalley-Eilenberg complexes
 $$
 	\begin{tikzcd}
 	f:\left(\Graphs_d^{\orient,\odd}(n)\right)^{\dual} \arrow[r] & \left(\Graphs_d(n)\right)^{\dual}
 	\end{tikzcd}
 $$
 where the pictorial description of the differential is also the vertex-splitting.
\end{proof}
The map $f$ does not change the underlying graph, therefore, in particular it preserves the complexity of a graph (see Definition~\ref{def::complexity}).
Note that the complexity of any graph from $\ICGod{d}$ is even, while the complexity of a graph from $\ICG_d$ may be arbitrary. However, with any given value of complexity $l$ the nontrivial cohomology of the graded component  $\ICGod{d}(n)_{(2l)}$ and $\ICG_d(n)_{(2l)}$ are concentrated in the same unique degree and are equal to corresponding graded components of $\todd_{d-1}(n)$ and $\tkd_d(n)$ respectively.
\begin{lemma}
\label{lem::f=xi}
The following diagram of maps of operads in the category of $\LL_{\infty}$-algebras is commutative:
\begin{equation}
\label{diag::orient::nonorient}
\begin{tikzcd}
	& \ICGod{d+1} \arrow[rrr,"f"] & & & \ICG_{d+1} & \\
\TICG_{d+1}^{\orient,\odd} \arrow[rrrrr, "f"]  
\arrow[ru] \arrow[dr]&   & &  && 
\TICG_{d+1} \arrow[lu] \arrow[dl]	\\
& \todd_{d} 
\arrow[rrr, "\xi"]
& & & \tkd_{d+1}. &  
\end{tikzcd}
\end{equation}
Where the lower map $\xi$ is defined by the following assignment on the generators: 
$\nu_{ijk}\mapsto [t_{ij},t_{jk}].$
\end{lemma}
\begin{proof}
Note that the described above map $\todd_d\to\tkd_d$ is indeed a map of Lie algebras as verified in~\cite{Etingof_Rains}.
The truncated complexes $\TICG_{d+1}^{\orient,\odd}$ and $\TICG_{d+1}$ are subcomplexes of the corresponding complexes of internally connected graphs, so the upper square of the diagram is commutative because the image $f(\TICG_{d+1}^{\orient,\odd})$ belongs to $\TICG_{d+1}$.
Moreover, we know that all diagonal (non-horisontal) arrows in Diagram~\eqref{diag::orient::nonorient} are quasi-isomorphisms and the truncation maps $\pi:\TICG_{d+1}^{\orient,\odd}\to\todd_{d}$ and $\pi:\TCG_{d}\to \tkd_d$ map exact elements to zero. Thus, the commutativity of the lower square is in fact equivalent to the commutativity of the corresponding cohomology diagram
\[
\begin{tikzcd}
   H(\TICG_{d+1}^{\orient,\odd})  \arrow{rrr}  \arrow[dr,"\cong"] 
   & & &
   H(\TICG_{d+1}) \arrow[dl,"\cong"'] 
   \\
&   \todd_{d} \ar{r} & \tkd_{d+1}. &
\end{tikzcd}
\]
This is a diagram of graded Lie algebras, and it suffices to check commutativity on the generators of the Lie algebra $H(\TICG_{d+1}^{\orient,\odd})\cong \todd_{d}(n)$. We know that the graph that represents the generator $\nu_{ijk}$ is the tripod graph drawn in Diagram~\eqref{diag::generators::map} below. When we forget the orientation of arrows in this graph we identify it with the commutator $[t_{ij},t_{jk}]\in\tkd_{d+1}(n)$, as one can easily see from the following pictorial description:
\begin{equation}
\label{diag::generators::map}
\begin{tikzcd}
{\begin{tikzpicture}[scale=0.5]
			\node[int] (v) at (0,1) {};
			\node[ext] (v1) at (-1,0) {\small{$i$}};
			\node[ext] (v0) at (0,-0.2) {\small{$j$}};
			\node[ext] (v2) at (1,0) {\small{$k$}};
			\draw[-triangle 60] (v) edge (v1)  edge (v0) edge (v2);
	\end{tikzpicture}} 
 \arrow[rr,"f",mapsto] \arrow[dr,mapsto] & & 
{\begin{tikzpicture}[scale=0.5]
		\node[int] (v) at (0,1) {};
		\node[ext] (v1) at (-1,0) {\small{$i$}};
		\node[ext] (v0) at (0,-0.2) {\small{$j$}};
		\node[ext] (v2) at (1,0) {\small{$k$}};
		\draw (v) edge (v1)  edge (v0) edge (v2);
\end{tikzpicture}} \arrow[r,equal]
& \left[\begin{tikzpicture}[scale=0.5]
	\node[ext] (v0) at (-0.5,0) {\small{$i$}};
	\node[ext] (v1) at (0.5,0) {\small{$j$}};
	\draw (v0) edge (v1);
\end{tikzpicture},
\begin{tikzpicture}[scale=0.5]
	\node[ext] (v0) at (-0.5,0) {\small{$j$}};
	\node[ext] (v1) at (0.5,0) {\small{$k$}};
	\draw (v0) edge (v1);
\end{tikzpicture}
\right] 
\arrow[dl,mapsto]
\\
& \nu_{ijk} \arrow[r,mapsto]  & \left[t_{ij},t_{jk}\right] & 
\end{tikzcd}
\end{equation}
\end{proof}

The map $\xi_n:\todd_{d}(n)\to \tkd_{d+1}(n)$ from Diagram~\eqref{diag::orient::nonorient} fits into the following map of short exact sequences of Lie algebras ($\LL_\infty$-algebras) (defined over $\QQ$), where by $\LICGod{d}(n)$ we denote the $\LL_{\infty}$-ideal of $\ICGod{d}(n)$ spanned by graphs connected with the last vertex $\extv{n}$ and by $\LICG_d(n)$ the corresponding $\LL_{\infty}$-ideal in $\ICG_d(n)$: 
\begin{equation}
\label{eq::t::last}	
\begin{tikzcd}
0 \arrow[r] & \LICGod{d+1}(n) \arrow[ddd,bend right =7em,"f"]  \arrow[d,Rightarrow,"H()"] \ar[r,hook] & \ICGod{d+1}(n) \ar[r, two heads] \arrow[d,Rightarrow,"H()"] \arrow[ddd,bend right =7em,"f"] & \ICGod{d+1}(n-1) \arrow[r] \arrow[d,Rightarrow,"H()"]  \arrow[ddd,bend right =7em,"f"] & 0 \\
	0 \arrow[r] & \ker \pi_n \ar[r,hook] \ar[d]  & \todd_d(n)  \ar[r,twoheadrightarrow,"\pi_n"] \ar[d,"\xi_n"]   & \todd_{d}(n-1) \ar[r] \ar[d,"\xi_{n-1}"] & 0 \\
	0 \arrow[r] & \Lie(t_{1n},\ldots,t_{n-1 n}) \ar[r,hook]  & \tkd_{d+1}(n)  \ar[r,twoheadrightarrow]   & \tkd_{d+1}(n-1) \ar[r]  & 0 \\
0 \arrow[r] & \LICG_{d+1}(n)  \arrow[u,Rightarrow,"H()"] \ar[r,hook] & \ICG_{d+1}(n) \ar[r, two heads] \arrow[u,Rightarrow,"H()"]  & \ICG_{d+1}(n-1) \arrow[r] \arrow[u,Rightarrow,"H()"]  & 0 \\	
\end{tikzcd}
\end{equation}
The ordinary arrows in Diagram~\eqref{eq::t::last} are morphisms of $\LL_{\infty}$ (respectively Lie algebras) and the homology functors $H()$ are denoted by double arrows.
\begin{definition}
\label{def::LLICGodd}
The set of internally connected directed graphs connected to the external vertex $\extv{n}$ satisfying the following properties:
\begin{itemize}
		\setlength{\itemsep}{0.05em}
	\item each internal vertex is at least trivalent;
	\item the number of outputs in each internal vertex is odd;
	\item  \emph{we allow outgoing edges for the latter external vertex $\extv{n}$}.
	\\ All other external vertices may have only incoming edges.
	\item we allow directed cycles.
\end{itemize}
assemble into an $\LL_{\infty}$-algebra denoted by $\LLICG{d}(n)$ with the vertex splitting differential and the standard  Lie bracket one uses for $\ICGod{d}$.
\end{definition}
Note, that the orientation-forgetting map $f$ from the $\LL_\infty$-subalgebra of odd internally connected oriented graphs $\LICGod{d}(n)$ to the $\LL_\infty$-subalgebra $\LICG_d(n)$ of internally connected graphs factors through this intermediate $\LL_\infty$-algebra:
$$
\begin{tikzcd}
	\LICGod{d}(n) \arrow[r,hook] \arrow[rr,bend right = 2em,"f"] & \LLICG{d}(n) \arrow[r,"\bar{f}"] & \LICG_d(n).
\end{tikzcd}
$$
Where the first arrow is just the simple embedding and $\bar{f}$ is the orientation-forgetting map, which is the map of complexes for the same reasons as the map $f$-is a map of complexes.
\begin{lemma}
\label{lm::forget::orient::ICG}	
	The map $\bar{f}:\LLICG{d}(n) \to \LICG_d(n)$ that forgets the orientation of edges is a quasi-isomorphism. In particular, its homology is isomorphic to the free Lie algebra generated by cocycles $t_{in}:=\left[\begin{tikzpicture}[scale=0.4]
		\node[ext] (v1) at (-1,0) {\small{i}};
		\node[ext] (v2) at (1,0) {\small{n}};
		\draw[-triangle 60] (v2) edge (v1);
	\end{tikzpicture}\right]. $ 
\end{lemma}
\begin{proof}
The proof that $\bar{f}$ is a quasi-isomorphism repeats the one given in~\cite[Appendix B]{Severa_Willwacher} where it is proved that $\LICG_{d}(n)$ is quasi-isomorphic to the free Lie algebra generated by $\{t_{in}\colon i=1..n-1\}$.

Consider the gradings of $\LLICG{d}(n)$ and $\LICG_d(n)$:
\[\LLICG{d}(n) = \LLICG{d}(n)^{1}\oplus \LLICG{d}(n)^{\geq 2};
\quad \LICG_{d}(n) = \LICG_{d}(n)^{1}\oplus \LICG_{d}(n)^{\geq 2};
\]
with $\LLICG{d}(n)^{1}$ (respectively $\LICG_{d}(n)$) consisting of graphs that have exactly one edge (either incoming or outgoing) connected to the vertex $\extv{n}$ and $\LLICG{d}(n)^{\geq 2}$ contains at least two edges attached to $\extv{n}$. The differential is the sum of two differentials -- the one that decreases the grading and the one that keeps it. 
The orientation-forgetting map $\bar{f}$ preserves the grading and we want to suggest the inductive arguments why the corresponding spectral sequences are isomorphisms after the second term.
Indeed, the first term of the corresponding spectral sequences (where $d_0$ -- decreases the grading)  consists of  the subspace $\LLICG{d}(n)^{1}_{disc}\subset \LLICG{d}(n)^{1}$
(resp. $\LICG{d}(n)^{1}_{disc}\subset \LICG_{d}(n)^{1}$)
 spanned by graphs that become internally disconnected after removing (contracting) the unique edge attached to $\extv{n}$.
Note that the loop order of a graph $\Gamma\in \LLICG{d}(n)^{1}_{disc}$ is strictly greater than the loop order of each connected component of $\Gamma$ with a removed/contracted edge apart from one case when $\Gamma$ consists of one edge between two external vertices $\extv{i}$ and $\extv{n}$.
For the bigger graphs we have isomorphism between first terms of spectral sequences and truncated Chevalley-Eilenberg complexes what induces
$$
\begin{tikzcd}
\LLICG{d}(n)^{1}_{disc}\simeq C^{CE}_{\ldot\geq 2}(\LLICG{d}(n)) \arrow[r,"\bar{f}"] &
\LICG_{d}(n)^{1}_{disc}\simeq C^{CE}_{\ldot\geq 2}(\LICG_{d}(n)) 
\end{tikzcd}
.$$
The induction by the loop order shows that the map $\bar{f}$ is a quasiisomorphism.
Recall that the main remaining argument of~\cite{Severa_Willwacher} consists of  the cohomology computation of a truncated free Lie algebra:
\begin{equation}
	\label{eq::CE::Free}
	H_{CE}^i(\Lie(V)) =\begin{cases}
		V, \text{ if } i=1, \\
		0, \text{ if } i\neq 1
	\end{cases}  \ 
	\Rightarrow
	H^{i}\left(C_{CE}^{\udot\geq 2} (\Lie(V))\right) = 
	\begin{cases}
		\Lie(V)_{+} = \Lie(V)/V , \text{ if } i = 2, \\
		0, \text{ if } i>2.
	\end{cases}.  
\end{equation}
Therefore, we end up with the conclusion that the map $\bar{f}$ is a quasi-isomorphism.
\end{proof}

Following the ideas we used for the proof of Theorem~\ref{thm::HICG} let us introduce a collection of intermediate combinatorial complexes.
For each subset $S\subset[1 n-1]$ let us define a subquotient complex $\LLICG{d,S}(n)$ of $\LLICG{d}(n)$ spanned by graphs that appear in $\LLICG{d}(n)$ yielding the properties $(\imath)$--$(\imath\imath\imath)$ of Definition~\ref{def::ICG_S}. Note that we exclude the external vertex $\extv{n}$ from the set $S$ because vertex $\extv{n}$ may have outgoing edges what is forbidden for $s\in S$. For the induction base we want to add the simple graph 
$\left[\begin{tikzpicture}[scale=0.4]
	\node[ext] (v1) at (-1,0) {\small{s}};
	\node[ext] (v2) at (1,0) {\small{n}};
	\draw[-triangle 60] (v2) edge (v1);
\end{tikzpicture}\right]$ 
to our family, so we weaken the property $(\imath\imath)$ saying that vertex $v$ connected by an edge with an external vertex $\extv{s}$ with $s\in S$ may coincide with an external vertex $\extv{n}$. 
\begin{lemma}
\label{lem::ICGS::embedding}
\begin{enumerate}
\item
	The cohomology of the graded component $\LLICG{d,S}(n)_{(l)}$ of graphs of complexity $l$ is concentrated in a unique homological degree $l(1-d)$.  
\item The embedding of complexes $\LICGod{d,S}(n)\hookrightarrow\LLICG{d,S}(n)$ induces the embedding on the homology $H(\LICGod{d,S}(n))\hookrightarrow H(\LLICG{d,S}(n))$.
\end{enumerate}
\end{lemma}
\begin{proof}
	The strategy of the proof of this lemma repeats the one we suggested for Lemma~\ref{lem::extS::trivalent}. Namely, 
 we consider the same filtrations and same spectral sequences that we defined in the proof of Lemma~\ref{lem::extS::trivalent}.
 The decreasing induction by the size of $S$, increasing induction on the complexity of a graph and increasing induction on the number $n$ of external vertices shows that the corresponding spectral sequences degenerates in the second terms and the cocycles representing the cohomology classes for the associated graded differential are spanned by internally trivalent trees. 
This will be sufficient to prove the first item of Lemma~\ref{lem::ICGS::embedding}.
One the other hand, the careful look on the recursive procedure of the internally trivalent trees representing cocycles in the associated graded complexes explains the embedding of the cohomology of these complexes.
 
Let us explain the details of the mentioned above strategy.
First, let us work out the inductive proof of the first item.
The induction step completely repeats the one suggested in
the proof of Lemma~\ref{lem::extS::trivalent}. For each given set $S$ we find an external vertex $m\notin S\cup \{n\}$ and we define the maximal subgraph $T_m(\Gamma)$ which is a rooted directed tree with $\extv{m}$ being a root following properties $(a0)-(a2)$.
We consider the filtration $\calF^{\udot}$ by the number of connected components 
of $\Gamma\setminus T_m(\Gamma)$ and show that:
\begin{itemize}
    \item 
The subcomplex $\calF^1\LLICG{d,S}$ is an extension of a complex $\LLICG{d,S\sqcup\{m\}}$ and the direct sum of almost acyclic complexes $K_s$ with $s\in S$.
\item The associated graded component is isomorphic to the tensor product:
\begin{equation}
\label{eq::filt::extv}
\calF^{k}/\calF^{k-1}\LLICG{d,S} \simeq 
\LL_\infty(k)\otimes_{S_k}(\calF^1\LLICG{d,S})^{\otimes k}.
\end{equation}
Here the complexity of each tensor multiple of the right-hand side is strictly less then the complexity of a graph representing an element from the left-hand side and the induction by complexity and koszulness of the $\Lie$ operad finishes inductive proof.
\end{itemize}
The nontrivial cocycles that appear in the explained above recursive procedure take their origin from one of three following possibilities:
\begin{itemize}
    \item from the $\Lie(k)$ which is equal to the homology of $\LL_\infty(k)$ in the right-hand side of~\eqref{eq::filt::extv};
   \item from the homology of the complex $K_s$ that may be equal to the tripod graphs 
$		\begin{tikzpicture}[scale=0.5]
\node (w) at (-0.4,1) {\small{$v$}};
\node[int] (v) at (0,1) {};
\node[ext] (v1) at (-1,0) {\small{$s$}};
\node[ext] (v0) at (0,-0.2) {\small{$m$}};
\node[ext] (v2) at (1,0) {\small{$t$}};
\draw[-triangle 60] (v) edge (v1)  edge (v0) edge (v2);
\end{tikzpicture}
$,
\item or to  the graph $\left[\begin{tikzpicture}[scale=0.4]
	\node[ext] (v1) at (-1,0) {\small{s}};
	\node[ext] (v2) at (1,0) {\small{n}};
	\draw[-triangle 60] (v2) edge (v1);
\end{tikzpicture}\right]$.
\end{itemize}
The latter appear, whenever the set of external vertices from $S$ connected to a graph $\Gamma$ consists of one element and the complexity of a graph is small. The difference with the case $\LICGod{d,S}$ is precisely within the graphs with one edge connecting two external vertices. In other words, all cohomology that appear through the recursive description of $\LICGod{d,S}$ do appear in the recursive description of the cohomology $\LLICG{d,S}$ what ensures the embedding of the cohomology (item 2. of our Lemma~\ref{lem::ICGS::embedding}).

There are two extremal case for the base of induction that we have to verify and this is the main place that slightly differs from the one in the proof of Lemma~\ref{lem::extS::trivalent}. However, the strategy repeats the one we use for the induction step.

First, let us discuss the case $S=[1n-1]$. 
Let us show (using the increasing induction on complexity $l$ and the number $n$ of external vertices) that the complex $\LLICG{d,[1n-1]}(n)_{(l)}$ is acyclic for all $n$ and all $l$ 
except two extremal cases: ($n=2\ \&\ l=1$) or ($n=1\ \& \ l=0$).
In these exceptional cases we have the cocycles with zero internal vertices: 
$\left[\begin{tikzpicture}[scale=0.4]
	\node[ext] (v1) at (-1,0) {\small{1}};
	\node[ext] (v2) at (1,0) {\small{2}};
	\draw[-triangle 60] (v2) edge (v1);
\end{tikzpicture}\right]$ and  $\extv{n}$ correspondingly. 
Let us assign to each graph $\Gamma\in\LLICG{d,[1n-1]}(n)$ the rooted tree $\bar{T}_n(\Gamma)$. The tree $\bar{T}_n(\Gamma)$ is the maximal full subgraph of $\Gamma$ which is a rooted tree with $\extv{n}$ being a root. In other terms, vertex $v$ of $\Gamma$ belongs to $\bar{T}_n(\Gamma)$ iff there exists a unique (undirected) path without self intersections connecting $v$ and $\extv{n}$. 
Consider the decreasing filtration by the number of connected components of $\Gamma\setminus \bar{T}_n(\Gamma)$. Following the same ideas as above we can define an operadic tree $\bar{T}_n^{\orient}(\Gamma)$ whose leaves are indexed by connected components of $\Gamma\setminus \bar{T}_n(\Gamma)$ and all whose internal vertices are at least trivalent.\footnote{ 
The orientation of edges of the tree $\bar{T}_n^{\orient}(\Gamma)$ does not matter because it is uniquely defined by $\Gamma$.
Indeed, we can set up the direction of an edge of $\bar{T}_n^{\orient}(\Gamma)$ that connects a leaf and the adjacent vertex to be incoming if the number of directed edges in $\Gamma$ from the corresponding connected component of $\Gamma\setminus \bar{T}_n(\Gamma)$ to $\bar{T}_n(\Gamma)$ is odd and to be outgoing otherwise. The orientation of the remaining edges is defined from the assumption that each internal vertex has odd number of outgoing edges.}
Similar to the previous situations we have an isomorphism for the associated graded to this filtration given by~\eqref{eq::filt::extv} with $S=[1n-1]$. So it remains to describe the homology of $\calF^1\LLICG{d,[1n-1]}(n)$ including the case $n=1$. As always we consider the grading on the latter subspace into the sum of graphs that has exactly one edge connected with the last external vertex $\extv{n}$ and the remaining part:
$$\calF^1\LLICG{d,[1n-1]}(n) = \calF^1\LLICG{d,[1n-1]}(n)^{1}\oplus 
\calF^1\LLICG{d,[1n-1]}(n)^{\geq 2}.
$$
The part of the vertex splitting differential $d_0$ that decreases the grading have a simple combinatorial description:
$d_0$ takes a graph $\Gamma\in \calF^1\LLICG{d,[1n-1]}(n)^{\geq 2}$, create a new vertex $v$, move all edges that were connected to $\extv{n}$  and create an edge between $v$ and $\extv{n}$.
The differential $d_0$ admits an inverse given by contraction of a unique edge connected with $\extv{n}$ in a graph from $\LLICG{d,[1n-1]}(n)^{1}$. Thus the complex $\calF^1\LLICG{d,[1n-1]}(n)_{(l)}$ is acyclic except the two extremal cases $n=1 \& l=0$ and $n=2\& l=1$.
However, in these two cases the complex $\calF^1\LLICG{d,[1n-1]}(n)_{(l)}$ consists of one element. 
It follows that the cohomology of the associated graded complex $\gr\calF\LLICG{d,[1n-1]}(n)_{(l)}$ is zero except for two particular cases mentioned above. 
\end{proof}

We outline the main corollary of the described above combinatorics of graphs in the following
\begin{theorem}
\label{thm::todd::t::emb}	
	The map $f:\ICGod{d}(n))\to \ICG_{d}(n)$ induces an embedding on the level of homology: 
 $$f^{*}=\xi_n:\todd_{d-1}(n) \hookrightarrow \tkd_{d}(n).$$
 Recall that $\xi_n(\nu_{ijk})=[t_{ij},t_{jk}]$.
\end{theorem}
\begin{proof}
Lemma~\ref{lem::f=xi} explains that $f^*=\xi_n$ as claimed in the Theorem. Diagram~\eqref{eq::t::last} explains that thanks to the induction in $n$ it is enough to show the embedding of the Lie ideals $\xi_n|_{\ker\pi_n}:\ker\pi_n\hookrightarrow \Lie(t_{1n},\ldots,t_{n-1 n})$.
Finally, let us substitute $S=\emptyset$ in Lemma~\ref{lem::ICGS::embedding} and we end up with the desired embedding of Lie algebras:
\begin{multline*}
\ker\pi_n = H(\LICGod{d+1}(n))= H(\LICGod{d+1,\emptyset}(n)) \stackrel{\text{Lemma~\ref{lem::ICGS::embedding}}}{\hookrightarrow} H(\LLICG{d+1,\emptyset}(n)) = 
\\
= H(\LLICG{d+1}(n)) \stackrel{\text{Lemma~\ref{lm::forget::orient::ICG}}}{=} H(\LICG{d+1}(n))
\end{multline*}
that finishes the proof.
\end{proof}

We want to end up this section showing a couple of pictorial examples that helps to visualise the direction forgetting map $f$ and the projection on to homology classes $f^*$:
\begin{equation*}
	\begin{array}{c}
		{
			\begin{tikzpicture}[scale=0.5]
				\node[int] (v) at (-1,1.2) {};
				\node[int] (w) at (1,1.2) {};
				\node[int] (w1) at (0,1) {};
				\node[ext] (v1) at (-2,0) {i};
				\node[ext] (v2) at (-1,0) {j};
				\node[ext] (v3) at (0,0) {\small{n}};
				\node[ext] (v4) at (1,0) {\small{k}};
				\draw[-triangle 60] (v) edge (v1)  edge (v2) edge (w1);
				\draw[-triangle 60] (w) edge (w1)  edge (v4) edge[bend right=90] (v1);
				\draw[-triangle 60] (w1) edge (v3);
			\end{tikzpicture}
			\quad
			\stackrel{f}{\mapsto}
			\quad
			\begin{tikzpicture}[scale= 0.5]
				\node[ext] (v) at (-2,2) {\small{j}};
				\node[ext] (v1) at (0,2) {i};
				\node[ext] (v2) at (2,2) {\small{k}};
				\node[int] (w3) at (-1,1) {};
				\node[int] (w4) at (1,1) {};
				\node[int] (w5) at (0,0) {};
				\node[ext] (v0) at (0,-1) {\small{n}};
				\draw (v) edge (w3);
				\draw (v1) edge (w3);
				\draw (v1) edge (w4);
				\draw (v2) edge (w4);
				\draw (w3) edge (w5);
				\draw (w4) edge (w5);
				\draw (w5) edge (v0);
			\end{tikzpicture}
			\quad
			\stackrel{f^*}{\mapsto}
			\quad
			[[t_{jn},t_{in}],[t_{in},t_{kn}]],
		}
		\\
		{
			\begin{tikzpicture}[scale=0.5]
				\node[int] (v) at (-1,1.2) {};
				\node[int] (w) at (1,1.2) {};
				\node[int] (w1) at (0,1) {};
				\node[ext] (v1) at (-2,0) {i};
				\node[ext] (v2) at (-1,0) {j};
				\node[ext] (v3) at (0,0) {\small{k}};
				\node[ext] (v4) at (1,0) {\small{n}};
				\draw[-triangle 60] (v) edge (v1)  edge (v2) edge (w1);
				\draw[-triangle 60] (w) edge (w1)  edge (v4) edge[bend right=90] (v1);
				\draw[-triangle 60] (w1) edge (v3);
			\end{tikzpicture}
			\quad
			\stackrel{f}{\mapsto}
			\quad
			\begin{tikzpicture}[scale= 0.5]
				\node[ext] (vi) at (-2,2) {{i}};
				\node[ext] (vk) at (0,1.5) {\small{k}};
				\node[ext] (vj) at (3,2) {\small{j}};
				\node[int] (w1) at (2,1.5) {};
				\node[int] (w2) at (1,1) {};
				\node[int] (w3) at (0,0.5) {};
				\node[ext] (v0) at (0,-0.5) {\small{n}};
				\draw (vi) edge (w3);
				\draw (vi) edge[bend left=45] (w1);
				\draw (vj) edge (w1);
				\draw (vk) edge (w2);
				\draw (w1) edge (w2);
				\draw (w2) edge (w3);
				\draw (w3) edge (v0);
			\end{tikzpicture}
			\quad
			\stackrel{f^*}{\mapsto}
			\quad
			[t_{in},[t_{kn},[t_{in},t_{jn}]].
		}
	\end{array}
\end{equation*}
The main strategy looks as follows:
$(i)$ project onto internal trivalent trees with only one edge connecting to the
external vertex $\extv{n}$, $(ii)$ forget the directions of arrows, $(iii)$ interpret the tree as a Lie tree in $t_{jn}$: the vertex $\extv{n}$ is considered to be a root, a leaf connected to the vertex $\extv{j}$ corresponds to a copy of a $t_{jn}$ and each inner vertex corresponds to a commutator.

\subsection{Rational homotopy type of $\MonR{n+1}$ for individual $n$}
\label{sec::Rat_Hom_MonR}
Following~\cite{Bousfeld_Kan} we say that the topological space $X$ is a rational $K(\pi,1)$-space if its $\QQ$-completion is a $K(\pi,1)$-space. 
Note that a $K(\pi,1)$-property does not imply the rational $K(\pi,1)$ property. For example, the complement to the arrangement of the root system $D_4$ is a $K(\pi,1)$ space but its $\QQ$-completion is not.

\begin{corollary}
\label{cor::Mon::Kp1}	
	The real locus of the moduli space $\MonR{n+1}$ is a rational $K(\pi,1)$-space.
\end{corollary}
\begin{proof}
It was proved in~\cite{Papadima} (Proposition 5.2) that if the rational cohomology of a connected topological space $X$ is a finite-dimensional Koszul algebra generated by the first homology $H^{1}(X;\QQ)$ then $X$ is a rational $K(\pi,1)$-space. So thanks to Corollary~\ref{cor::koszul::Twopois} we are done. 
\end{proof}

However, in order to get a bit more feeling of the Lie algebra $\LCacti_n$ which is the Lie algebra assigned to the prounipotent completion $\widehat{\PCacti}$ of the pure cactus group we want to give an independent proof of Corollary~\ref{cor::Mon::Kp1} adopted to our case.

\begin{proof}[Second proof of Corollary~\ref{cor::Mon::Kp1}]
Recall that the dgca-model of the space of cochains on $\MonR{n+1}$ for individual $n$ discussed in Theorem~\ref{thm::Mos::Model} is given by the space of odd oriented graphs with $n$ external vertices and the vertex-splitting differential twisted by $\gamma^{\orient}$.
This model has a filtration by the loop order of a graph and the associated graded complex coincides with the nontwisted complex.
Consider the corresponding filtration of the Harrison complex of the dgca of twisted graphs. The spectral sequence argument predicts that the cohomology of this Harrison complex differs from zero only in $0$ cohomological degree and coincides with the quadratic Koszul Lie algebra $\todd_2(n)$. Hence there is no space for higher Lie brackets and the homotopy Lie algebra of the $\QQ$-completion of $\MonR{n+1}$ is an ordinary Lie algebra denoted by $\LCacti_n$ with all higher homotopy groups vanishing.
Moreover, we know that there is a filtration on $\LCacti_n$ such that the associated graded is isomorphic to $\todd_2(n)$ that remembers the loop order of a graph.
\end{proof}

\begin{corollary}
\label{cor::Lcacti::tkd}	
The map $\tilde{\xi}_n^{\QQ}: \LCacti_n^{\QQ} \rightarrow \tkd_2(n)^{\QQ}$
is an embedding.	
\end{corollary}
\begin{proof}
	Thanks to Theorem~\ref{thm::todd::t::emb} we know that the associated graded map
	 ${\xi}_n^{\QQ}: \todd_1(n)^{\QQ} \rightarrow \tkd_2(n)^{\QQ}$ is an embedding.
\end{proof}

It was checked by computer and announced in Section~3.9 of~\cite{Etingof_Rains} that the spaces $\MonR{n+1}$ are not formal for $n\geq 5$. The dgca models that we suggested lead to the following equivalent statement.
\begin{corollary}
	The dgca's $(\Graphs_{2}^{\orient,\odd}(n),d)$ and $(\Graphs_{2}^{\orient,\odd}(n),d+\gamma^{\orient}\cdot \ttt)$ are not weakly equivalent. Equivalently, the Koszul-dual Lie algebra $\LCacti_n$ is not isomorphic to its associated graded $\todd_1(n)$. 
\end{corollary}
Unfortunately, we were not able to find a simple nontrivial Massey product.
However, one can try to consider the filtration on the twisted complex of odd graphs given by the mod $4$ internal loop order. The corresponding  associated graded differential increases the loop order by at most $2$ (since it should be an even number) is presented by the action of the leading terms~\eqref{eq::SHoikhet} of the Shoikhet cocycle $\gamma^{\orient}$ and what is important the associated graded differential does not break the connectivity of a graph, thus one can work directly with $\ICGod{2}(n)$ with a twisted differential.
We do not present a nasty computation with graphs but we state that the cocycles of the twisted complex of odd graphs representing the generators 
\(\nu_{ijk}\in \LCacti_n\)
contain internally trivalent trees of loop order $4$ that illustrates (but does not give a self contained proof of) the fact that the Lie algebra $\LCacti_n$ is not isomorphic to its associated graded.

\section{Around deformations of $\TwoPois_d$ and the mosaic operad}
\label{sec::Deformations}

The deformation theory of the little $n$-disks operads $E_n$, and that of the natural maps $E_m\to E_n$ is well studied by now, at least rationally.
Given our models for the odd Poisson and the mosaic operads, we may transcribe a list of results in the literature to extend to those operads.
In this section we shall outline a few statements that are obtainable in this manner.
We shall however defer a detailed discussion, and complete proofs to elsewhere, in order to avoid lengthy technical recollections.

We recall that the homotopy deformations and automorphisms of the operads (or Hopf operads) $\Pois_d$ are controlled by the graph complexes $\GC_d$, see \cite{Willwacher_grt} or \cite{Fresse_Turchin_Willwacher} (in the Hopf setting).
Furthermore, the deformation theory of the natural operad maps $E_m\to E_n$ has similarly been studied. In short, the results are that the deformation theory of the operad maps $E_{n-1}\to E_n$ can be identified with that of the operad $E_n$, and that the deformations of $E_{m}\to E_n$ (for $n-m\geq 2$) are governed by hairy graph complexes \cite{Turchin_Wilwacher_Hairy, Fresse_Turchin_Willwacher}.

Our realization of the odd Poisson and mosaic operads as the cobar-duals to the $\ZZ_2$-invariants of the Poisson and associative operads, and our graphical models $\Graphs_{d+1}^{\orient,\odd}$ allow us to extend most of the aforementioned results for the little disks operads to the odd Poisson and mosaic operads.

First, recall from Section~\ref{sec::Graphs::MonR} that our graphical model $\Graphs_{d+1}^{\orient,\odd}$ of the odd Poisson operad $\Pois_d^{\odd}$ carries an action of the even-loop-order part of the oriented graph complex $\GC_{d+1}^{\orient,\odd}$. Hence we obtain a map of complexes 
\begin{equation}
\label{equ:GCtoDefodd}
\QQ L \ltimes \GC_{d+1}^{\orient,\odd} \to 
\Def(\Omega([\Pois_d]^{\dual}_{\ZZ_2})\{1-d\}
\to \Graphs_{d+1}^{\orient,\odd})[1]
\end{equation}
into the the complex of operadic derivations of the natural quasi-isomorphism 
\[
\Omega([\Pois_d]^{\dual}_{\ZZ_2})\{1-d\}
\to
\ho\Pois_d^{\odd}
\to \Graphs_{d+1}^{\orient,\odd}.
\]
The element $L$ on the left-hand side of \eqref{equ:GCtoDefodd} is the generator of the grading by loop order. The right-hand side of \eqref{equ:GCtoDefodd} can also be identified, up to a degree shift, with the homotopy derivations of the odd Poisson operad. Now the cohomology of that right-hand side can be computed along the lines of \cite[Proposition 4]{Willwacher_oriented}.
Up to degree shifts, it is a symmetric product of its connected part $\Def(...)_{conn}$, and one can show that the map \eqref{equ:GCtoDefodd} is a quasi-isomorphism onto that connected part.
Noting the (loop order preserving) isomorphism \eqref{equ:GCGCor}, we can hence conclude that the deformation theory of the odd Poisson operads is controlled by the even-loop-order parts of the graph cohomology $H(\GC_{d})$.

Next, our model $(\Graphs_{2}^{\orient,\odd})^{\gamma^{\orient}}$ of the mosaic operad of Section~\ref{sec::Graphs::MonR} is obtained by twisting the above model $\Graphs_{2}^{\orient,\odd}$ of the odd Poisson operad with the Shoikhet's Maurer-Cartan element $\gamma^{\orient}\in \GC_2^{\orient,\odd}$.
One can use this to show that the deformation theory of the mosaic operad is controlled by the twisted graph complex $(\QQ L \ltimes \GC_2^{\orient,\odd})^{\gamma^{\orient}}$. However, this complex is acyclic, as one can derive from the acyclicity result of \cite{Extradif}.
In particular we can conclude that the deformation complex for the mosaic operad is acyclic
and hence the mosaic operad is rigid.

Finally, one may study the space of maps from the cochains of the mosaic operad to the chains operad of the little 2-disks operad. To this end, one can essentially re-use, with only small adaptations, the computation of the deformations of the operad maps $E_1\to E_2$ from \cite{Turchin_Wilwacher_Hairy}. 
The end result is the following. Any operad map from the chains operad of the mosaic operad into the chains operad of the little 2-disks, which agrees up to arity three with the standard map, is obtained from the standard map by an action of the homotopy automorphisms of $E_2$, i.e., by the Grothendieck-Teichm\"uller group.

In particular, note that any functorial construction of a coboundary category structure from a braided monoidal structure as in Section~\ref{sec::Drinfeld::unit} can be seen as an implicit description of a map from a model of the mosaic operad to a model of (chains of the) the little disks operad. Hence, such a construction is unique, up to the Grothendieck-Teichm\"uller group action, in the appropriate sense.

Now we know from the literature~\cite{HK_crystal},\cite{Kamn_Uq},\cite{Kamnitzer_Kash},\cite{Rybnikov} that there are many different definitions of a coboundary category on the set of crystals of a given quantum group $U_q(\mathfrak{g})$. We hope that one can use the deformation theory of the Mosaic operad in order to show that there exists a unique structure of coboundary category on the space of crystals of $U_q(\mathfrak{g})$.

\end{document}